\documentclass[10pt]{amsart}
\usepackage[margin=1.5in]{geometry}

\usepackage{amsmath, amsthm, amssymb} 

\usepackage[dvipsnames]{xcolor}
\usepackage[colorlinks, allcolors=purple!90!black]{hyperref} 
\hypersetup{
    colorlinks = true,
    linkcolor = purple!90!black,
    citecolor = Blue,
    urlcolor = Blue,
}
\usepackage{tikz-cd} 
\usepackage[nameinlink,capitalize]{cleveref} 
\usepackage{mathtools}
\usepackage[shortlabels]{enumitem} 
\usepackage{microtype} 
\usepackage{mathrsfs}
\usepackage{lscape}
\usepackage{float}
\usepackage{multicol}
\usepackage{euscript}
\newcommand{\euscr}[1]{\EuScript{#1}}
\usepackage[
    backend=biber,
    style=alphabetic,
    maxbibnames=99,
    isbn=false,
    url=false,
    doi=false,
    maxalphanames=9,
    minalphanames=4,
]{biblatex}
\newcommand{\C}{\mathbb{C}} 
\newcommand{\R}{\mathbb{R}}
\newcommand{\Z}{\mathbb{Z}} 

\DeclareMathOperator{\modmod}{/\!/}

\DeclareMathSymbol{\mhyphen}{\mathord}{AMSa}{"39}

\numberwithin{equation}{section} 

\newtheorem{letterthm}{Theorem}

\newtheorem{thm}{Theorem}[section]
\AddToHook{env/thm/begin}{\crefalias{section}{thm}}

\newtheorem{corollary}[thm]{Corollary}
\AddToHook{env/corollary/begin}{\crefalias{section}{corollary}}

\newtheorem{lemma}[thm]{Lemma}
\AddToHook{env/lemma/begin}{\crefalias{section}{lemma}}

\newtheorem{proposition}[thm]{Proposition}
\AddToHook{env/proposition/begin}{\crefalias{section}{proposition}}

\newtheorem*{thm*}{Theorem}
\newtheorem*{cor*}{Corollary}
\newtheorem*{lem*}{Lemma}
\newtheorem*{prop*}{Proposition}

\theoremstyle{definition}
\newtheorem{definition}[thm]{Definition}
\AddToHook{env/defintion/begin}{\crefalias{thm}{definition}}

\AddToHook{env/notation/begin}{\crefalias{thm}{notation}}

\newtheorem{example}[thm]{Example}
\AddToHook{env/example/begin}{\crefalias{thm}{example}}

\newtheorem{remark}[thm]{Remark}
\AddToHook{env/remark/begin}{\crefalias{thm}{remark}}

\numberwithin{thm}{section}
\numberwithin{equation}{section}
\numberwithin{figure}{section}

\usepackage[textwidth=25mm, textsize=tiny]{todonotes}
\usepackage[dvipsnames]{xcolor}
\usepackage{tcolorbox}
\usepackage{color}
\definecolor{orange'}{HTML}{D55E00}

\begin{document}

\title{On the ring of cooperations for real hermitian K-theory}
\author{Jackson Morris}
\address{Department of Mathematics, University of Washington, Seattle, Washington}
\email{\href{mailto:jackmann@uw.edu}{jackmann@uw.edu}}


\begin{abstract}
	Let $\textup{kq}$ denote the very effective cover of the motivic hermitian K-theory spectrum. We analyze the ring of cooperations $\pi^\mathbb{R}_{**}(\textup{kq} \otimes \textup{kq})$ in the stable motivic homotopy category $\text{SH}(\mathbb{R})$, giving a full description in terms of Brown--Gitler comodules. To do this, we decompose the $\textup{E}_2$-page of the motivic Adams spectral sequence and show that it must collapse. The description of the $\textup{E}_2$-page is accomplished by a series of algebraic Atiyah--Hirzebruch spectral sequences which converge to the summands of the $\textup{E}_2$-page. Along the way, we prove a splitting result for the very effective symplectic K-theory ksp over any base field of characteristic not two.
\end{abstract}

\maketitle

\begin{center}
    \includegraphics[scale=.4]{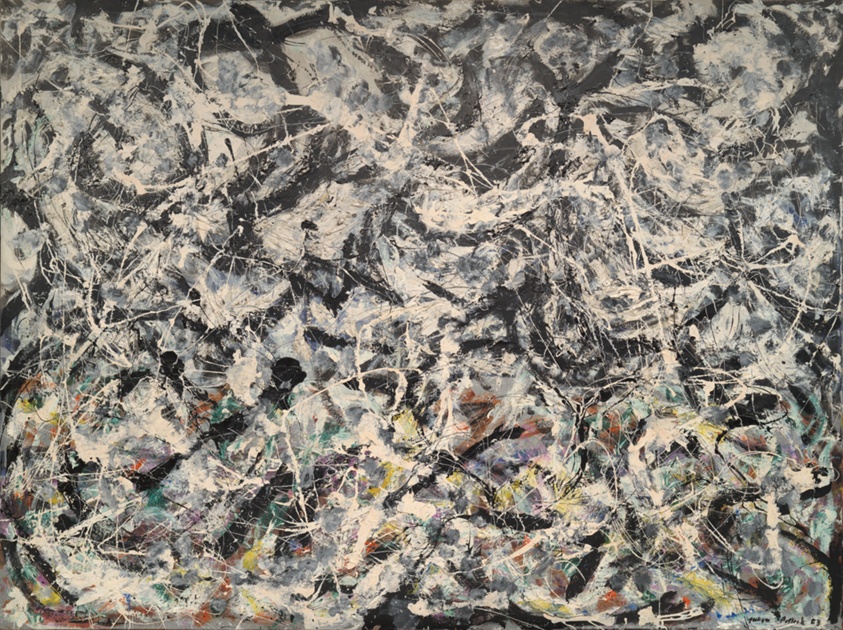}\\
    
    \emph{Greyed Rainbow}, Jackson Pollock (1953)
\end{center}
\vspace*{\fill}

\newpage
\tableofcontents

\section{Introduction}
\subsection{Motivation}
In stable homotopy theory, one of the most powerful tools we have for computing the homotopy groups of spheres is the Adams spectral sequence. For a ring spectrum E, the E-Adams spectral sequence (E-$\textbf{ASS}(\mathbb{S})$) has signature
\[\textup{E}_1^{s, f} = \pi_{s+f}(\text{E} \otimes \overline{\text{E}}^{\otimes f})\implies \pi_{s}(\mathbb{S}_\text{E}^\wedge),\]
where $\overline{\text{E}}$ is the cofiber of the unit map $\mathbb{S} \to \text{E}$. Under mild assumptions on E, this spectral sequence strongly converges to the homotopy groups of $\mathbb{S}_\text{E}^\wedge$, the E-nilpotent completion of the sphere spectrum \cite{Rav86}. To begin to compute with the E-$\textbf{ASS}(\mathbb{S})$, one must understand the smash powers $\text{E}^{\otimes f}$. 

One may specialize from desiring to understand the entirety of $\pi_*(\mathbb{S})$ to desiring to understand a particular piece of $\pi_*(\mathbb{S})$. This is one of the payoffs of chromatic homotopy theory, which organizes the elements of the stable homotopy groups of spheres into $v_n$-periodic layers. The $v_0$-periodic elements of $\pi_*(\mathbb{S})$ are those which survive rationalization, hence the elements in degree 0. The next interesting layer of elements in the stable stems are those which are $v_1$-periodic, and so one may attempt to determine the $v_1$-torsion free component of $\pi_*(\mathbb{S})$ by some $\text{E-}\textbf{ASS}(\mathbb{S})$.

Let bo denote the connective cover of the real topological K-theory spectrum KO. This spectrum is well-enough behaved that one may begin to study its associated Adams spectral sequence. Mahowald and others have studied the bo-$\textbf{ASS}(\mathbb{S})$, known as the bo-resolution, at the prime 2 extensively \cite{Mah81, LM87, BBBCX}. Analysis of this spectral sequence has lead to, among many things, the $v_1$-torsion free component of $\pi_*(\mathbb{S})$. Key to Mahowald's analysis is a spectrum level splitting of $\text{bo} \otimes \text{bo}$ as the following wedge sum:
\[\text{bo} \otimes \text{bo} \simeq \bigoplus_{k \geq 0}\Sigma^{4k}\text{bo} \otimes \textup{H}\mathbb{Z}_k^{cl}.\]
The spectra $\textup{H}\mathbb{Z}_k^{cl}$ are finite complexes known as the integral Brown--Gitler spectra \cite{BG73, GJM86}.

In stable motivic homotopy theory, a primary object of investigation is the homotopy ring of the motivic sphere spectrum. This ring has rich connections with the Grothendieck-Witt ring of symmetric bilinear forms \cite{Mor12}, Milnor K-theory \cite{bachmannburklundxu}, and hermitian K-theory \cite{RSOfirst, RSOsecond}. Similar to the classical case, one of the most powerful and well-studied tools is the motivic Adams spectral sequence. Working in $\text{SH}(F)$ for a field $F$, for a motivic ring spectrum E, there is an E-motivic Adams spectral sequence (E-\textbf{mASS}$^F(\mathbb{S})$) which can be used to compute these homotopy groups that, under mild conditions, converges strongly to the homotopy groups of the E-nilpotent completion $\mathbb{S}_\text{E}^\wedge$ \cite{DImASS, HKO-convergencemASS}. This takes the form
\[\textup{E}_1^{s, f,w} = \pi^F_{s+f,w}(\textup{E} \otimes \overline{\textup{E}}^{\otimes f})\implies \pi_{s,w}^F(\mathbb{S}_\textup{E}^\wedge).\]
To begin to compute with the E-\textbf{mASS}$^F(\mathbb{S})$, one must  understand the smash powers $\textup{E}^{\otimes f}$. Similar to the classical case, one may organize the motivic stable homotopy groups of spheres into $v_n$-periodic layers and look for a particular $\text{E-}\textbf{mASS}^F(\mathbb{S})$ which is useful for detecting $v_n$-periodicity.

Let $\textup{kq}$ denote the very effective cover of the hermitian K-theory spectrum KQ \cite{ARO20} (see \cref{section ksp} for details on the very effective slice filtration). In \cite{CQ21}, Culver--Quigley analyze the $\textup{kq}$-\textbf{mASS}$^\mathbb{C}(\mathbb{S})$ at the prime 2 in $\C$-motivic homotopy theory. This spectral sequence is called the $\textup{kq}$-resolution and has signature
\[\textup{E}_1^{s, f,w} = \pi^\mathbb{C}_{s+f,w}(\textup{kq} \otimes \overline{\textup{kq}}^{\otimes f})\implies \pi_{s,w}^\mathbb{C}(\mathbb{S}_{\textup{kq}}^\wedge).\]Their analysis of this spectral sequence leads to, among many things, the $v_1$-torsion free component of $\pi^\mathbb{C}_{**}(\mathbb{S})$. Toward computing the $\textup{E}_1$-page, Culver--Quigley compute the ring of cooperations $\pi^\mathbb{C}_{**}(\textup{kq} \otimes \textup{kq})$ by the $\textup{H}\mathbb{F}_2\text{-}\textbf{mASS}^\mathbb{C}(\text{kq} \otimes \text{kq})$ which takes the form
\[\textup{E}_2^{s,f,w} = \text{Ext}^{s,f,w}_{\euscr{A}_\mathbb{C}^\vee}\left(\mathbb{M}_2^\mathbb{C}, \textup{H}_{**}(\textup{kq} \otimes \textup{kq})\right) \implies \pi_{s,w}^\mathbb{C}(\textup{kq} \otimes \textup{kq}).\]
Here, $\euscr{A}^\vee_\mathbb{C}$ denotes the $\mathbb{C}$-motivic dual Steenrod algebra, $\text{H}_{**}(-)$ denotes mod-2 motivic homology, and $\mathbb{M}_2^\mathbb{C}$ denotes the mod-2 motivic homology of a point. We implicitly 2-complete the target of this spectral sequence, as we will throughout this paper, to ensure convergence.

\subsection{Main results}
The goal of this paper is to begin the study of the $\textup{kq}$-resolution in $\mathbb{R}$-motivic homotopy theory. As a step in this direction, we compute the ring of cooperations $\pi_{**}^\mathbb{R}(\textup{kq} \otimes \textup{kq})$ as a module over $\pi_{**}^\mathbb{R}(\text{kq}).$ Throughout the rest of this paper, let $\textbf{mASS}^F(X)$ denote the $F$-motivic $\text{H}\mathbb{F}_2$-based motivic Adams spectral sequence, where $\text{H}\mathbb{F}_2$ is the motivic Eilenberg-MacLane spectrum representing mod-2 motivic cohomology.

\begin{letterthm}[\cref{e2 mass}, \cref{main}]
\label{thmA}
    The $\textup{\textbf{mASS}}^{\mathbb{R}}(\textup{kq} \otimes \textup{kq})$ has signature
    \[\textup{E}_2^{s,f,w} = \bigoplus_{k \geq 0}\Sigma^{4k, 2k}\textup{Ext}_{\euscr{A}(1)^\vee}^{s,f,w}(\mathbb{M}_2^{\mathbb{R}}, B_0^{\mathbb{R}}(k)) \implies \pi_{s,w}^{\mathbb{R}}(\textup{kq} \otimes \textup{kq})\]
    and collapses on the $\textup{E}_2$-page.
    Here $B_0^{\mathbb{R}}(k)$ denotes the $k^{th}$ integral motivic Brown--Gitler comodule. We describe the $\textup{E}_2=\textup{E}_\infty$-page, modulo $v_1$-torsion, as a module over $\pi_{**}^{\mathbb{R}}\textup{(kq)}.$
\end{letterthm}

The reader familiar with  Mahowald's analysis of the bo-resolution \cite{Mah81} may be inclined to believe that this result may be proven in a similar fashion. This is not the case. Mahowald's arguments make great use of integral Brown--Gitler spectra. However, there is currently no known construction of motivic integral Brown--Gitler spectra except in very particular cases (see \cref{HZ1 construction}, \cref{motivic BG spectra}), and so another method of analysis must be employed.
As is the case over $\mathbb{C}$, we generally follow the strategy employed by \cite{BOSS19}, in which they compute the classical ring of cooperations for $\textup{tmf}$ using algebraic methods. We outline this strategy below.
\begin{itemize}
    \item[(1)] First, we introduce the integral motivic Brown--Gitler comodules $B_0^\mathbb{R}(k)$ and show that there are isomorphisms of $\euscr{A}(1)^\vee_\mathbb{R}$-comodules:
    \[\textup{H}_{**}(\textup{kq}) \cong(\euscr{A} \modmod \euscr{A}(1))_\mathbb{R}^\vee \cong \bigoplus_{k \geq 0} \Sigma^{4k, 2k}B_0^\mathbb{R}(k).\]
    \item[(2)] Next, we produce short exact sequences relating the integral Brown--Gitler comodules which allow for induction.
    \item[(3)] Last, we make base-case computations in Ext using algebraic Atiyah--Hirzebruch spectral sequences (\textbf{aAHSS}) and induct to complete the proof of the theorem.
\end{itemize}
Steps (1) and (2) are direct consequences of \cite[Section 3]{CQ21}. Step (3) is the most difficult and occupies most of this paper. This is because the algebra $\text{Ext}^{***}_{\euscr{A}(1)^\vee_\mathbb{R}}(\mathbb{M}_2^\mathbb{R}, \mathbb{M}_2^\mathbb{R})$ has more elaborate structure than its classical and $\mathbb{C}$-motivic counterparts. Indeed, when presenting the charts for our spectral sequence computations, we find it best to follow the strategy of Guillou--Hill--Isaksen--Ravenel \cite{GHIRkoc2} and organize each page by coweight $cw = stem - weight$. There is an  element $\tau^4 \in \text{Ext}^{0, 0, -4}_{\mathcal{A}(1)^\vee_\mathbb{R}}(\mathbb{M}_2^\mathbb{R}, \mathbb{M}_2^\mathbb{R})$ which gives a periodicity allowing us to organize our data into 4 sets of charts. These computation are manageable with careful bookkeeping.

As an application, we compute the $n$-line of the $\text{E}_1$-page of the $\text{kq}$-resolution over $\mathbb{R}$.

\begin{letterthm}[\Cref{prop:n-lineE2}, \Cref{thm:n-lineDifs}]
\label{thm:B}
    The $\textup{\textbf{mASS}}^{\mathbb{R}}(\textup{kq} \otimes \overline{\textup{kq}}^{\otimes n})$ has signature
    \[\textup{E}_2^{s,f,w} = \bigoplus_{K \in \euscr{K}_n} \Sigma^{4|K|, 2|K|}\textup{Ext}_{\euscr{A}(1)^\vee_\mathbb{R}}^{s,f,w}(\mathbb{M}_2^{\mathbb{R}}, B_0^{\mathbb{R}}(K)) \implies \pi_{s, w}^{\mathbb{R}}(\textup{kq} \otimes \overline{\textup{kq}}^{\otimes n}),\]
    and collapses on the $\textup{E}_2$-page, where $\euscr{K}_n = \{K = (k_1, \dots, k_n): k_j \geq 1 \textup{ for all } j\}$, $|K| = \sum_{j=1}^nk_j$, and $B_0^{\mathbb{R}}(K) = \bigotimes_{j=1}^nB_0^{\mathbb{R}}(k_j).$
\end{letterthm}

Our method of computation describes the $\text{E}_1$-page of the kq-resolution as a module over $\pi_{**}^\mathbb{R}(\text{kq})$ (reviewed in \Cref{R-background and notation}). We imagine that this structure will be useful in working with the kq-resolution in the style of \cite{BBBCX}.

Through our inductive procedure, we naturally obtain the following result, where ksp denotes the very effective cover of $\Sigma^{4,2}\textup{KQ}$ (see also \cref{section ksp}).
\begin{letterthm}[\cref{R-ksp}]
\label{thm-ksp b0(1)}
    The $\textup{\textbf{mASS}}^{\mathbb{R}}(\textup{ksp})$ has signature
    \[\textup{E}_2^{s,f,w} = \textup{Ext}^{s,f,w}_{\euscr{A}(1)^\vee_\mathbb{R}}(\mathbb{M}_2^\mathbb{R}, B_0^\mathbb{R}(1)) \implies \pi_{s,w}^\mathbb{R}(\textup{ksp}),\]
    and collapses on the $\textup{E}_2$-page.
\end{letterthm}
This is an easy consequence of a splitting result which we show in a more general setting. 
\begin{letterthm}[\cref{ksp hz1}]
\label{thm-hz1 ksp} 
    Let $F$ be any field with 2 invertible. Then there is an equivalence of motivic spectra
    \[\textup{ksp} \simeq \textup{H}\mathbb{Z}_1^F \otimes \textup{kq},\]
    where $\textup{H}\mathbb{Z}_1^F$ denotes the first integral $F$-motivic Brown--Gitler spectrum \textup{(see \cref{HZ1 construction})}.
\end{letterthm}
We prove this result by analyzing the very effective slice tower for KQ similar to how one would analyze the Whitehead tower for the classical spectrum KO.

\subsection{Future directions}
We next outline several directions in which we will extend the results of this work.

\subsubsection*{The \textup{kq}-resolution}
\hfill

This paper initiates the study of the $\textup{kq}$-resolution in $\text{SH}(\mathbb{R})$. In particular, we calculate the $\textup{E}_1$-page. While many of the differentials in this spectral sequence can be determined using the base change functor
\[-\otimes \mathbb{C}: \text{SH}(\mathbb{R}) \to \text{SH}(\mathbb{C}),\]
there is $\rho$-periodic data which must be determined using the $C_2$-equivariant Betti realization functor
\[\text{Be}^{C_2}:\text{SH}(\mathbb{R}) \to \text{Sp}^{C_2}\]
along with information involving the $\text{ko}_{C_2}$-resolution. We plan to study this in future work. Additionally, the $\textup{kq}$-resolution is intricately related to $v_1$-periodic homotopy. The work of Belmont--Isaksen--Kong \cite{belmontisaksenkong-v1R} determines a version of $v_1$-periodic homotopy theory in $\pi_{**}^\mathbb{R}\mathbb{S}$ by computing the homotopy groups $\pi_{**}^\mathbb{R}(\textup{L})$, where L sits in a cofiber sequence
\[\textup{L} \to \textup{kq} \xrightarrow{\psi^3-1}\textup{kq}.\]
It will be interesting to compare the homotopy groups $\pi_{**}^\mathbb{R}(\textup{L})$ with the $v_1$-periodic homotopy detected by the $\textup{kq}$-resolution (see also \cref{motivic j}).

\subsubsection*{Connections with $C_2$-equivariant homotopy theory}
\hfill

In $C_2$-equivariant homotopy theory, Li--Petersen--Tatum have produced $C_2$-equivariant analogues of the classical integral Brown--Gitler spectra \cite{LiPetTat25}. These are finite spectra $\textup{H}\mathbb{Z}_k^{C_2}$ such that, as comodules over the the sub-Hopf algebra $\euscr{A}(1)^\vee_{C_2}$ of the $C_2$-equivariant dual Steenrod algebra, there is an isomorphism
\[\textup{H}_{\star}(\textup{H}\mathbb{Z}_k^{C_2}) \cong B_0^{C_2}(k),\]
where $B_0^{C_2}(k)$ is the $k^{th}$ $C_2$-equivariant integral Brown--Gitler comodule. It was recently shown \cite{LPT-kuRsplitting} that there is a splitting of $\textup{ku}_\mathbb{R}$-module spectra
\[\textup{ku}_\mathbb{R} \otimes \textup{ku}_\mathbb{R} \simeq \bigoplus_{k \geq 0}\Sigma^{\rho k}\textup{ku}_\mathbb{R} \otimes \textup{H}\mathbb{Z}_k^{C_2}.\]
One can ask for a similar splitting of the $C_2$-spectrum $\textup{ko}_{C_2}$. In particular, it is conjectured that there is a splitting of $\textup{ko}_{C_2}$-module spectra
\[\textup{ko}_{C_2} \otimes \textup{ko}_{C_2} \simeq \bigoplus_{k \geq 0}\Sigma^{2 \rho k}\textup{ko}_{C_2} \otimes \textup{H}\mathbb{Z}_k^{C_2}.\]
One way to show that such a splitting exists is to construct maps at the level of homology
\[\theta_k:\Sigma^{2 \rho k}\textup{H}_\star (\textup{H}\mathbb{Z}_k^{C_2}) \to \textup{H}_\star (\textup{ko}_{C_2})\]
and show that they survive the $\textup{ko}_{C_2}$-relative Adams spectral sequence
\[\textup{E}_2 = \text{Ext}^{V, f}_{\euscr{A}(1)^\vee_{C_2}}\left(\textup{H}_\star(\Sigma^{2\rho k} \textup{H}\mathbb{Z}_k^{C_2}), \textup{H}_\star (\textup{ko}_{C_2})\right) \implies [\Sigma^{2\rho k}\textup{ko}_{C_2} \otimes \textup{H}\mathbb{Z}_k^{C_2}, \textup{ko}_{C_2} \otimes \textup{ko}_{C_2}]^{\textup{ko}_{C_2}}_{V}.\]
Due to the isomorphism \cite{GHIRkoc2}
\[\text{Ext}^{\star, *}_{\euscr{A}(1)^\vee_{C_2}}(\mathbb{M}_2^{\text{C}_2}, \mathbb{M}_2^{\text{C}_2}) \cong \text{Ext}_{\euscr{A}(1)^\vee_{\mathbb{R}}}^{***}(\mathbb{M}_2^{\mathbb{R}}, \mathbb{M}_2^{\mathbb{R}}) \oplus \text{Ext}^{***}_{\euscr{A}(1)^\vee_\mathbb{R}}(NC, \mathbb{M}_2^{\mathbb{R}}),\]
where $NC$ dentoes the ``negative cone", we expect that the Ext computations in this paper will be valuable in solving this problem. These ideas will be further investigated in future work, joint with Petersen and Tatum.

\subsubsection*{The ring of cooperations over general bases}
\hfill

One can also ask about the ring of cooperations over more general base schemes. A necessary input to using the methods of \Cref{inductive process section} to compute $\pi_{**}(\text{kq} \otimes \text{kq})$ is a computation of the algebra $\text{Ext}_{\euscr{A}(1)^\vee}^{***}(\mathbb{M}_2, \mathbb{M}_2)$. In \cite{Kylingkqfinite}, Kylling determined $\text{Ext}^{***}_{\euscr{A}(1)^\vee_{\mathbb{F}_q}}(\mathbb{M}_2^{\mathbb{F}_q}, \mathbb{M}_2^{\mathbb{F}_q})$ for $\text{char}(\mathbb{F}_q) \neq 2$, which allowed us to follow the program layed out in this paper and compute the ring of cooperations $\pi_{**}^{\mathbb{F}_q}(\textup{kq} \otimes \textup{kq})$ \cite{finitekqcoop}. In forthcoming work with Petersen and Tatum, we compute both $\text{Ext}_{\euscr{A}(1)_F^\vee}^{***}(\mathbb{M}_2^F, \mathbb{M}_2^F)$ and $\pi_{**}^{F}(\text{kq} \otimes \text{kq})$ when $F = \mathbb{Q}$ and $F=\mathbb{Q}_p$ for $p$ any prime \cite{MorPetTat}.

The case of finite fields is particularly interesting when paired with our computation and the existing $\mathbb{C}$-motivic computation. A remarkable result of Bachmann--\O stv\ae r \cite{BO22} gives a useful pullback square that allows us to make a deeper statement. Namely, if $\textup{E} \in \text{SH}(\mathbb{Z}[1/2])^{\text{cell}}$, then there is a pullback square
    \[\begin{tikzcd}
	{\text{map}_{\mathbb{Z}[1/2]}(\mathbb{S}, \textup{E})_2^\wedge} & {\text{map}_\mathbb{R}(\mathbb{S}, \textup{E})_2^\wedge} \\
	{\text{map}_{\mathbb{F}_3}(\mathbb{S}, \textup{E})_2^\wedge} & {\text{map}_\mathbb{C}(\mathbb{S}, \textup{E})_2^\wedge}
	\arrow[from=1-1, to=1-2]
	\arrow[from=1-1, to=2-1]
	\arrow["\lrcorner"{anchor=center, pos=0.125}, draw=none, from=1-1, to=2-2]
	\arrow[from=1-2, to=2-2]
	\arrow[from=2-1, to=2-2]
    \end{tikzcd}\]
where $\text{map}_F(\mathbb{S}, \textup{E})$ denotes the (ordinary) spectrum of maps between $\mathbb{S}$ and $\textup{E}$ in $\text{SH}(F)$. The computation of $\pi_{**}^F(\textup{kq} \otimes \textup{kq})$ for $F = \mathbb{C}, \mathbb{R}, \mathbb{F}_3$ combined with this square gives a long exact sequence which allows one to access the arithmetic invariant $\pi_{**}^{\mathbb{Z}[1/2]}(\textup{kq} \otimes \textup{kq})$. This bypasses the fact that, although we understand the dual motivic Steenrod algebra $\euscr{A}^\vee_{\mathbb{Z}[1/2]}$ (see \cite[Theorem 11.24]{Spitzweck-HZ-overdedekind}, also \cite[Section 4.3]{dundasostvaer-integralmotivicsteenrod}), there is currently no description of $\text{Ext}^{***}_{\euscr{A}(1)_{\mathbb{Z}[1/2]}^\vee}(\mathbb{M}_2^{\mathbb{Z}[1/2]}, \mathbb{M}_2^{\mathbb{Z}[1/2]})$.
\subsection{Organization}
\cref{part1} of this paper presents background and recalls the work of Culver--Quigley in \cite{CQ21}. In \cref{section 2}, we recall relevant facts about the stable motivic homotopy category $\text{SH}(F)$, the kq-resolution, and motivic Brown--Gitler comodules. In \cref{section ksp}, we prove a novel result about symplectic K-theory. In \cref{inductive process section}, we outline a process for computing the ring of cooperations $\pi_{**}^F(\text{kq} \otimes \text{kq})$. Much of the material of these sections is contained \cite[Section 3]{CQ21}, but we phrase things in slightly more generality, which will be useful to a broader audience. In \cref{section 3}, we review the computation of $\pi_{**}^\mathbb{C}(\text{kq} \otimes \text{kq})$ as a warm up for the main contents of our work.

\cref{part 2} of this paper contains the computation of the ring of cooperations $\pi_{**}^\mathbb{R}(\text{kq} \otimes \text{kq})$ and applications. In \cref{section 4}, we describe the algebra $\text{Ext}^{***}_{\euscr{A}(1)^\vee_\mathbb{R}}(\mathbb{M}_2^\mathbb{R},\mathbb{M}_2^\mathbb{R})$ and some modules which will frequently appear. In \cref{section R-aAHSS}, we compute $\text{Ext}^{***}_{\euscr{A}(1)^\vee_\mathbb{R}}\left(\mathbb{M}_2^\mathbb{R}, B_0^\mathbb{R}(1)^{\otimes i}\right)$ by a series of algebraic Atiyah--Hirzebruch spectral sequences. In \cref{section-Rcoop}, we compute the ring of cooperations $\pi_{**}^\mathbb{R}(\text{kq} \otimes \text{kq})$ and apply our results to describe the $\textup{E}_1$-page of the kq-resolution.

In \cref{charts}, we present a set of charts representing the modules described in \cref{section 4}.

\subsection{Notation and Conventions}
\label{notation section}
We use the following notation throughout.
\begin{itemize}
    \item $\textup{H}\mathbb{F}_2$ is the motivic Eilenberg-Mac Lane spectrum representing mod-2 motivic cohomology.
    \item $\textup{H}_{**}(X)$ is the mod-2 motivic homology of $X$.
    \item $\pi_{**}^F(X)$ are the bigraded motivic homotopy groups of $X$. These are suitably completed so that the motivic Adams spectral sequence always converges.
    \item $\mathbb{M}_2^F$ is the mod-2 motivic homology of a point.
    \item $\euscr{A}^\vee_F$ denotes the $F$-motivic dual motivic Steenrod algebra. 
    \item $\euscr{A}(n)^\vee_F$ denotes the dual of the subalgebra of the $F$-motivic Steenrod algebra generated by $\text{Sq}^1, \text{Sq}^2, \dots, \text{Sq}^{2^n}$. 
    \item KQ denotes the motivic spectrum representing hermitian K-theory, kq denotes its very effective cover, and ksp denotes the very effective cover of $\Sigma^{4,2}\text{KQ}$.
    \item $\text{Ext}^{***}_B(\mathbb{M}_2^F, X)$ will be abbreviated by $\text{Ext}^{***}_B(X)$.
    \item We grade Ext groups as $(s, f, w)$, where $s$ is the stem, $f$ is the Adams filtration, and $w$ is the motivic weight. We also let $cw = s-w$ denote the coweight, sometimes referred to in the literature as the Milnor-Witt degree.
    \item All charts are in Adams grading, meaning $(s, f)$, with motivic weight suppressed.
    \item All motivic spectra are implicitly 2-complete unless explictly stated otherwise.
    \item For E a motivic ring spectrum, we let $\text{E-}\textbf{mASS}^F(X)$ denote the motivic Adams spectral sequence based on E and converging to $\pi_{**}^F(X_\text{E}^\wedge)$. When $\text{E} = \text{H}\mathbb{F}_2$, we denote this spectral sequence by $\textbf{mASS}^F(X).$
    \item For $M$ an $\euscr{A}(1)^\vee_F$-comodule, we let $\textbf{aAHSS}(M)$ denote the algebraic Atiyah--Hirzebruch spectral sequence detailed in \cref{aAHSS generic}.
    \item We employ two algebraic shift functors, denoted $\Sigma^{p,q}(-)$ and $(-)\langle n \rangle$. For $M$ a module over $\text{Ext}$ and $x \in M$ an element of degree $(s,f,w)$, we have 
    \[|\Sigma^{p,q}x| = (s+p, f, w+q)\] 
    and 
    \[|x \langle n \rangle| = (s, f+n, w).\]
\end{itemize}

\subsection{Acknowledgments}
The work presented here constitutes a part of the author's thesis. The author thanks their advisors, Kyle Ormsby and John Palmieri, for their indispensable wisdom and valuable advice. The author also thanks J.D. Quigley for his generosity in sharing insight into this problem, and thanks Madeline Borowski, Sarah Petersen, Jay Reiter, and Alex Waugh for helpful discussions throughout this project. Finally, the author thanks Noah Wisdom for providing valuable feedback on an early draft, MJ Lenderman for inspiration \cite{ManningFireworks}, and an anonymous referee for helpful suggestions which clarified many technical points.

\part{Preliminaries}
\label{part1}
\section{Stable motivic homotopy theory}
\label{section 2}
In this section we work over an arbitrary field $F$ of characteristic different from 2 such that $\text{vcd}_2(F) < \infty$. We recall the $\textup{kq}$-resolution then review the dual motivic Steenrod algebra and motivic Brown--Gitler comodules.

\subsection{The $\textup{kq}$-resolution}
\label{2.1}
Originally constructed in \cite{Hor05} over fields with 2 invertible, and more recently, over more general base schemes in \cite{CalHarNar25}, there is an $(8, 4)$-periodic motivic $\mathbb{E}_\infty$-ring spectrum $\textup{KQ} \in \text{SH}(F)$ known as the \emph{hermitian \textup{K}-theory} spectrum. This spectrum represents hermitian K-theory in that there is an isomorphism \cite[Theorem 8.18]{CalHarNar25}
\[\pi^F_{s, w}(\textup{KQ}) = \pi_{s-2w}(\textup{GW}^\text{s}(F)),\]
where $\text{GW}^\text{s}(F)$ is the ordinary spectrum representing the symmetric hermitian K-theory of $F$.
In particular, $\pi_{8n, 4n}^F\textup{KQ} = \textup{GW}(F)$, the Grothendieck--Witt ring of quadratic forms over $F$, and the unit map $\mathbb{S} \to \text{KQ}$ induces Morel's isomorphism \cite{RSOfirst}
\[\pi^F_{-n, -n}\mathbb{S} \cong \textup{K}^{\textup{MW}}_n(F),\]
where $\text{K}^{\text{MW}}_n(F)$ denotes the $n^{th}$ Milnor K-theory of $F$.

Let $\textup{kq}$ denote the very effective cover of $\textup{KQ}$ \cite{ARO20} (see \cref{section ksp} for more details on the very effective filtration), which we will call the \emph{very effective hermitian \textup{K}-theory} spectrum.
There is a canonical Adams tower associated to the unit map $\mathbb{S} \to \textup{kq}$:
\[\begin{tikzcd}
	{\mathbb{S}} & {\Sigma^{-1, 0}\overline{\textup{kq}}} & {\Sigma^{-2, 0}\overline{\textup{kq}} \otimes \overline{\textup{kq}}} & \cdots \\
	\textup{kq} & {\Sigma^{-1, 0}\overline{\textup{kq}} \otimes \textup{kq}} & {\Sigma^{-2, 0}\overline{\textup{kq}} \otimes \overline{\textup{kq}} \otimes \textup{kq}}
	\arrow[from=1-1, to=2-1]
	\arrow[from=1-2, to=1-1]
	\arrow[from=1-2, to=2-2]
	\arrow[from=1-3, to=1-2]
	\arrow[from=1-3, to=2-3]
	\arrow[from=1-4, to=1-3]
\end{tikzcd}\]
Applying $\pi^F_{**}(-)$ yields the motivic Adams spectral sequence kq-$\textbf{mASS}^F(\mathbb{S})$ known as the \emph{$\textup{kq}$-resolution}. Its properties were first studied by Culver--Quigley \cite{CQ21}. We remark that we are working in the 2-complete category.
\begin{thm}[{\cite[Thm 2.1]{CQ21}}]
\label{kqres}
    The $\textup{kq}$-resolution is a strongly convergent spectral sequence of the form
    \[\textup{E}_1^{s, f, w} = \pi^F_{s+f, w}(\textup{kq} \otimes \overline{\textup{kq}}^{\otimes f}) \implies \pi^F_{s, w}\mathbb{S}_2^\wedge.\]
    The $d_r$-differentials have the form
    \[d_r:\textup{E}_r^{s,f,w} \to \textup{E}_r^{s-1, f+r, w}.\]
\end{thm}
\begin{remark}
    We can also look at the $\textup{kq}$-resolution in the integral, uncompleted setting. The unit map $\mathbb{S} \to \textup{kq}$ gives an isomorphism 
    \[\Pi_0^F(\textup{kq}) \cong \Pi_0^F(\mathbb{S}),\]
    where $\Pi_n^F(X) := \bigoplus_{k \in \mathbb{Z}}\pi_{n+k, k}^F(X)$ denotes the $k$-th coweight stem (also known as the $k$-th Milnor--Witt stem). This implies that $\textup{kq} \otimes \overline{\textup{kq}}^{\otimes n}$ is $n$-connected in Morel's homotopy $t$-structure. In fact, this is enough to show that $\mathbb{S}_{\textup{kq}}^\wedge \simeq \mathbb{S}$. 
    
    However, we will compute the $\textup{E}_1$-page of the kq-resolution by a motivic Adams spectral sequence based on the Eilenberg--Mac Lane spectrum $\textup{H}\mathbb{F}_2$. More precisely, we study te $\textbf{{mASS}}^F(\text{kq} \otimes \text{kq})$. This spectral sequence computes the homotopy groups of the $(2,\eta)$-completion of $\textup{kq} \otimes \textup{kq}$. When $F$ is a field of finite virtual 2-cohomological dimension, then this coincides with the 2-completion \cite[Theorem 1]{HKO-convergencemASS}. These conventions  force us to work in the 2-complete category (see also \cite[Section 2.2]{CQ21}).
\end{remark}

Let bo denote the connective cover of the classical real K-theory spectrum KO. The $\textup{kq}$-resolution is the motivic analogue of the bo-resolution in the following way. The $\textup{E}_1$-page of the bo-resolution is of the form
\[\textup{E}_1^{s, f} = \pi_{s+f}(\textup{bo} \otimes \overline{\textup{bo}}^{\otimes f}) \implies \pi_s(\mathbb{S}_2^\wedge),\]
where we may identify $\mathbb{S}_{\text{bo}}^\wedge \simeq \mathbb{S}_2^\wedge$ in the 2-complete setting \cite{Bousfield-localization}. There is a \textit{complex Betti realization functor}
\[\text{Be}^\mathbb{C}:\text{SH}(\mathbb{C}) \to \text{Sp},\]
induced by the analytic topology on a scheme $X \in \text{Sm}_\mathbb{C}$ after taking $\mathbb{C}$-points. Moreover, if $F$ admits an embedding into $\mathbb{C}$, then there is complex realization $\text{Be}^\mathbb{C}:\text{SH}(F) \to \text{Sp}$ defined in the same way.
The following result lets us compare between $\textup{kq}$ and bo:
\begin{thm}[{\cite[Lemma 2.13]{ARO20}}]
\label{C-betti}
    If $F$ admits a complex embedding, then there is an equivalence of spectra
    \[\textup{Be}^\mathbb{C}(\textup{kq})  \simeq \textup{bo}.\]
\end{thm}
Betti realization is symmetric monoidal and exact \cite{HelOrm-Galois}, giving an immediate corollary:
\begin{corollary}[{\cite[Cor 2.7]{CQ21}}]
\label{C-betti kq res}
    Betti realization sends the $\textup{kq}$-resolution to the $\textup{bo}$-resolution. In particular, there is a multiplicative map of spectral sequences:
    \[\textup{E}_r^{s, f, *} \to \textup{E}_r^{s, f}.\]
\end{corollary}

Just as in the classical case, the spectrum $\textup{kq}$ is not flat in the sense of Adams. Thus the $\textup{E}_1$-page does not compute the Hopf algebroid cohomology $\text{Ext}^{***}_{\pi^F_{**}(\textup{kq}\otimes \textup{kq})}(\pi^F_{**}(\textup{kq}), \pi^F_{**}(\textup{kq}))$, and so we must analyze the $\textup{E}_1$-page of the $\textup{kq}$-resolution.

\begin{remark}
\label{betti remark r vs c}
    There is a complex Betti realization functor in the case of $F = \mathbb{R}$ to which we may apply \cref{C-betti}. It is important to note that there is an obvious \textit{real Betti realization}
    \[\text{Be}^\mathbb{R}:\text{SH}(\mathbb{R}) \to \text{Sp}\]
    which is induced by the real analytic topology on a scheme $X \in \text{Sm}_\mathbb{R}$ after taking $\mathbb{R}$-points.
    These functors behave very differently.  For example, consider $\text{kq} \in \text{SH}(\mathbb{R})$. It was shown in \cite[Proposition 4.1]{Bann-realbetti} that $\pi_*(\text{Be}^\mathbb{R}(\text{kq})) \cong \mathbb{Z}[x]$, where $x \in \pi_4(\text{Be}^\mathbb{R}(\text{kq}))$. In particular, 
    $\text{Be}^\mathbb{C}(\text{kq}) \not\simeq \text{Be}^\mathbb{R}(\text{kq}).$
\end{remark}
\subsection{The dual motivic Steenrod algebra} 
\label{section motivic steenrod algebra}
Let $\mathbb{M}_2^F$ denote the \emph{mod-2 motivic homology of a point}. There is an isomorphism due to Voevodsky \cite{Voemotiviccohomology}:
\[\mathbb{M}_2^F = (\text{K}^\text{M}_*(F)/2)[\tau],\]
where $\text{K}^\text{M}_*(F)$ denotes the Milnor K-theory of $F$ and $|\tau|=(0, -1)$.
Recall the \emph{dual motivic Steenrod algebra}:
\begin{thm}[{\cite[Section 12]{Voereduced}}]
\label{dual motivic steenrod}
     The dual motivic Steenrod algebra $\euscr{A}^\vee_F = \pi_{**}^F(\textup{H}\mathbb{F}_2 \otimes \textup{H}\mathbb{F}_2)$ is given by
    \[\euscr{A}^\vee_F = \mathbb{M}_2^F[\overline{\xi}_1, \overline{\xi}_2, \hdots, \overline{\tau}_0, \overline{\tau}_1, \hdots]/(\overline{\tau}_i^2 = \rho\overline{\tau}_{i+1} + \rho\overline{\tau}_0\overline{\xi}_{i+1} + \tau \overline{\xi}_{i+1}),\]
    where $|\overline{\xi}_i|=(2^{i+1}-2, 2^i-1)$ and $|\overline{\tau}_i| = (2^{i+1}-1, 2^i-1)$. There are structure maps 
    \[\eta_L, \eta_R:\mathbb{M}_2^F \to \euscr{A}^\vee_F, \quad\Delta:\euscr{A}^\vee_F \to\euscr{A}^\vee_F \otimes_{\mathbb{M}_2^F}A_F^\vee\]
    which are determined by the following formulae:
    \[\begin{array}{ll}
    \eta_L(\rho) = \rho &   \eta_L(\tau) = \tau; \\
    \eta_R(\rho)=\rho &   \eta_R(\tau)=\tau+\overline{\tau}_0\rho; \\
    \Delta(\overline{\xi}_k) = \sum_{i+j=k}\overline{\xi}_i \otimes \overline{\xi}_j^{2^i}, & \Delta(\overline{\tau}_i) = \sum_{i+j=k}\overline{\tau}_i \otimes \overline{\xi}_j^{2^i} + 1 \otimes \overline{\tau}_k.
    \end{array}\]
\end{thm}
We will also encounter various subalgebras and quotient algebras of $\euscr{A}^\vee_F$.

\begin{definition}
    For $n \geq 0$, let $\euscr{A}(n)^\vee_F$ be the quotient algebra
    \[\euscr{A}(n)_F^\vee \cong \euscr{A}^\vee_F/(\overline{\xi}_1^{2^n}, \overline{\xi}_2^{2^{n-1}}, \dots, \overline{\xi}^2_n, \overline{\xi}_{n+1}, \overline{\xi}_{n+2,} \dots, \overline{\tau}_{n+1}, \overline{\tau}_{n+2}, \dots).\]
    There is a sub-Hopf algebra of the motivic Steenrod algebra
    \[\euscr{A}(n)_F = \langle \text{Sq}^1, \text{Sq}^2, \hdots, \text{Sq}^{2^n} \rangle;\]
    let $(\euscr{A}\modmod \euscr{A}(n))^\vee_F$ be the subalgebra of $\euscr{A}^\vee_F$ given by
    \[(\euscr{A} \modmod \euscr{A}(n))^\vee_F = \mathbb{M}_2^F[\overline{\xi}_1^{2^n}, \overline{\xi}_2^{2^{n-1}}, \hdots, \overline{\tau}_{n+1}, \hdots]/(\overline{\tau}_i^2 = \rho \overline{\tau}_{i+1}+\rho \overline{\tau_0}\overline{\xi}_{i+1} + \tau \overline{\xi}_{i+1}).\]
\end{definition}

\begin{remark}
\label{algebroid}
    An important distinction between $\euscr{A}^\vee_\mathbb{R}$ and $\euscr{A}^\vee_\mathbb{C}$ is that while the latter is a Hopf algebra, the former is a Hopf algebroid \cite[Appendix A]{Rav86}. There is a canonical class $\rho \in \text{K}^\text{M}_1(F)/2$ representing $-1$, giving rise to a class in $\mathbb{M}_2^F$. If $-1$ is a square in $F$, then this class is trivial in motivic homology and $\euscr{A}^\vee_F$ is an honest Hopf algebra. If $-1$ is not a square, then there is an action of the motivic Steenrod algebra:
    \[\text{Sq}^1(\tau) = \rho.\]
    In particular, upon dualizing we see that $\mathbb{M}_2^F$ is not central in $\euscr{A}_F^\vee$, implying that $(\mathbb{M}_2^F, \euscr{A}_F^\vee)$ is a Hopf algebroid.
\end{remark}

As in the classical case, one can express the comodule algebra $(\euscr{A} \modmod \euscr{A}(n))_F^\vee$ as a cotensor product:
\[(\euscr{A} \modmod \euscr{A}(n))^\vee_F = \euscr{A}_F^\vee \Box_{\euscr{A}(n)^\vee_F}\mathbb{M}_2^F.\]
These $\euscr{A}^\vee$-comodule algebras come up naturally in our work as the motivic homology of certain motivic spectra. For instance, we have that $\textup{H}_{**}(\textup{H}\mathbb{Z}) \cong (\euscr{A} \modmod \euscr{A}(0))^\vee_F$. We also have the following:

\begin{thm}[{\cite[Remark 2.14]{ARO20}}]
\label{kq homology}
    There is an isomorphism of $\euscr{A}^\vee_F$-comodule algebras
    \[\textup{H}_{**}(\textup{kq}) \cong (\euscr{A} \modmod \euscr{A}(1))_F^\vee = \mathbb{M}_2^F[\overline{\xi}_1^2, \overline{\xi}_2, \hdots, \overline{\tau}_2, \overline{\tau}_3, \hdots]/(\overline{\tau}_i^2 = \rho \overline{\tau}_{i+1}+\rho \overline{\tau_0}\overline{\xi}_{i+1} + \tau \overline{\xi}_{i+1})   
    .\]
\end{thm}

We investigate now the $\euscr{A}^\vee_F$-comodule structure on the homology $\textup{H}_{**}(\textup{kq})$. As $(\euscr{A} \modmod \euscr{A}(1))^\vee_F \subseteq \euscr{A}_F^\vee$ is a subcoalgebra, the coaction comes from the restriction of the comultiplication on $\euscr{A}^\vee_F$:
\[\psi:\textup{H}_{**}(\textup{kq}) \to \euscr{A}_F^\vee \otimes_{\mathbb{M}_2^F} \textup{H}_{**}(\textup{kq}).\]
Later, we will need to understand the $\euscr{A}(1)^\vee_F$-comodule structure on the homology $\textup{H}_{**}(\textup{kq})$. Observe that by definition, $\euscr{A}(1)_F^\vee$ takes the form
\[\euscr{A}(1)^\vee_F = \mathbb{M}_2^F[\overline{\xi}_1, \overline{\tau}_0, \overline{\tau}_1]/(\overline{\tau}_0^2 = \rho \overline{\tau}_1 + \rho \overline{\tau}_0\overline{\xi}_1+\tau\overline{\xi}_1, \overline{\tau}_1^2, \overline{\xi}_1^2).\]
Since $\euscr{A}(1)^\vee_F$ is a quotient of $\euscr{A}_F^\vee$, we can deduce that the coaction is given by projection of the usual coaction of $\euscr{A}^\vee_F$ on $\textup{H}_{**}(\textup{kq})$:
\[\psi: \textup{H}_{**}(\textup{kq}) \to \euscr{A}^\vee_F \otimes_{\mathbb{M}_2^F}\textup{H}_{**}(\textup{kq}) \to \euscr{A}(1)^\vee_F \otimes_{\mathbb{M}_2^F} \textup{H}_{**}(\textup{kq}).\]
Moreover, since $\textup{kq}$ is a ring spectrum, its homology is a comodule algebra. Thus it suffices to determine the coaction on the generators. The following is immediate from this discussion.
\begin{proposition}[{\cite[Corollary 3.6]{CQ21}}]
\label{a(1) coaction}
    The $\euscr{A}(1)^\vee_F$-coaction on $\textup{H}_{**}(\textup{kq})$ is determined by the following formulae:
    \[\begin{array}{lr}
    \psi(\overline{\xi}_1^2) = 1 \otimes \overline{\xi}_1^2; &\\
    \psi(\overline{\xi}_i) = 1 \otimes \overline{\xi}_i + \overline{\xi}_1 \otimes \overline{\xi}_{i-1}^2 & i \geq 2;\\
    \psi(\overline{\tau}_i) = 1 \otimes \overline{\tau}_i + \overline{\tau}_0 \otimes \overline{\xi}_i + \overline{\tau}_1 \otimes \overline{\xi}_{i-1}^2 &  i \geq 2.\end{array}\]
\end{proposition}

The following result will be useful in the sequel.
\begin{proposition}
\label{Kunneth}
    There is a K\"unneth isomorphism of $\euscr{A}_F^\vee$-comodule algebras
    \[\textup{H}_{**}(\textup{kq} \otimes \textup{kq}) \cong \textup{H}_{**}(\textup{kq}) \otimes_{\mathbb{M}_2^F}\textup{H}_{**}(\textup{kq}).\]
\end{proposition}

\begin{proof}
    It was shown by Voevodsky \cite{Voereduced} that the motivic Steenrod algebra $\euscr{A}_F$ is free as a module over the cohomology of a point. One basis for this algebra is the motivic Milnor basis \cite{Voereduced,Kylingmilnor}. Since both $\text{Sq}^1$ and $\text{Sq}^2$ are elements of this basis, the quotient $\euscr{A}\modmod \euscr{A}(1)_F \cong \euscr{A}_F \otimes_{\euscr{A}(1)_F}\mathbb{M}_2^F$ is also free over the cohomology of a point. Consider the following K\"unneth spectral sequence \cite{DI05}:
    \[\textup{E}_2 = \text{Tor}^{\mathbb{M}_2^F}(\textup{H}_{**}(\textup{kq}), \textup{H}_{**}(\textup{kq})) \implies \textup{H}_{**}(\textup{kq} \otimes \textup{kq}). \]
    Since $\textup{H}_{**}(\textup{kq}) \cong (\euscr{A}\modmod \euscr{A}(1))_F^\vee$ is the $\mathbb{M}_2^F$-linear dual of a finitely-generated free module, it is also free. Thus the higher Tor vanishes and the spectral sequence collapses with no extensions, giving the result.
\end{proof} 

\subsection{Motivic Brown--Gitler comodules}
\label{motivic bg}
We next introduce motivic Brown--Gitler comodules. We define the \textit{Mahowald weight filtration} on $\euscr{A}^\vee_F$ by setting 
\[wt(\overline{\xi}_i)=wt(\overline{\tau}_i)=2^i, \quad wt(\tau)=wt(\rho)=0\] and letting $wt(xy)=wt(x)+wt(y)$. This naturally extends to the subalgebras $(\euscr{A} \modmod \euscr{A}(n))_F^\vee$. The \textit{$F$-motivic Brown--Gitler comodule} $B_n^F(k)$ is the $\euscr{A}(n)^\vee_F$-comodule
\[B_n^F(k) = \langle x \in (\euscr{A}\modmod \euscr{A}(n))_F^\vee : wt(x) \leq 2^{n+1}k\rangle.\]
\begin{remark}
    One can also define Brown--Gitler comodules using the same weight filtration on the subalgebras $(\euscr{A} \modmod \euscr{E}(n))_F^\vee$, leading to a different family of comodules. These families of Brown--Gitler comodules were used by the author, Petersen and Tatum in \cite{MorPetTat-BPGL1}, particularly in the $n=0$ case. As we will only be working with the case defined above, and since $\euscr{A}(0) \cong \euscr{E}(0)$, there should be no ambiguity in the notation.
\end{remark}
We will specialize to the following two cases throughout our work.
\begin{definition}[{compare with \cite[Def. 3.10]{CQ21}}]
The \textit{$k$-th integral motivic Brown--Gitler comodule} is 
\[B_0^F(k) := \langle x \in (\euscr{A} \modmod \euscr{A}(0))_F^\vee : wt(x) \leq 2k\rangle.\]
The \textit{$k$-th $\textup{kq}$ motivic Brown--Gitler comodule} is
\[B^F_1(k):=\langle x \in (\euscr{A} \modmod \euscr{A}(1))^\vee_F: wt(x) \leq 4k\rangle .\]
\end{definition}
We remark that in \cite{CQ21}, different notation is used for the Brown--Gitler comodules.

\begin{example}
    The motivic Brown--Gitler comodule $B_0^F(0)$ is given by
    \[B_0^F(0) = \langle x \in (\euscr{A} \modmod \euscr{A}(0))_F^\vee: wt(x) \leq 0 \rangle \cong \mathbb{M}_2^F,\]
    with trivial $\euscr{A}(1)^\vee_F$-comodule structure. 
\end{example}
\begin{example} 
\label{HZ1 construction}
    The motivic Brown--Gitler comodule $B_0^F(1)$ is given by
    \[B_0^F(1) = \langle x \in (\euscr{A} \modmod \euscr{A}(0))^\vee_F : wt(x) \leq 2\rangle = \mathbb{M}_2^F\{1, \bar{\xi}_1, \bar{\tau}_1\},\]
    where the $\euscr{A}(1)_F^\vee$-comodule structure is given by 
    \[\psi(\bar{\xi}_1) = 1 \otimes \bar{\xi}_1+\bar{\xi}_1 \otimes 1, \quad \psi(\bar{\tau}_1) = 1 \otimes \bar{\tau}_1 + \bar{\tau}_0 \otimes \bar{\xi}_1.\]
    In other words, $B_0^F(1)$ is presented as the algebraic cell complex in \Cref{B_0(1) diagram}.
\end{example}
\begin{figure}[h]
    \begin{center}
    \begin{tikzpicture}
    \node[shape = circle, draw=black, scale=.75] (0) at (0,0) {}; 
    \node[shape = circle, draw=black, scale=.75] (2) at (0,2) {};
    \node[shape = circle, draw=black, scale=.75] (3) at (0,3) {};

    \path (0) edge[color=blue, thick, bend right =50] (2);
    \draw (0.75,1) node{$h_1$};
    \path (2) edge[thick]  (3);
    \draw (0.75,2.5) node{$h_0$};
    \end{tikzpicture}  
    \end{center}
    \caption{The motivic Brown--Gitler comodule $B_0^F(1)$}
    \label{B_0(1) diagram}
\end{figure}
In fact, we can realize $B_0^F(1)$ as the mod-2 motivic cohomology of a spectrum \cite[Example 3.12]{CQ21}. There is a geometric classifying space \cite[Section 4]{MV99} $B\mu_2$ of the affine group scheme $\mu_2$ of second roots of unity. Take $L_2$ to be the simplicial 4-skeleton of $B\mu_2$. The inclusion of the $(3,2)$ cell of $L_2$ gives a map
 \[S^{3,2} \hookrightarrow L_2.\]
Let $X$ be the cofiber; then $\textup{H}\mathbb{Z}^F_1:=\Sigma^{4,2}F(X, \mathbb{S})$ is a finite complex realizing $B_0^F(1)$ in homology.

\begin{remark}
\label{motivic BG spectra}
    In classical topology there exist integral and bo Brown--Gitler spectra \cite{Shi83, GJM86}. These are finite spectra $\textup{H}\mathbb{Z}_k$ and $\text{bo}_k$, respectively, with the property that 
    \[\text{H}_*(\textup{H}\mathbb{Z}_k) \cong B^{\text{cl}}_0(k), \quad \text{H}_*(\text{bo}_k) \cong B_1^{\text{cl}}(k).\] 
    However, while $C_2$-equivariant analogues exist for integral Brown--Gitler spectra \cite{LiPetTat25}, there are currently no constructions of any motivic Brown--Gitler spectra. One possible construction could come from the motivic lambda algebra of Balderamma--Culver--Quigley \cite{BalCulQui25} and using the classical work of Goerss--Jones--Mahowald \cite{GJM86} as inspiration. An alternative construction, at least over $\mathbb{C}$, is in using the theory of filtered spectra in \cite{GIKR22}. We will explore the construction of motivic Brown--Gitler spectra in future work.
\end{remark}
The $\euscr{A}(1)_F^\vee$-coaction on $(\euscr{A} \modmod \euscr{A}(1))^\vee_F$ preserves the Mahowald weight, as can be seen by our formulae in \cref{a(1) coaction}. This gives the following:
\begin{proposition}
\label{preserves weight}
    The motivic integral Brown--Gitler comodules $B_0^F(k)$ are $\euscr{A}(1)^\vee_F$-subcomodules of $(\euscr{A} \modmod \euscr{A}(0))^\vee_F$.
\end{proposition}
A key property of the integral motivic Brown--Gitler comodules is that they give rise to a splitting of the homology $\textup{H}_{**}(\textup{kq})$. We refer the reader to \cite{CQ21} for a proof.

\begin{thm} [{\cite[Theorem 3.20]{CQ21}}]
\label{kq brown gitler homology}
    There is an isomorphism of $\euscr{A}(1)_F^\vee$-comodules
    \[(\euscr{A}\modmod \euscr{A}(1))_F^\vee \cong \bigoplus_{k \geq 0}\Sigma^{4k, 2k}B_0^F(k).\]
\end{thm}

The following result will also be useful.
\begin{proposition}
\label{kq bar homology}
    There is an isomorphism of $\euscr{A}(1)^\vee_F$-comodules:
    \[\textup{H}_{**}(\overline{\textup{kq}}) \cong \bigoplus_{k \geq 1}\Sigma^{4k, 2k}B_0^F(k).\]
\end{proposition}

\begin{proof}
    The defining cofiber sequence
    \[\mathbb{S} \to \textup{kq} \to \overline{\textup{kq}}\]
    gives a long exact sequence in homology of the form
    \[ \cdots \to \textup{H}_{**}(\mathbb{S}) \to \textup{H}_{**}(\textup{kq}) \to \textup{H}_{**}(\overline{\textup{kq}}) \to \text{H}_{*-1, *}(\mathbb{S}) \to \cdots .\]
    We have isomorphisms 
    \[\textup{H}_{**}(\mathbb{S}) \cong \mathbb{M}_2^F \cong B_0^F(0)
    \quad \text{and}\quad \textup{H}_{**}(\textup{kq})\cong \bigoplus_{k \geq 0}\Sigma^{4k, 2k}B_0^F(k).\]
    The unit map $\mathbb{S} \to \text{kq}$ induces an inclusion of the bottom summand $B_0^F(0) \hookrightarrow \bigoplus_{k \geq 0}B_0^F(k)$, so by the long exact sequence, $\textup{H}_{**}(\overline{\textup{kq}})$ is trivial in those degrees. The long exact sequence also gives an isomorphism between $\textup{H}_{**}(\textup{kq})$ and $\textup{H}_{**}(\overline{\textup{kq}})$ outside of these degrees since $\textup{H}_{**}(\mathbb{S})$ vanishes, giving the result.
\end{proof}

We close this section by recalling useful short exact sequences of $\euscr{A}(1)^\vee_F$-comodules relating motivic Brown--Gitler comodules. We refer the reader to \cite{CQ21} for a proof.

\begin{proposition}[{\cite[Lemma 3.21]{CQ21}}]
\label{ses bg}
    There are short exact sequences of $\euscr{A}(1)^\vee_F$-comodules:
    \[0 \to \Sigma^{4k, 2k}B_0^F(k) \to B_0^F(2k) \to B_1^F(k-1) \otimes_{\mathbb{M}_2^F} (\euscr{A}(1) \modmod \euscr{A}(0))_F^\vee \to 0,\]
    \[0 \to \Sigma^{4k, 2k}B_0^F(k) \otimes_{\mathbb{M}_2^F} B_0^F(1) \to B_0^F(2k+1) \to B_1^F(k-1) \otimes_{\mathbb{M}_2^F} (\euscr{A}(1) \modmod \euscr{A}(0))_F^\vee \to 0.\]
\end{proposition}

\begin{remark}
    The short exact sequences in \cref{ses bg} are motivic analogues of classical short exact sequences of Brown--Gitler comodules. These arise by applying mod-2 homology to cofiber sequences relating integral and bo Brown--Gitler spectra \cite{Pearson-koBG}. One would imagine the same to be true for the conjectural motivic Brown--Gitler spectra.
\end{remark}

\section{A splitting of symplectic K-theory}
\label{section ksp}
In this section, we record a motivic analogue of a classical splitting result. Let $\text{bsp} \simeq \tau_{\geq 0}\Sigma^4\text{KO}$ denote the connective cover of $\Sigma^4\text{KO}$. Then there is an equivalence of spectra \cite[Theorem 1.5]{MahMil76}:
\[\text{bsp} \simeq \text{bo} \otimes \text{H}\mathbb{Z}_1,\]
where $\text{H}\mathbb{Z}_1$ denotes the first integral Brown--Gitler spectrum.
We prove an analogous splitting for the motivic symplectic K-theory spectrum ksp by using the very effective slice filtration.

\subsection{Very effective slice filtration}
The very effective slice filtration was introduced by Spitzweck--\O stv\ae r in \cite[Section 5]{SpitzweckOstvaer-twistedKtheory} and further studied by Bachmann in \cite{Bachmann-generalized}. We follow the treatments given in \cite[Section 1]{Bann-realbetti} and \cite[Section 13]{BachmannHoyois-Norms}.

The category of \textit{very effective motivic spectra}, denoted $\text{SH}(F)^{\text{veff}}$, is the full subcategory of $\text{SH}(F)$ generated under colimits by $\Sigma^{n,0}\Sigma_+^\infty X$, where $X \in \text{Sm}_F$ and $n \geq 0$. The category of \textit{very n-effective motivic spectra} for $n \in \mathbb{Z}$, denoted $\text{SH}(F)^{\text{veff}}(n)$, has objects $\Sigma^{2n, n}X$ for $X \in \text{SH}(F)^{\text{veff}}$. These categories assemble into colimit-preserving inclusions
\[\cdots \subset \text{SH}(F)^{\text{veff}}(n+1) \subset \text{SH}(F)^{\text{veff}}(n) \subset \text{SH}(F)^{\text{veff}}(n-1) \subset \cdots \subset\text{SH}(F).\]
In particular, the inclusion functor $\tilde{\text{i}}_n:\text{SH}(F)^{\text{veff}}(n) \to \text{SH}(F)$ admits a right adjoint called the \textit{n-th very effective cover}, denoted $\tilde{\text{f}}_n$. If $X \in \text{SH}(F)$ is any motivic spectrum, the very effective covers give a functorial filtration
\[\cdots \to\tilde{\text{f}}_{n+1}X \to \tilde{\text{f}}_nX \to \tilde{\text{f}}_{n-1}X \to \cdots \to X\]
which we will call the \textit{very effective slice tower} for $X$.
\begin{example}
    In this notation, we have that $\text{kq} = \tilde{\text{f}}_0\text{KQ}$.
\end{example}
The \textit{n-th very effective slice} of a motivic spectrum $X$, denoted $\tilde{\text{s}}_nX$, is defined by the cofiber sequence
\[\tilde{\text{f}}_{n+1}X \to \tilde{\text{f}}_nX \to \tilde{\text{s}}_nX.\]
By construction, the very effective cover and slice functors behave well with respect to $\mathbb{P}^1$-suspension.

\begin{lemma}[{\cite[Lemma 8]{Bachmann-generalized}}]\label{slice lemma bachmann}
    For any $X \in \textup{SH}(F)$, we have
    \[\begin{array}{lr}
        \tilde{\textup{f}}_{n+1}(\Sigma^{2,1}X) \simeq \Sigma^{2,1}\tilde{\textup{f}}_{n}X, & \tilde{\textup{s}}_{n+1}(\Sigma^{2,1}X) \simeq \Sigma^{2,1}\tilde{\textup{s}}_nX.
    \end{array}\]
\end{lemma}

\begin{remark}
    Heuristically, the very effective slice filtration acts as a motivic stand in for the Whitehead tower in classical stable homotopy. We give an example to justify this comparison. Let $A$ and $Y$ be classical spectra, and let $P_nY$ be defined by the cofiber sequence
    \[\tau_{\geq n+1}Y \to \tau_{\geq n}Y \to P_nY\]
    One way to obtain the Atiyah--Hirzebruch spectral sequence 
    \[\textup{E}_2 = A_*(P_nY) \implies A_*(Y)\]
    is by smashing the Whitehead tower for $Y$ with $A$ and applying homotopy groups. Now, let $E$ and $X$ be motivic spectra. Then there is a (generalized) very effective slice spectral sequence \cite{CarrickHillRavenel-homologicalslice}
    \[\textup{E}_2=E_{**}(\tilde{\text{s}}_nX) \implies E_{**}(X)\]
    obtained by smashing the very effective slice tower for $X$ with $E$ and applying homotopy groups. Over $F=\mathbb{C}$, Betti realization sends the very effective slice spectral sequence to an Atiyah--Hirzebruch spectral sequence.
\end{remark}

\subsection{A splitting of symplectic K-theory}
Let ksp denote the very effective cover $\tilde{\text{f}}_0(\Sigma^{4,2}\text{KQ})$. Note that by \cref{slice lemma bachmann}, this implies that $\Sigma^{4,2}\text{ksp} \simeq \tilde{\text{f}}_2\text{KQ}.$

\begin{proposition}
\label{ksp homology} 
    There is an isomorphism of $\euscr{A}^\vee_F$-comodules:
    \[\textup{H}_{**}(\textup{ksp}) \cong (\euscr{A} \modmod \euscr{A}\langle \textup{Sq}^1, \textup{Sq}^5 \rangle )_F^\vee\]
\end{proposition}

\begin{proof}
    The idea is to mimic the classical proof, replacing the Whitehead tower with the very effective slice tower for KQ. This tower gives a cofiber sequence 
    \begin{align}
    \label{slice cofiber}
    \tilde{\text{f}}_3\text{KQ} \to \tilde{\text{f}}_2\text{KQ} \to \tilde{\text{s}}_2\text{KQ}.
    \end{align}
    By definition, we have $\tilde{\text{f}}_2\text{KQ} \simeq \Sigma^{4,2}\text{ksp}$. Bachmann's identification of the very effective slices of KQ \cite[Theorem 16]{Bachmann-generalized} shows that 
    \[\begin{array}{lr}
    \tilde{\text{s}}_2\text{KQ} \simeq \Sigma^{4,2}\text{H}\mathbb{Z}, & \tilde{\text{s}}_3\text{KQ} \simeq 0.
    \end{array}\]
    By Bott periodicity for KQ and \cref{slice lemma bachmann}, we have
    \[\tilde{\text{f}}_4\text{KQ} \simeq \tilde{\text{f}}_4(\Sigma^{8,4}\text{KQ}) \simeq \Sigma^{8,4}\tilde{\text{f}}_0\text{KQ} = \Sigma^{8,4}\text{kq}.\]
    Since $\tilde{\text{s}}_3\text{KQ} \simeq 0$, the defining cofiber sequence gives an equivalence $\tilde{\text{f}}_4\text{KQ} \simeq \tilde{\text{f}}_3\text{KQ} \simeq \Sigma^{8,4}\text{kq}$. This allows us to rewrite the cofiber sequence (\ref{slice cofiber}) as
    \begin{align}
    \label{slice cofiber good}
    \Sigma^{8,4}\text{kq} \to \Sigma^{4,2}\text{ksp} \to \Sigma^{4,2}\text{H}\mathbb{Z}.
    \end{align}
    Using the long exact sequence in homology and the isomorphisms from \cref{kq homology}
    \[\text{H}_{**}(\text{H}\mathbb{Z}) \cong ( \euscr{A} \modmod \euscr{A}(0))^\vee_F, \quad \text{H}_{**}(\text{kq}) \cong (\euscr{A} \modmod \euscr{A}(1))_F^\vee,\]
    and comparing with the short exact sequence of $\euscr{A}^\vee_F$-comodules:
    \[0 \to \Sigma^{4,2}(\euscr{A} \modmod \euscr{A} \langle \text{Sq}^1, \text{Sq}^5 \rangle)^\vee_F \to \Sigma^{4,2}(\euscr{A} \modmod \euscr{A}(0))^\vee_F \to \Sigma^{8,4}(\euscr{A} \modmod \euscr{A}(1))^\vee_F \to 0\]
    gives the result, as in the classical case.
\end{proof}

Recall the integral motivic Brown--Gitler spectrum $\textup{H}\mathbb{Z}_1^F$ constructed in \cref{HZ1 construction}.

\begin{thm}
\label{ksp hz1}
    There is an equivalence of motivic spectra
    \[\textup{ksp} \simeq \textup{kq} \otimes \textup{H}\mathbb{Z}_1^F.\]
\end{thm}

\begin{proof}
    As $\textup{kq}$ is a ring spectrum, there is a multiplication map
    \[\mu:\textup{kq} \otimes \textup{kq} \to \textup{kq}.\]
    In homology, this map is the usual multiplication in $\euscr{A}^\vee_F$, where we identify the left side using \cref{Kunneth}.
    This induces the $\textup{kq}$-module structure on ksp, realized as a map
    \[\mu':\text{ksp} \otimes \textup{kq} \to \textup{ksp}.\]
    By \cref{ksp homology}, both $\text{Sq}^2$ and $\text{Sq}^1\text{Sq}^2$ act nontrivially on the degree $(0,0)$ generator of $\text{H}_{**}(\text{ksp})$, hence there is an inclusion
    \[\iota:\textup{H}\mathbb{Z}_1^F \to \text{ksp}\]
    inducing the obvious map in homology. One can see now that the composition
    \[\textup{H}\mathbb{Z}_1^F \otimes \textup{kq} \xrightarrow{\iota \otimes 1}\text{ksp} \otimes \textup{kq} \xrightarrow{\mu'}\text{ksp}\]
    is an isomorphism in homology. Since we are working in the 2-complete category, this gives an equivalence, finishing the proof.
\end{proof}

As we will see, this spectrum level splitting allows us to calculate $\pi_{**}^F(\text{ksp})$ in a particularly nice way.

\begin{remark}
\label{motivic j}
    There is another way to relate the motivic spectra ksp and kq that is relevant to our work. In \cite{BH21}, Bachmann--Hopkins construct Adams operations $\psi^q$ on kq. In \cite{kongquigley}, the homotopy groups of the spectrum L were computed, where L sits in a cofiber sequence
    \[\text{L} \to \text{kq} \xrightarrow{\psi^3-1} \text{kq}.\]
    However, Bachmann--Hopkins show that the map $\psi^3-1$ factors through the second very effective cover of kq. At least 2-locally, this defines a motivic image of J spectrum $\text{j}_\text{O}$, which sits in a cofiber sequence
    \[\text{j}_{\text{O}} \to \text{kq} \xrightarrow{\psi^3-1}\Sigma^{4,2}\text{ksp}.\]
    Culver--Quigley show that the $\mathbb{C}$-motivic 0- and 1-lines of the $\text{kq}$-resolution are isomorphic to $\pi_{**}^\mathbb{C}(\text{j}_\text{O})$ \cite[Theorem 7.13]{CQ21}. We imagine that the same is true over arbitrary base fields.
\end{remark}

\section{The inductive process to calculate $\pi_{**}^F(\mathrm{kq} \otimes \mathrm{kq})$}
\label{inductive process section}
In this section, we consider the motivic Adams spectral sequence converging to $\pi_{**}^{F}(\textup{kq} \otimes \textup{kq})$. We show that the $\textup{E}_2$-page splits using motivic Brown--Gitler comodules, then outline an inductive process to compute the $\textup{E}_2$-page by a series of algebraic Atiyah--Hirzebruch spectral sequences.

\subsection{The motivic Adams spectral sequence for $\pi_{**}^F(\textup{kq} \otimes \textup{kq})$}
There is an $\textup{H}\mathbb{F}_2$-motivic Adams spectral sequence computing the cooperations (\textbf{mASS}$^F$($\textup{kq} \otimes \textup{kq}$)). Since $\text{H}\mathbb{F}_2$ is flat in the sense of Adams, this takes the form \cite{HKO-convergencemASS}:
\[\textup{E}_2^{s,f,w} = \text{Ext}^{s,f,w}_{\euscr{A}_F^\vee}(\textup{H}_{**}(\textup{kq} \otimes \textup{kq})) \implies \pi_{s,w}^F(\textup{kq} \otimes \textup{kq}), \quad d_r:\text{E}_r^{s,f,w} \to \text{E}_r^{s-1, f+r, w}.\]
We can rewrite the $\textup{E}_2$-page of this spectral sequence in a convenient way.

\begin{thm}
\label{e2 mass}
    The $\textup{E}_2$-page of the \textup{\textbf{mASS}$^F$($\textup{kq} \otimes \textup{kq}$)} is given by
    \[\textup{E}_2^{s,f,w} = \bigoplus_{k \geq 0}\textup{Ext}^{s,f,w}_{\euscr{A}(1)_F^\vee}(\Sigma^{4k, 2k}B_0^F(k)).\]
\end{thm}

\begin{proof}
    By \cref{kq homology} and \cref{Kunneth}, we have an isomorphism of $\euscr{A}^\vee_F$-comodule algebras
    \[\textup{H}_{**}(\textup{kq} \otimes \textup{kq}) \cong (\euscr{A} \modmod \euscr{A}(1))_F^\vee \otimes_{\mathbb{M}_2^F} (\euscr{A} \modmod \euscr{A}(1))^\vee_F.\]
    Combined with a standard change of rings argument \cite[Appendix A]{Rav86}, we can rewrite the $\textup{E}_2$-page of the \textbf{mASS}$^F$($\textup{kq} \otimes \textup{kq}$) as
    \[\textup{E}_2^{s,f,w} = \textup{Ext}^{s,f,w}_{\euscr{A}(1)^\vee_F}((\euscr{A} \modmod \euscr{A}(1))^\vee_F).\]
    Finally, using \cref{kq brown gitler homology} and the fact that $\mathbb{M}_2^F$ is compact as an $\euscr{A}(1)^\vee_F$-comodule, we can rewrite this as
    \[\textup{E}_2^{s,f,w} = \bigoplus_{k \geq 0}\text{Ext}_{\euscr{A}(1)_F^\vee}^{s,f,w}(\Sigma^{4k, 2k}B_0^F(k)),\]
    proving the theorem.
\end{proof}
Thus, to compute the $\textup{E}_2$-page of the \textbf{mASS}$^F(\textup{kq} \otimes \textup{kq})$, we must compute the trigraded groups $\text{Ext}^{***}_{\euscr{A}(1)^\vee_F}(B_0^F(k))$ for $k \geq 0$. Note that as the integral Brown--Gitler comodules are subcomodules of $(\euscr{A} \modmod \euscr{A}(0))_F^\vee$, we understand their structure as $\euscr{A}(1)^\vee_F$-comodules.

\begin{remark}
    Up until this point, all computation has been a precise motivic analogue of the classical work of Mahowald \cite{Mah81}. However, Mahowald's computation of $\pi_{*}(\text{bo} \otimes \text{bo})$ heavily relies on the fact that the classical integral Brown--Gitler spectra $\textup{H}\mathbb{Z}_k$ give rise to a topological splitting:
    \[\text{bo} \otimes \text{bo} \simeq \bigoplus_{k \geq 0} \Sigma^{4k}\text{bo} \otimes \textup{H}\mathbb{Z}_k.\]
    As we have noted, there are currently no known motivic analogues of integral Brown--Gitler spectra except in special cases, and so we must use different methods.
\end{remark}

\subsection{Computing the summands of the $\textbf{mASS}^F(\text{kq} \otimes \text{kq})$}
\label{2.5}
\label{aAHSS generic}
We now wish to compute the groups $\text{Ext}^{***}_{\euscr{A}(1)_F^\vee}(B_0^F(k))$ appearing in the decomposition of the $\textup{E}_2$-page from \cref{e2 mass}. To do this, one can use a series of algebraic Atiyah--Hirzebruch spectral sequences (\textbf{aAHSS}) along with the short exact sequences of Brown--Gitler comodules from \cref{ses bg}. This is originally inspired by the work of Behrens--Ormsby--Stapleton--Stojanoska \cite{BOSS19}. We explain the method here for clarity:
\begin{enumerate}
    \item Compute $\text{Ext}^{***}_{\euscr{A}(1)^\vee_F}(B^F_0(1))/(v_1\text{-torsion})$ using an \textbf{aAHSS}:
    \[\text{Ext}^{***}_{\euscr{A}(1)^\vee_F}(\mathbb{M}_2^F) \otimes \mathbb{M}_2^F\{[1], [\overline{\xi}_1], [\overline{\tau}_1]\}\implies \text{Ext}^{***}_{\euscr{A}(1)^\vee_F}(B_0^F(1));\]
    \item Compute $\text{Ext}^{***}_{\euscr{A}(1)^\vee_F}(B^F_0(k))/(v_1\text{-torsion})$ by induction and \cref{ses bg}.
\end{enumerate}
\begin{remark}
\label{v1 torsion}
    Since we are only concerned with using the $\textup{kq}$-resolution to compute the $v_1$-periodic component of $\pi_{**}^F(\mathbb{S})$, passing to $v_1$-torsion free Ext groups loses no information. Additionally, we will see that there are multiple $v_1$-torsion classes arising in these spectral sequences, and ignoring these classes drastically simplifies our calculations.
\end{remark}

\subsubsection{The algebraic Atiyah--Hirzebruch spectral sequence}
We have that $B_0^F(0) \cong \mathbb{M}_2^F$, and so the first summand of the $\textbf{mASS}^F(\textup{kq} \otimes \textup{kq})$ is given by $\text{Ext}^{***}_{\euscr{A}(1)^\vee_F}(\mathbb{M}_2^F)$.
Notice that since $\textup{H}_{**}(\textup{kq}) \cong (\euscr{A} \modmod \euscr{A}(1))_F^\vee$ by \cref{kq homology}, a standard change of rings argument shows that the $\textup{E}_2$-page of the $\textbf{mASS}^F(\textup{kq})$ is given by
\[\textup{E}_2 = \text{Ext}^{***}_{\euscr{A}^\vee_F}(\textup{H}_{**}(\textup{kq})) \cong \text{Ext}^{***}_{\euscr{A}(1)^\vee_F}(\mathbb{M}_2^F).\]
There are explicit formulae for this algebra in $\mathbb{C}$-motivic \cite[Theorem 6.6]{IS11}, $\mathbb{R}$-motivic \cite[Theorem 4.8]{Hil11}, \cite[Theorem 6.2]{GHIRkoc2}, and $\mathbb{F}_q$-motivic homotopy \cite[Theorems 4.4.2, 4.4.3]{Kylingkqfinite}.
We can use $\text{Ext}^{***}_{\euscr{A}(1)_F^\vee}(\mathbb{M}_2^F)$ as input for an algebraic Atiyah--Hirzebruch spectral sequence $\textbf{aAHSS}(B_0^F(1))$ computing the next summand, $\text{Ext}^{***}_{\euscr{A}(1)^\vee_F}(B_0^F(1))$, which we now outline.

Filter $B_0^F(1) = \mathbb{M}_2^F\{1, \bar{\xi}_1, \bar{\tau}_1\}$ by topological degree. Let $\text{F}_iB_0^F(1)$ denote the subspace spanned by generators of degree $\leq i$. This gives a finite filtration of the form
\begin{equation}
\label{filtration:aAHSS}
\begin{tikzcd}
	0 & {\text{F}_0B_0^F(1)} & {\text{F}_1B_0^F(1)} & {\text{F}_2B_0^F(1)} & {\text{F}_3B_0^F(1) = B_0^F(1)} \\
	& {\mathbb{M}_2^F\{[1]\}} & 0 & {\mathbb{M}_2^F\{[\overline{\xi}_1]\}} & {\mathbb{M}_2^F\{[\overline{\tau}_1]\}}
	\arrow[from=1-1, to=1-2]
	\arrow[from=1-2, to=1-3]
	\arrow[from=1-2, to=2-2]
	\arrow[from=1-3, to=1-4]
	\arrow[from=1-3, to=2-3]
	\arrow[from=1-4, to=1-5]
	\arrow[from=1-4, to=2-4]
	\arrow[from=1-5, to=2-5]
\end{tikzcd}
\end{equation}
We refer to this filtration as the \textit{Atiyah--Hirzebruch filtration}.
The $\textbf{aAHSS}(B_0^F(1))$ comes from the exact couple arising from applying the functor $\text{Ext}^{***}_{\euscr{A}(1)^\vee_F}(-)$ to \eqref{filtration:aAHSS}. This spectral sequence has signature
\[\textup{E}_1^{s,f,w,a} = \text{Ext}^{s,f,w}_{\euscr{A}(1)_F^\vee}(\mathbb{M}_2^F) \otimes \mathbb{M}^F_2\{[1], [\overline{\xi}_1], [\overline{\tau}_1]\}\implies \text{Ext}^{s,f,w}_{\euscr{A}(1)_F^\vee}(B_0^F(1)).\]
The $\textup{E}_1$-page of this spectral sequence consists of 3 copies of $\text{Ext}^{***}_{\euscr{A}(1)^\vee_F}(\mathbb{M}_2^F)$ in Atiyah--Hirzebruch filtration degrees 0, 2, and 3, which we index by $a$. Each Atiyah--Hirzebruch filtration piece is suitably shifted by the degree of the cell it is attached to. To make this precise, we can rewrite the $\textup{E}_1$-page as
\[\textup{E}_1 \cong \text{Ext}^{***}_{\euscr{A}(1)^\vee_F}(\mathbb{M}_2^F)\{[1]\} \oplus \Sigma^{2, 1}\text{Ext}^{***}_{\euscr{A}(1)_F^\vee}(\mathbb{M}_2^F)\{[\overline{\xi}_1]\}\oplus \Sigma^{3, 1}\text{Ext}^{***}_{\euscr{A}(1)^\vee_F}(\mathbb{M}_2^F)\{[\overline{\tau}_1]\}.\]
We give a heuristic for this spectral sequence in \Cref{heuristic b0(1)}.
\begin{figure}[h]
    \begin{center}
    \begin{tikzpicture}
        \node[shape=rectangle,draw=black] (0) at (0,0) {$[1] \otimes \text{Ext}^{***}_{\euscr{A}(1)^\vee}(\mathbb{M}_2)$}; 
        \node[shape = rectangle, draw=black] (2) at (0,2) {$[\overline{\xi}_1] \otimes \text{Ext}^{***}_{\euscr{A}(1)^\vee}(\mathbb{M}_2)$};
        \node[shape = rectangle, draw=black] (3) at (0,3) {$[\overline{\tau}_1]\otimes \text{Ext}^{***}_{\euscr{A}(1)^\vee}(\mathbb{M}_2)$};

        \path (0) edge[thick,bend right = 60, draw=blue] node[right]{} (2);
        \path (2) edge[thick] node[right] {}(3);
    \end{tikzpicture}
    \end{center}
    \caption{A heuristic for the \textbf{aAHSS}$(B_0^F(1))$.}
    \label{heuristic b0(1)}
\end{figure}

If $x \in \text{Ext}^{***}_{\euscr{A}(1)^\vee_F}(\mathbb{M}_2^F)$ has \textbf{aAHSS} degree $(s, f, w,a)$, then $|\Sigma^{p, q}x| = (s+p, f, w+q, a).$
We will let $\alpha[i]$ denote the copy of the class $\alpha \in \text{Ext}^{***}_{\euscr{A}(1)_F^\vee}(\mathbb{M}_2^F)$ in Atiyah--Hirzebruch filtration $i$. With this notation, differentials in the $\textbf{aAHSS}(B_0^F(1))$ are of the form
\[d_r:\text{E}_r^{s,f,w,a} \to \text{E}_r^{s-1, f+1,w, a-r}.\]
Thus the $d_r$ differential lowers Atiyah--Hirzebruch filtration by $r$. In particular, the cell structure of $B_0^F(1)$ ensures that each differential only takes nonzero value on one Atiyah--Hirzebruch filtration piece.

The differentials in this spectral sequence are computed by the cobar complex: to determine the differential on a class $\alpha[k]$, we first lift to a representative in the cobar complex for $\text{F}_kB_0^F(1)$, compute the cobar differential on this representative, then project down to the appropriate filtration. In other words, the differentials witness the attaching maps in the algebraic cell complex for $B_0^F(1)$. 

The $d_1$-differential witnesses the $h_0$-attaching map between the cells $[\overline{\tau}_1]$ and $[\overline{\xi}_1]$, so is nonzero only on Atiyah--Hirzebruch filtration 3 and lands in Atiyah--Hirzebruch filtration 2. This gives the following formula:
\begin{proposition}
\label{aAHSS d1}
    In the $\textup{\textbf{aAHSS}}(B_0^F(1))$, the $d_1$-differential is determined by
    \[d_1(\alpha[3]) = h_0\alpha[2].\]
\end{proposition}
Linearity over $\text{Ext}^{***}_{\euscr{A}(1)^\vee_F}(\mathbb{M}_2^F)$ combined with the above formula gives the behavior of $d_1$. The $d_2$-differential witness the $h_1$-attaching map between the cells $[\overline{\xi}_1]$ and $[1]$, so is nonzero only on Atiyah--Hirzebruch filtration 2 and lands in Atiyah--Hirzebruch filtration 0. This gives the following formula:
\begin{proposition}
    In the $\textup{\textbf{aAHSS}}(B_0^F(1))$, the $d_2$-differential is determined by
    \[d_2(\alpha[2]) = h_1\alpha[0].\]
\end{proposition}
Linearity combined with the above formula gives the behavior of $d_2$. By inspection, the only possible nonzero $d_3$ differential is between Atiyah--Hirzebruch filtrations 3 and 0, and will be given by the Massey product
\[d_3(\alpha[3]) = \langle\alpha, h_0, h_1 \rangle[0].\]
We will see that this differential is dependent on the base field $F$, as is evident from the formula. 

By inspection of the length of \eqref{filtration:aAHSS}, we see the following.
\begin{proposition}
\label{aAHSS convergenece general}
    In the $\textup{\textbf{aAHSS}}(B_0^F(1))$, we have that $\textup{E}_4=\textup{E}_\infty$.
\end{proposition}
The process of solving extension problems on the $\textup{E}_\infty$-page is more involved, so we postpone the discussion until we specialize to particular fields $F$.

\begin{remark}
    Though not tantamount to computing the ring of cooperations, it still provides insight to compute $\text{Ext}^{***}_{\euscr{A}(1)^\vee_F}(B_0^F(1)^{\otimes i})$. Classically, Mahowald is able to express this group in terms of Adams covers of bo and bsp. This is not quite the case here; see \cref{C-b0(1)^i} and \cref{b_0(1) powers proof}. 
\end{remark}

\subsubsection{\textup{Ext} of higher integral Brown--Gitler comodules}
Now, we discuss the second step of our method. While one could analyze the algebraic Atiyah--Hirzebruch filtration on $B_0^F(k)$, the resulting spectral sequence becomes drastically more complicated. We will proceed in a different manner. After finding a formula for $\text{Ext}^{***}_{\euscr{A}(1)_F^\vee}(B_0^F(1))$, one can induct along the short exact sequences of \cref{ses bg} to compute $\text{Ext}^{***}_{\euscr{A}(1)_F^\vee}(B_0^F(k))$. 

When $k$ is even, then there is a short exact sequence
\[0 \to \Sigma^{2k, k}B_0^F(\tfrac{k}{2}) \to B_0^F(k) \to B_1^F(\tfrac{k}{2} -1) \otimes_{\mathbb{M}_2^F} (\euscr{A}(1) \modmod \euscr{A}(0))_F^\vee \to 0.\]
Applying $\text{Ext}^{***}_{\euscr{A}(1)^\vee_F}(-)$ gives a long exact sequence of $\text{Ext}^{***}_{\euscr{A}(1)^\vee_F}(\mathbb{M}_2^F)$-modules. By a change of rings isomorphism, the Ext group of the cokernel may be rewritten:
\[\text{Ext}^{***}_{\euscr{A}(1)^\vee_F}(B_1^F(\tfrac{k}{2}-1) \otimes_{\mathbb{M}_2^F}(\euscr{A}(1)\modmod \euscr{A}(0))^\vee_F) \cong \text{Ext}_{\euscr{A}(0)^\vee_F}^{***}(B_1^F(\tfrac{k}{2}-1)).\]
Since $v_1$ acts trivially over $\euscr{A}(0)^\vee$, we see that, modulo $v_1$-torsion, the connecting homomorphism is trivial. This reduces the computation of $\text{Ext}^{***}_{\euscr{A}(1)^\vee_F}(B_0^F(k))$ to the kernel and cokernel in these short exact sequences in Ext. The kernels are taken care of by the inductive hypothesis, and the cokernels are much simpler after reducing to $\euscr{A}(0)^\vee_F$. This finishes the computation when $k$ is even. 

When $k$ is odd, there is a short exact sequence
\[0 \to \Sigma^{2(k-1), k-1}B_0^F(\tfrac{k-1}{2}) \otimes_{\mathbb{M}_2^F} B_0^F(1) \to B_0^F(k) \to B_1^F(\tfrac{k-1}{2}-1) \otimes_{\mathbb{M}_2^F} (\euscr{A}(1) \modmod \euscr{A}(0))_F^\vee \to 0.\]
In the same way as before, the long exact sequence induced by $\text{Ext}^{***}_{\euscr{A}(1)^\vee_F}(-)$ splits into short exact sequences modulo $v_1$-torsion. The cokernels are again a simple computation, but the kernels require a more delicate analysis. Indeed, to compute $\text{Ext}^{***}_{\euscr{A}(1)^\vee_F}\left(B_0^F(\tfrac{k-1}{2}) \otimes_{\mathbb{M}_2^F} B_0^F(1)\right)$, we use another \textbf{aAHSS}. This spectral sequence comes from applying $\text{Ext}^{***}_{\euscr{A}(1)_F^\vee}\left(B_0^F(\tfrac{k-1}{2}) \otimes_{\mathbb{M}_2^F}-\right)$ to \eqref{filtration:aAHSS}. The $\textup{E}_1$-page is given by
\[\textup{E}_1 = \text{Ext}^{***}_{\euscr{A}(1)^\vee}(B_0^F(\tfrac{k-1}{2})) \otimes_{\mathbb{M}_2^F}\mathbb{M}_2^F\{[1], [\overline{\xi}_1], [\overline{\tau}_1]\}.\]
The induction hypothesis handles the left side of the tensor product, and a careful analysis of the spectral sequence gives us the Ext group in question. This finishes the computation when $k$ is odd.

By \cref{e2 mass}, assembling these results gives the $\textup{E}_2$-page of the $\textbf{mASS}^F(\textup{kq} \otimes \textup{kq})$, modulo $v_1$-torsion. We can now hope to analyze this spectral sequence by comparing with the classical case and base change. In \cite{CQ21}, this was shown to collapse on the $\textup{E}_2$-page for all algebraically closed fields of characteristic 0, such as $\mathbb{C}$. We will show that this is also true over $\mathbb{R}$.

\section{Computations in $\text{SH}(\mathbb{C})$}
\label{section 3}
In this section, we review Culver--Quigley's \cite{CQ21} computation of $\pi_{**}^\mathbb{C}(\textup{kq} \otimes \textup{kq})$. This serves both as an example of the methods described in \cref{section 2} and as a place to recall results which will be useful in the sequel via base change. There are a few places where our arguments differ from the ones given in \cite{CQ21}; we will explicitly indicate when this is the case.
\subsection{Background}
Let $\mathbb{M}_2^\mathbb{C} = \mathbb{F}_2[\tau]$ denote the mod-2 motivic homology of a point \cite[Corollary 6.10]{Voemotiviccohomology}, where $|\tau| = (0,-1).$ Recall that as $\rho=0$ in $\mathbb{M}_2^\mathbb{C}$, the dual Steenrod algebra is an honest Hopf algebra (see \cref{algebroid}). For the reader's convenience, we review its structure.
\begin{thm}[{\cite[Theorem 12.6]{Voereduced}}]
     The dual motivic Steenrod algebra $\euscr{A}^\vee_\mathbb{C} = \pi_{**}^\mathbb{C}(\textup{H}\mathbb{F}_2 \otimes \textup{H}\mathbb{F}_2)$ is the commutative Hopf algebra given by
    \[\euscr{A}^\vee_\mathbb{C} = \mathbb{M}_2^\mathbb{C}[\overline{\xi}_1, \overline{\xi}_2, \hdots, \overline{\tau}_0, \overline{\tau}_1, \hdots]/(\overline{\tau}_i^2 =  \tau \overline{\xi}_{i+1}),\]
    where $|\overline{\xi}_i|=(2^{i+1}-2, 2^i-1)$ and $|\overline{\tau}_i| = (2^{i+1}-1, 2^i-1)$. The unit map is given by inclusion, and the coproduct is determined by the formulae: 
    \[ 
    \Delta(\overline{\xi}_k) = \sum_{i+j=k}\overline{\xi}_i \otimes \overline{\xi}_j^{2^i}, \quad \Delta(\overline{\tau}_i) = \sum_{i+j=k}\overline{\tau}_i \otimes \overline{\xi}_j^{2^i} + 1 \otimes \overline{\tau}_k.\]    
\end{thm}
\begin{remark}
     There is a well understood relationship between $\euscr{A}_\mathbb{C}^\vee$ and the classical dual Steenrod algebra $\euscr{A}^\vee_{\text{cl}}$. Complex Betti realization gives a map $\text{Be}^\mathbb{C}:\euscr{A}^\vee_\mathbb{C} \to \euscr{A}^\vee_{\text{cl}}$ determined by
     \[\text{Be}^\mathbb{C}(\tau) = 1, \quad \text{Be}^\mathbb{C}(\overline{\xi}_i) = \overline{\xi}_i^2, \quad \text{Be}^\mathbb{C}(\overline{\tau}_i) = \overline{\xi}_i\]
     This extends to an isomorphism in Ext \cite{DImASS}:
     \[\text{Ext}^{***}_{\euscr{A}_\mathbb{C}^\vee}(\mathbb{M}_2^\mathbb{C}) \otimes_{\mathbb{M}_2^\mathbb{C}}\mathbb{M}_2^\mathbb{C}[\tau^{-1}] \cong \text{Ext}^{**}_{\euscr{A}_{\text{cl}}^\vee}(\mathbb{F}_2) \otimes \mathbb{F}_2[\tau^{\pm 1}].\]
     This observation has proved very useful in $\mathbb{C}$-motivic computations \cite{Isaksen19,IWX23} and sits at the heart of the theory of synthetic spectra \cite{GheWanXu21, GIKR22, Pst23}. We will use this isomorphism to determine motivic information from classical information.
\end{remark}
In addition to an understanding of the homology $\textup{H}_{**}(\textup{kq}) \cong (\euscr{A} \modmod \euscr{A}(1))_\mathbb{C}^\vee$, we also have a computation of the homotopy groups $\pi_{**}^{\mathbb{C}}(\textup{kq})$ \cite{IS11}.
\begin{thm}[{\cite[Theorem 4.10]{IS11}}] 
\label{C-Ext A(1)}
    The $\textup{\textbf{mASS}}^\mathbb{C}(\textup{kq})$ takes the form
    \[\textup{E}_2^{s,f,w}=\textup{Ext}^{s,f,w}_{\euscr{A}(1)^\vee_\mathbb{C}}(\mathbb{M}_2^\mathbb{C}) = \frac{\mathbb{M}_2^\mathbb{C}[h_0, h_1, a, b]}{(h_0h_1, \tau h_1^3, h_1 a, a^2=h_0^2b)} \implies \pi_{s,w}^\mathbb{C}(\textup{kq}),\]
    where $|h_0| = (0,1,0)$, $|h_1| = (1,1,1)$, $|a| = (4,3,2)$ and $|b|=(8,4,4)$. The spectral sequence collapses at $\textup{E}_2$ with no extensions, giving an isomorphism
    \[ \pi_{**}^\mathbb{C}(\textup{kq} )\cong\frac{\Z_2[\tau, \eta, \alpha, \beta]}{(2\eta, \tau \eta^3, \eta \alpha, \alpha^2=4\beta)}, \]
    where the stem and filtration of the generators match those above.
\end{thm}

Following our discussion in \cref{section 2}, the first step towards computing the $\textup{kq}$-resolution is computing the ring of cooperations $\pi_{**}^\mathbb{C}(\textup{kq} \otimes \textup{kq})$.
Recall from \cref{e2 mass} that the $\textbf{mASS}^{\mathbb{C}}(\textup{kq} \otimes \textup{kq})$ takes the form
\[
\textup{E}_2^{s,f,w} = \bigoplus_{k \geq 0} \text{Ext}_{\euscr{A}(1)^\vee_\mathbb{C}}^{s,f,w}(\Sigma^{4k, 2k}B_0^\mathbb{C}(k)) \implies \pi_{s,w}^\mathbb{C}(\text{kq} \otimes \text{kq}),
\]
where $B_0^\mathbb{C}(k)$ are the $\mathbb{C}$-motivic integral motivic Brown--Gitler comodules.
Therefore, we must compute the trigraded groups $\text{Ext}^{***}_{\euscr{A}(1)^\vee_\mathbb{C}}(B_0^\mathbb{C}(k))$ for $k \geq 0.$ Since $B_0^\mathbb{C}(0) = \mathbb{M}_2^\mathbb{C}$, the first summand is computed by \cref{C-Ext A(1)}. We move then to compute $\text{Ext}^{***}_{\euscr{A}(1)^\vee_\mathbb{C}}(B_0^\mathbb{C}(1))$ by the algebraic Atiyah--Hirzebruch spectral sequence outlined in \cref{aAHSS generic}.

\subsection{$\text{Ext}^{***}_{\euscr{A}(1)^\vee_\mathbb{C}}(B_0^\mathbb{C}(1)^{\otimes i})$}
Recall that the $\textbf{aAHSS}(B_0^\mathbb{C}(1))$ takes the form
\[\textup{E}_1^{s,f,w,a} = \text{Ext}^{s,f,w}_{\euscr{A}(1)_\mathbb{C}^\vee}(\mathbb{M}_2^\mathbb{C}) \otimes_{\mathbb{M}_2^\mathbb{C}} \mathbb{M}^\mathbb{C}_2\{[1], [\overline{\xi}_1], [\overline{\tau}_1]\}\implies \text{Ext}^{s,f,w}_{\euscr{A}(1)_\mathbb{C}^\vee}(B_0^\mathbb{C}(1)),\]
with differentials 
\[
d_r:\text{E}_r^{s,f,w,a} \to \text{E}_r^{s-1, f+1, w,a-r}.
\]
We will let $\alpha[i]$ denote the copy of a class $\alpha \in \text{Ext}^{***}_{\euscr{A}(1)^\vee_\mathbb{C}}(\mathbb{M}_2^\mathbb{C})$ in Atiyah--Hirzebruch filtration $i$. We have discussed in \cref{aAHSS generic} the differentials for this spectral sequence. Namely, the $d_1$-differential is only nonzero on Atiyah--Hirzebruch filtration 3, where 
\[d_1(\alpha[3]) = h_0\alpha[2],\]
and the $d_2$-differential is only nonzero on Atiyah--Hirzebruch filtration 2, where
\[d_2(\alpha[2])=h_1\alpha[0].\]
There is a potential nonzero third differential $d_3:\textup{E}_3^{s,f,w,3} \to \textup{E}_3^{s-1, f+1, w,0}$. However, for degree reasons we have that $\langle \alpha, h_0, h_1 \rangle = 0$ for all remaining classes $\alpha[3] \in \textup{E}_3^{s,f,w,3}$, hence this differential is 0. Thus the spectral sequence collapses on $\textup{E}_3=\textup{E}_\infty$. 

We depict the $\textup{E}_1$-page in \Cref{aAHSS(B_0^C(1)) E_1}, the $\textup{E}_2$-page in \Cref{aAHSS(B_0^C(1)) E_2}, and the $\textup{E}_3=\textup{E}_\infty$-page in \Cref{aAHSS(B_0^C(1)) E_3=E_infty}. Charts are written in $(s,f)$-grading, with motivic weight and Atiyah--Hirzebruch filtration suppressed. A black $\bullet$ represents $\mathbb{M}_2^\mathbb{C}$. A hollow $\circ$ represents $\mathbb{F}_2$. A black vertical line represents multiplication by $h_0$. A diagonal line represents multiplication by $h_1$. Differentials are blue, linear with respect to the $\text{Ext}^{***}_{\euscr{A}(1)_\mathbb{C}^\vee}(\mathbb{M}_2^\mathbb{C})$-module structure, and preserve motivic weight. A dashed blue line represents a differential where either source or target is $\tau$-torsion.
\begin{figure}[h]
\centering
\includegraphics[scale = .75]{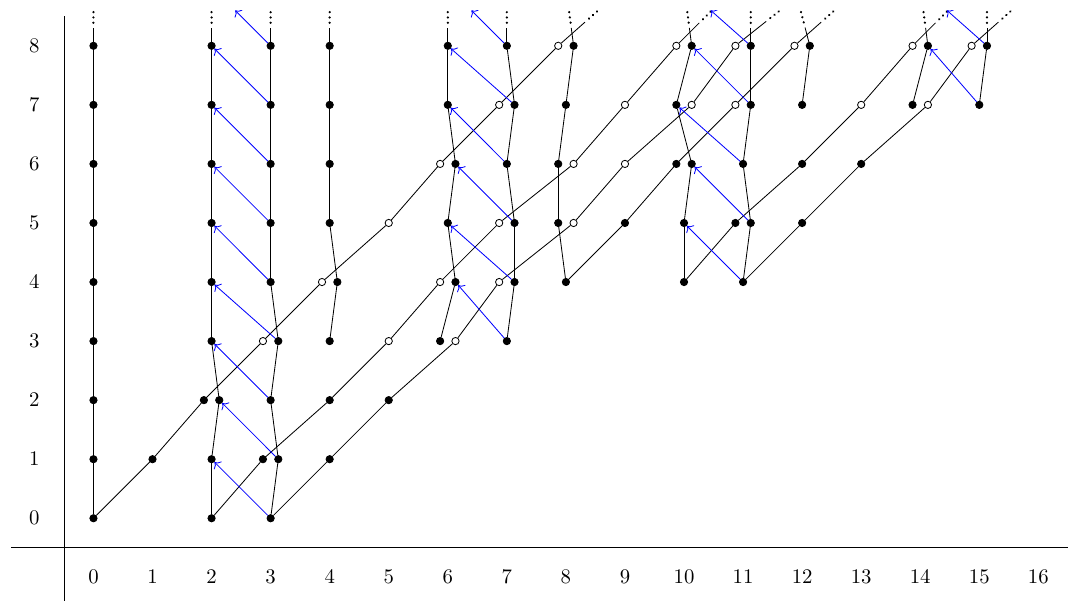}
\caption{The $\textup{E}_1$-page of the $\textbf{aAHSS}(B_0^\mathbb{C}(1))$}
\label{aAHSS(B_0^C(1)) E_1}
\end{figure}

\begin{figure}[h]
    \centering
    \includegraphics[scale=0.75]{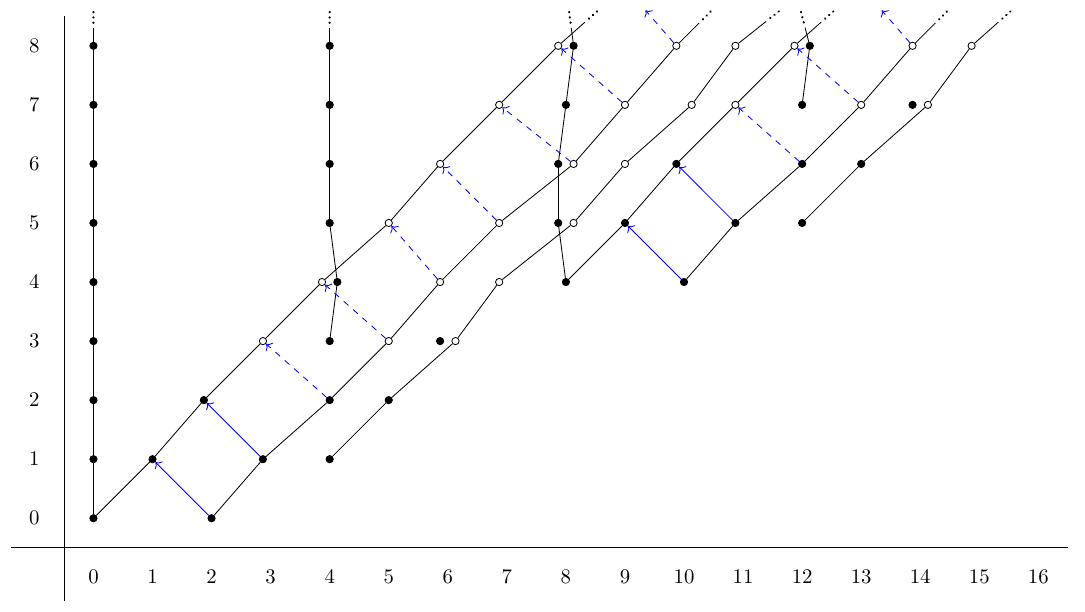}
    \caption{The $\textup{E}_2$-page of the \textbf{aAHSS}($B_0^\mathbb{C}(1)$)}
    \label{aAHSS(B_0^C(1)) E_2}
\end{figure}

\begin{figure}[h]
    \centering
    \includegraphics[scale=0.75]{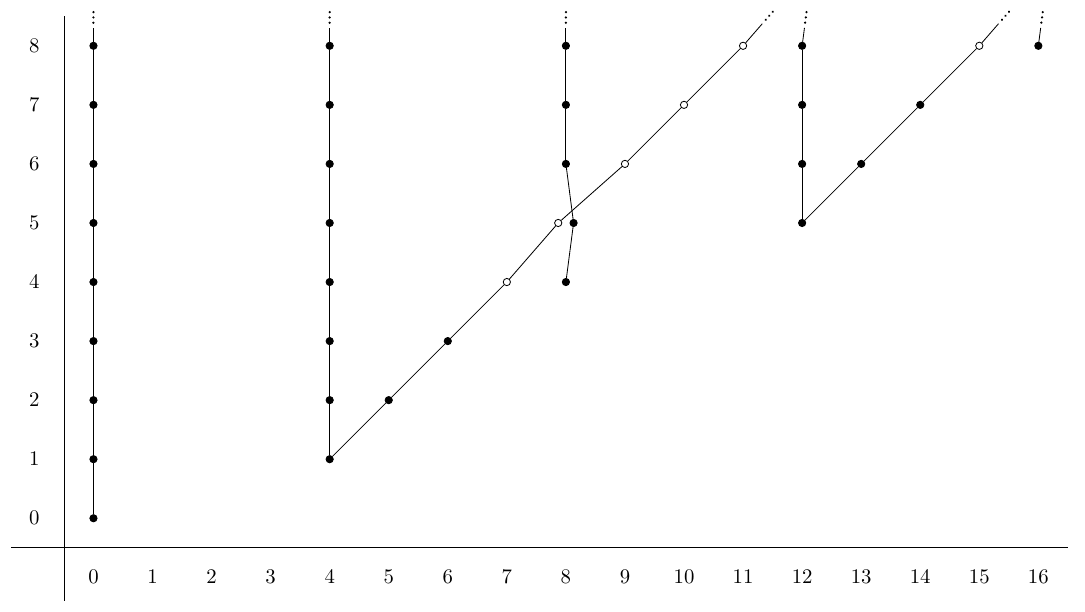}
    \caption{The $\textup{E}_3=\textup{E}_\infty$-page of the \textbf{aAHSS}($B_0^\mathbb{C}(1))$ with hidden extensions.}
    \label{aAHSS(B_0^C(1)) E_3=E_infty}
\end{figure}

The hidden extensions on the $\textup{E}_\infty$-page can be resolved using complex Betti realization. Recall from \cref{2.1} that this is a functor $\text{Be}^\mathbb{C}:\text{SH}(\mathbb{C}) \to \text{Sp}$. For $X \in \text{SH}(\mathbb{C})$, this gives a map in homotopy
\[\pi_{s,w}^\mathbb{C}(X) \to \pi_s(\text{Be}^\mathbb{C}(X)).\]
In particular, there is a map of dual Steenrod algebras $\euscr{A}^\vee_\mathbb{C} \to \euscr{A}_{\text{cl}}^\vee$ sending $\bar{\xi}_i \mapsto \bar{\xi}_i^2$ and $\bar{\tau}_i \mapsto \bar{\xi}_i$. Additionally, there is an isomorphism of $\euscr{A}(1)_\mathbb{C}^\vee$-comodules
\[(B_0^\mathbb{C}(1))[\tau^{-1}] \cong B_0^\text{cl}(1)[\tau^{\pm 1}],\]
where $B_0^\text{cl}(1)$ is the classical Brown--Gitler comodule.
This induces a map of Ext groups
\[\text{Ext}^{s,f,w}_{\euscr{A}(1)^\vee_\mathbb{C}}(B_0^\mathbb{C}(k)) \to \text{Ext}_{\euscr{A}(1)_{\text{cl}}^\vee}^{s,f}(B_0(k)),\]
and allows one to lift hidden extensions from the classical case in \cite{Mah81} to the $\mathbb{C}$-motivic case. 

The following is a motivic analogue of a classical result.
\begin{corollary}[{\cite[Corollary 3.31]{CQ21}}]
\label{C-ksp}
    The $\textup{\textbf{mASS}}^{\mathbb{C}}(\textup{ksp})$ has signature
    \[\textup{E}_2^{s,f,w} = \textup{Ext}^{s,f,w}_{\euscr{A}(1)_\mathbb{C}^\vee}(B_0^\C(1)) \implies \pi^\C_{s,w}(\textup{ksp})\]
    and collapses on the $\textup{E}_2$-page.
\end{corollary}

\begin{proof}
     By \cref{ksp hz1}, there is an equivalence of motivic spectra
     \[\textup{ksp} \cong \textup{H}\mathbb{Z}^\mathbb{C}_1 \otimes \textup{kq},\]
     where $\textup{H}\mathbb{Z}_1^\mathbb{C}$ is the motivic spectrum constructed in \cref{motivic bg}. The K\"unneth spectral sequence for mod-2 homology collapses for $\textup{H}\mathbb{Z}_1^\mathbb{C} \otimes \textup{kq}$, giving an isomorphism
     \[\textup{H}_{**}(\text{ksp}) \cong B_0^\mathbb{C}(1) \otimes_{\mathbb{M}_2^\mathbb{C}}(\euscr{A} \modmod \euscr{A}(1))_\mathbb{C}^\vee.\]
     A change of rings isomorphism gives the $\textup{E}_2$-page for the $\textbf{mASS}^{\mathbb{C}}(\text{ksp})$.
     This Ext group was computed by the $\textbf{aAHSS}(B_0^\mathbb{C}(1))$ and is depicted in \Cref{aAHSS(B_0^C(1)) E_3=E_infty}. There are no differentials for degree reasons, so the spectral sequence collapses and gives the result.
\end{proof}
Classically, one can express the groups $\text{Ext}^{**}_{\euscr{A}(1)_{\text{cl}}^\vee}(B_0(1)^{\otimes i})$ in terms of Adams covers of bo and bsp. This does not translate well to motivic homotopy theory. Over $\mathbb{C}$, Culver--Quigley \cite{CQ21} show that one cannot express the groups $\text{Ext}^{***}_{\euscr{A}(1)^\vee_\mathbb{C}}(B_0^\mathbb{C}(1)^{\otimes i})$ in terms of Adams covers of kq and ksp. We show a similar result over $\mathbb{R}$ in \cref{R-aAHSS section}. Instead, there is a periodic family of groups which will determine $\text{Ext}^{***}_{\euscr{A}(1)^\vee_\mathbb{C}}(B_0^\mathbb{C}(1))$ composed of Adams covers of bo, bsp, kq and ksp.

In what follows, we will use the notation defined in \cref{notation section}. For $i \geq 0$, let $Z_i^\mathbb{C}$ be the trigraded group defined as follows:
\begin{itemize}
    \item When $i \equiv 0 \, (4)$, let
    \[Z^\mathbb{C}_i = \bigoplus_{j=0}^{i/2-1}\Sigma^{4j, 2j}\mathbb{M}_2^\C[h_0] \oplus \Sigma^{2i, i}\text{Ext}^{***}_{\euscr{A}(1)_\C^\vee}(\mathbb{M}_2^\C)\]
    \item When $i \equiv 1 \, (4)$, let
    \[Z^\mathbb{C}_i = \bigoplus_{j=0}^{(i-1)/2-1}\Sigma^{4j, 2j}\mathbb{M}_2^\C[h_0] \oplus \Sigma^{2i-1, i-1}\text{Ext}^{***}_{\euscr{A}(1)_\C^\vee}(B_0^\C(1));\]
    \item When $i \equiv 2 \, (4)$, let
    \[Z^\mathbb{C}_i = \bigoplus_{j=0}^{i/2-1}\Sigma^{4j, 2j}\mathbb{M}_2^\C[h_0]  \oplus\Sigma^{2i-2, i}\mathbb{M}_2^\C \oplus \Sigma^{2i, i}\text{Ext}^{***}_{\euscr{A}(1)_\C^\vee}(B_0^\C(1))\langle 1\rangle;\]
    \item When $i \equiv 3 \, (4)$, let
    \[Z^\mathbb{C}_i = \bigoplus_{j=0}^{(i-1)/2}\Sigma^{4j, 2j}\mathbb{M}_2^\C[h_0] \oplus \Sigma^{2i-1, i}\mathbb{M}_2^\C[h_1]/(h_1^2) \oplus \Sigma^{2i+2, i+1}\text{Ext}^{***}_{\euscr{A}(1)_\C^\vee}( B_0^\C(1))\langle 2\rangle.\]
\end{itemize}

Recall that we are only interested in information up to $v_1$-torsion (see \cref{v1 torsion}).
\begin{proposition}[{\cite[Lemma 3.36]{CQ21}}]
\label{C-b0(1)^i}
    There is an isomorphism of $\euscr{A}(1)^\vee_\mathbb{C}$-comodules and $\textup{Ext}^{***}_{\euscr{A}(1)^\vee_\mathbb{C}}(\mathbb{M}_2^\mathbb{C})$-modules:
    \[\frac{\textup{Ext}^{***}_{\euscr{A}(1)^\vee_\mathbb{C}}(B_0^\mathbb{C}(1)^{\otimes i})}{v_1\textup{-torsion}} \cong Z_i^\mathbb{C}.\]
\end{proposition}

\begin{proof}
    By applying the functor $\text{Ext}^{***}_{\euscr{A}^\vee_\mathbb{C}}\left(B_0^\mathbb{C}(1)^{\otimes i} \otimes_{\mathbb{M}_2^\mathbb{C}} -\right)$ to \eqref{filtration:aAHSS}, we may inductively compute these Ext groups by an \textbf{aAHSS}, where the base case was computed in \cref{C-ksp}. This spectral sequence takes the form
    \[\textup{E}_1=\text{Ext}^{***}_{\euscr{A}(1)^\vee_\mathbb{C}}(B_0^\mathbb{C}(1)^{\otimes i}) \otimes_{\mathbb{M}_2^\mathbb{C}}\mathbb{M}_2^\mathbb{C}\{[1], ]\bar{\xi}_1], [\bar{\tau}_1]\} \implies\text{Ext}^{***}_{\euscr{A}(1)_\mathbb{C}^\vee}(B_0^\mathbb{C}(1)^{\otimes i+1}).\]
    The differentials take the same form as those in $\textbf{aAHSS}(B_0^\mathbb{C}(1))$, namely we have $d_1(\alpha[3])=h_0\alpha[2]$ and $d_2(\alpha[2]) = h_1\alpha[0]$, where $\alpha \in \text{Ext}^{***}_{\euscr{A}(1)^\vee_\mathbb{C}}(B_0^\mathbb{C}(1)^{\otimes i})$. For degree reasons, there is no $d_3$-differential. Hidden extensions may be obtained by comparison with the classical case, using the algebra map
    \[\text{Ext}^{***}_{\euscr{A}(1)^\vee_\mathbb{C}}(B_0^\mathbb{C}(1)^{\otimes i}) \to \text{Ext}^{**}_{\euscr{A}(1)^\vee}(B_0(1)^{\otimes i}).\]
    This resolves all hidden extensions and concludes the proof.
\end{proof}

\subsection{$\text{Ext}^{***}_{\euscr{A}(1)^\vee_\mathbb{C}}(B_0^\mathbb{C}(k))$ and the ring of cooperations}
Next, we compute $\text{Ext}^{***}_{\euscr{A}(1)^\vee_\mathbb{C}}(B_0^\mathbb{C}(k))$ for all $k \geq 1$. 

\begin{thm}[Corrected from {\cite[Theorem 3.38]{CQ21}}]
\label{C-b0(k) ext}
    There is an isomorphism of $\euscr{A}(1)^\vee_\mathbb{C}$-comodules and $\textup{Ext}^{***}_{\euscr{A}(1)^\vee_\mathbb{C}}(\mathbb{M}_2^\mathbb{C})$-modules:
    \[\frac{\textup{Ext}^{***}_{\euscr{A}(1)^\vee_\mathbb{C}}(B_0^\mathbb{C}(k))}{v_1\textup{-torsion}} \cong \Sigma^{4k-4, 2k-2}Z^\mathbb{C}_{\alpha(k)} \oplus \bigoplus_{j=0}^{k-2}\Sigma^{4j, 2j}\mathbb{M}_2^\mathbb{C}[h_0],\]
    where $\alpha(k)$ is the number of 1's in the dyadic expansion of $k$.
\end{thm}

\begin{proof}
    For notational convenience, all Ext groups in this proof will be implicitly computed modulo $v_1$-torsion. We will induce on $k$. The case of $k=1$ was shown in \cref{C-ksp}. Now, suppose the theorem is true for all $i<k$.

    Suppose that $k$ is even. \cref{ses bg} gives a short exact sequence of $\euscr{A}(1)_\mathbb{C}^\vee$-comodules
    \begin{align}
    \label{ses bg C even}
    0 \to \Sigma^{2k, k}B_0^\mathbb{C}(\tfrac{k}{2}) \to B_0^\mathbb{C}(k) \to B_1^\mathbb{C}(\tfrac{k}{2} -1) \otimes_{\mathbb{M}_2^\mathbb{C}} (\euscr{A}(1) \modmod \euscr{A}(0))_\mathbb{C}^\vee \to 0.
    \end{align} 
    Applying the functor $\text{Ext}^{***}_{\euscr{A}(1)^\vee_\mathbb{C}}(-)$ gives a long exact sequence of $\text{Ext}^{***}_{\euscr{A}(1)^\vee_\mathbb{C}}(\mathbb{M}_2^\mathbb{C})$-modules. Moreover, after killing $v_1$-torsion the connecting homomorphism is trivial, giving short exact sequences whose middle term is $\text{Ext}^{***}_{\euscr{A}(1)^\vee_\mathbb{C}}(B_0^\mathbb{C}(k))$. Therefore, this Ext group decomposes into the Ext groups of the kernel and cokernel of the original short exact sequence (\ref{ses bg C even}). The kernel is handled by the inductive hypothesis: we have
    \[\text{Ext}^{***}_{\euscr{A}(1)^\vee_\mathbb{C}}(\Sigma^{2k, k}B_0^\mathbb{C}(\tfrac{k}{2})) \cong \Sigma^{2k, k}\left(\Sigma^{2k-4, k-2}Z_{\mathbb{\alpha}(k/2)}^\mathbb{C} \oplus \bigoplus_{j=0}^{2k-8}\Sigma^{4j, 2j}\mathbb{M}_2^\mathbb{C}[h_0]\right) .\]
    We can use a change of rings isomorphism to rewrite the cokernel as
    \[\text{Ext}^{***}_{\euscr{A}(1)^\vee_\mathbb{C}}\left(B_1^\mathbb{C}(\tfrac{k}{2}-1) \otimes_{\mathbb{M}_2^\mathbb{C}}(\euscr{A}(1) \modmod \euscr{A}(0))_\mathbb{C}^\vee\right) \cong \text{Ext}^{***}_{\euscr{A}(0)^\vee_\mathbb{C}}(B_1^\mathbb{C}(\tfrac{k}{2}-1)).\]
    Since $\euscr{A}(0)^\vee_\mathbb{C}\cong \mathbb{M}_2^\mathbb{C}[\bar{\tau}_0]/(\bar{\tau}_0^2)$ is exterior, the Ext groups in question are polynomial. Modulo $v_1$-torsion, we are only left with $h_0$-towers:
    \[\text{Ext}^{***}_{\euscr{A}(0)^\vee_\mathbb{C}}(B_1^\mathbb{C}(\tfrac{k}{2}-1)) \cong \bigoplus_{j=0}^{k/2-1}\Sigma^{4j, 2j}\mathbb{M}_2^\mathbb{C}[h_0].\]
    Putting these pieces together with some reindexing on the $h_0$-towers and using the fact that $\alpha(k/2)=\alpha(k)$ gives the result.
    
    Suppose now that $k$ is odd. Again, \cref{ses bg} gives a short exact sequence of $\euscr{A}(1)^\vee_\mathbb{C}$-comodules
    \begin{align}
    \label{ses BG C odd}
    0 \to \Sigma^{2(k-1), k-1}B_0^\mathbb{C}(\tfrac{k-1}{2}) \otimes_{\mathbb{M}_2^\mathbb{C}} B_0^\mathbb{C}(1) \to B_0^\mathbb{C}(k) \to B_1^\mathbb{C}(\tfrac{k-1}{2}-1) \otimes_{\mathbb{M}_2^\mathbb{C}} (\euscr{A}(1) \modmod \euscr{A}(0))_\mathbb{C}^\vee \to 0.
    \end{align}
    Applying the functor $\text{Ext}^{***}_{\euscr{A}(1)^\vee_\mathbb{C}}(-)$ gives a long exact sequence of $\text{Ext}^{***}_{\euscr{A}(1)^\vee_\mathbb{C}}(\mathbb{M}_2^\mathbb{C})$-modules, and as in the even case, the connecting homomorphism is trivial modulo $v_1$-torsion. Thus, the Ext group we are interested in decomposes into the Ext groups of the kernel and cokernel of \eqref{ses BG C odd}. In this case, the cokernel is simpler to compute. A change of rings isomorphism gives
    \[\text{Ext}^{***}_{\euscr{A}(1)^\vee_\mathbb{C}}\left(B_1^\mathbb{C}(\tfrac{k-1}{2}-1) \otimes_{\mathbb{M}_2^\mathbb{C}}(\euscr{A}(1) \modmod \euscr{A}(0))_\mathbb{C}^\vee\right) \cong \text{Ext}^{***}_{\euscr{A}(0)_\mathbb{C}^\vee}(B_1^\mathbb{C}(\tfrac{k-1}{2}-1)).\]
    By the same argument as before, modulo $v_1$-torsion this Ext group consists of purely $h_0$-towers:
    \[\text{Ext}^{***}_{\euscr{A}(0)_\mathbb{C}^\vee}(B_1^\mathbb{C}(\tfrac{k-1}{2}-1)) \cong \bigoplus_{j = 0}^{(k-1)/2-1}\Sigma^{4j, 2j}\mathbb{M}_2^\mathbb{C}[h_0].\]
    The Ext group of the kernel may be computed using an algebraic Atiyah--Hirzebruch spectral sequence $\textbf{aAHSS}\left(B_0^\mathbb{C}(\tfrac{k-1}{2}) \otimes_{\mathbb{M}_2^\mathbb{C}} B_0^\mathbb{C}(1)\right)$. This has signature
    \[\textup{E}_1=\text{Ext}^{***}_{\euscr{A}(1)^\vee_\mathbb{C}}(B_0^\mathbb{C}(\tfrac{k-1}{2})) \otimes_{\mathbb{M}_2^\mathbb{C}} \mathbb{M}_2^\mathbb{C}\{[1], [\bar{\xi}_1], [\bar{\tau}_1]\} \implies \text{Ext}^{***}_{\euscr{A}(1)^\vee_\mathbb{C}}\left(B_0^\mathbb{C}(\tfrac{k-1}{2}) \otimes_{\mathbb{M}_2^\mathbb{C}}B_0^\mathbb{C}(1)\right).\]
    By induction, we may rewrite $\text{Ext}^{***}_{\euscr{A}(1)^\vee_\mathbb{C}}(B_0^\mathbb{C}(\tfrac{k-1}{2}))$ as
    \[\text{Ext}^{***}_{\euscr{A}(1)^\vee_\mathbb{C}}(B_0^\mathbb{C}(\tfrac{k-1}{2})) \cong \Sigma^{2k-6, k-3}Z_{\alpha(k)-1}^\mathbb{C} \oplus \bigoplus_{j=0}^{2k-10}\Sigma^{4j, 2j}\mathbb{M}_2^\mathbb{C}[h_0],\]
    using that $\alpha(\tfrac{k-1}{2}) = \alpha(k)-1.$
    This splitting leads to a splitting of the $\textup{E}_1$-page of the algebraic Atiyah--Hirzebruch spectral sequence, and we may analyze each summand individually. On the left-hand summand, \cref{C-b0(1)^i} implies that result of the spectral sequence is isomorphic to $\Sigma^{2k-2, k-1}Z^\mathbb{C}_{\alpha(k)}$. On the right-hand summand, the spectral sequence collapses on the $\textup{E}_1$-page, giving back the original summand modulo $v_1$-torsion. Thus, after reindexing the $h_0$-towers, the Ext groups of the kernel and cokernel and right-hand side of the Ext groups assemble to give the result. 
\end{proof}

\begin{remark}
    The above proof is where our argument differs most from the one given in \cite{CQ21}. We explain these differences here. 
    
    The results of Culver--Quigley assert that over $\mathbb{C}$ there is an isomorphism
    \[\frac{\textup{Ext}^{***}_{\euscr{A}(1)^\vee_\mathbb{C}}(B_0^\mathbb{C}(k))}{v_1\text{-torsion}} \cong Z^\mathbb{C}_{2k-\alpha(k)}.\]
    However, this is not quite correct. Let us illustrate this difference in the case of $B_0^\mathbb{C}(2)$. 
    
    By \Cref{ses bg}, there is a short exact sequence of $\euscr{A}(1)^\vee_\mathbb{C}$-comodules
    \begin{equation}
    \label{SES:rmk_fixing_CQ}
        0 \to \Sigma^{4,2}B_0^{\mathbb{C}}(1) \to B_0^{\mathbb{C}}(2) \to  B_1^\mathbb{C}(0) \otimes (\euscr{A}(1)\modmod \euscr{A}(0))^\vee_\mathbb{C} \to 0.
    \end{equation}
    Applying the functor $\text{Ext}_{\euscr{A}(1)^\vee_\mathbb{C}}^{***}(-)$ gives a long exact sequence of $\text{Ext}_{\euscr{A}(1)^\vee_\mathbb{C}}^{***}(\mathbb{M}_2^\mathbb{C})$-modules. The Ext group of the kernel of \eqref{SES:rmk_fixing_CQ} has already been explicitly determined (see \Cref{aAHSS(B_0^C(1)) E_3=E_infty}). The Ext group of the cokernel may be written using a change of rings isomorphism as
    \[\text{Ext}^{***}_{\euscr{A}(1)^\vee_{\mathbb{C}}}(B_1^\mathbb{C}(0) \otimes (\euscr{A}(1) \modmod \euscr{A}(0))^\vee_\mathbb{C}) \cong \text{Ext}^{***}_{\euscr{A}(0)^\vee_\mathbb{C}}(B_1^\mathbb{C}(0)) \cong \mathbb{M}_2^\mathbb{C}[h_0].\]
    For degree reasons, the connecting homomorphism
    \[\text{Ext}^{s,f,w}_{\euscr{A}(0)^\vee_\mathbb{C}}(B_1^\mathbb{C}(0)) \xrightarrow{\delta} \text{Ext}^{s-1,f,w}_{\euscr{A}(1)^\vee_\mathbb{C}}(B_0^\mathbb{C}(1))\]
    is trivial, implying that the long exact sequence in Ext decomposes into short exact sequences. Thus, we have an isomorphism
    \[\text{Ext}_{\euscr{A}(1)^\vee}^{***}(B_0^\mathbb{C}(2)) \cong \Sigma^{4,2}\text{Ext}^{***}_{\euscr{A}(1)^\vee_\mathbb{C}}(B_0^\mathbb{C}(1)) \oplus \mathbb{M}_2^\mathbb{C}[h_0] \cong \Sigma^{4,2}Z_1^\mathbb{C} \oplus \mathbb{M}_2^\mathbb{C}[h_0],\]
    which agrees with \Cref{C-b0(k) ext} as $\alpha(2)=1$. We depict this group in \Cref{fig:ext_B02_C}.

    \begin{figure}[H]
        \centering
        \includegraphics[scale=.75]{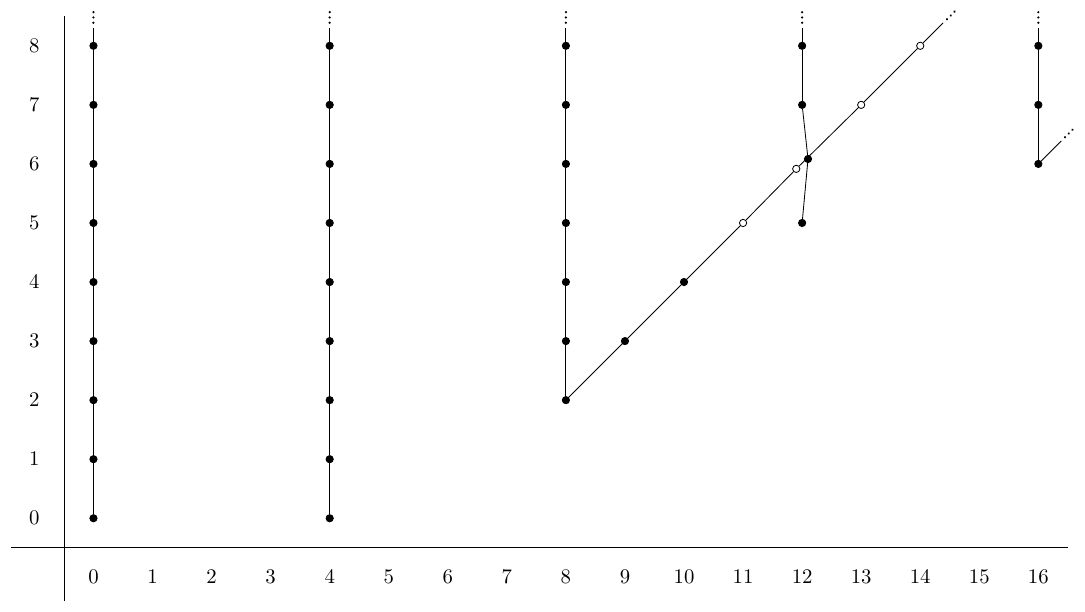}
        \caption{The group $\text{Ext}_{\euscr{A}(1)^\vee_\mathbb{C}}^{***}(B_0^\mathbb{C}(2))$.}
        \label{fig:ext_B02_C}
    \end{figure}    
    
    However, Culver--Quigley claim that $\text{Ext}^{***}_{\euscr{A}(1)^\vee_\mathbb{C}}(B_0^\mathbb{C}(2))$ is isomorphic to
    \[Z_3^\mathbb{C} = \mathbb{M}_2^\mathbb{C}[h_0] \oplus\Sigma^{4, 2}\mathbb{M}_2^\mathbb{C}[h_0] \oplus \Sigma^{5, 3}\mathbb{M}_2^\mathbb{C}[h_1]/(h_1^2) \oplus \Sigma^{8, 4}\text{Ext}^{***}_{\euscr{A}(1)^\vee_\mathbb{C}}(B_0^\mathbb{C}(1))\langle 2 \rangle.\]
    There is a class in stem 5 which supports an $h_1$-multiplication. We depict this group in \Cref{fig:C_Z3}.

    \begin{figure}[H]
        \centering
        \includegraphics[scale=.75]{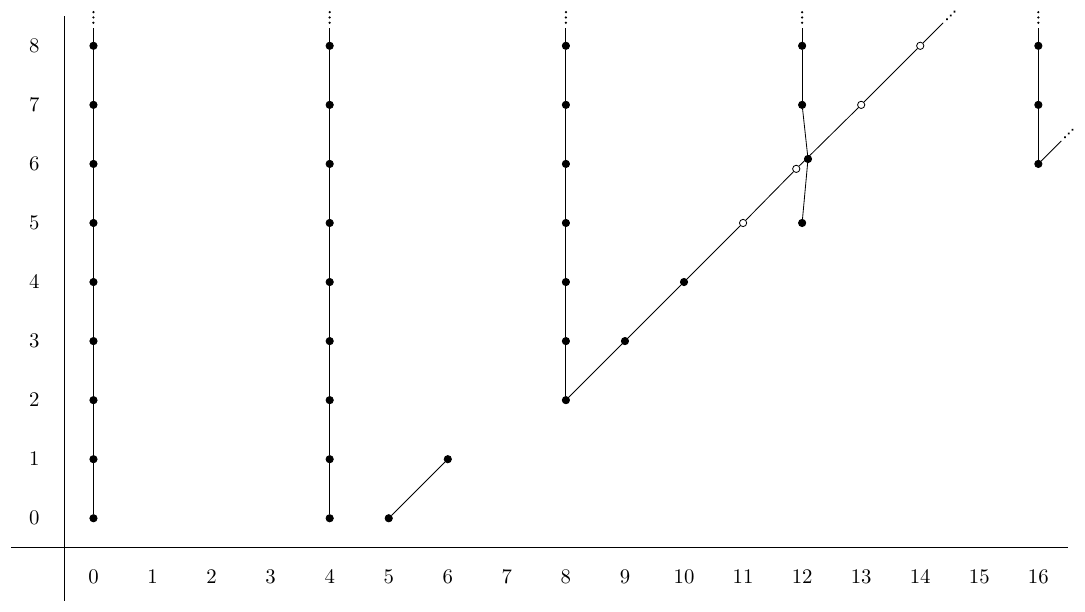}
        \caption{The group $Z_3^\mathbb{C}$.}
        \label{fig:C_Z3}
    \end{figure}
    
    This type of difference extends to all values of $k \geq 2$, showing that the formula given in \cite[Theorem 3.38]{CQ21} is incorrect. 
    However, this observation is actually in line with the rest of the work done in Culver--Quigley and doesn't affect any end results. More precisely:
    \begin{enumerate}
        \item[(1)] It was shown in \cite[Lemma 3.38]{CQ21} that the groups $\text{Ext}^{***}_{\euscr{A}(1)^\vee_\mathbb{C}}(B_0^\mathbb{C}(1)^{\otimes i})$ are not expressible in terms of Adams covers of kq and ksp. The purpose of the groups $Z_i^\mathbb{C}$ is exactly to account for this failure, mixing in an appropriate trigraded version of $\text{Ext}^{**}_{\euscr{A}(1)_\text{cl}^\vee}(\text{H}_{*}(\text{bo}^{\langle i \rangle}))$ and $\text{Ext}^{**}_{\euscr{A}(1)_\text{cl}^\vee}(\text{H}_{*}(\text{bsp}^{\langle i \rangle})).$ The difference being observed in the computation of $\text{Ext}^{***}_{\euscr{A}(1)^\vee_\mathbb{C}}(B_0^\mathbb{C}(k))$ can be accounted for by noting that this group \textit{also} cannot be expressed using just $Z_{2k-\alpha(k)}^\mathbb{C}$. Rather, we must express these groups using shifts of $Z_{\alpha(k)}^\mathbb{C}$ and an appropriate trigraded version of the $h_0$-towers found in the classical groups $\text{Ext}^{**}_{\euscr{A}(1)^\vee_\text{cl}}(\mathbb{F}_2)$; this last term is precisely what is given by $\text{Ext}^{***}_{\euscr{A}(0)^\vee_\mathbb{C}}(B_1^\mathbb{C}(\tfrac{k-1}{2}))$ in the case that $k \equiv 0 \, (2),$ and by $\text{Ext}^{***}_{\euscr{A}(0)^\vee_\mathbb{C}}(B_1^\mathbb{C}(\tfrac{k-1}{2}-1))$ in the case that $k \equiv 1 \, (2).$
        \item[(2)] Apart from this discrepancy, the contents of \cite{CQ21} are unaffected. In particular, we will see that the $\textbf{mASS}^\mathbb{C}(\text{kq} \otimes \text{kq})$ collapses on the $\textup{E}_2$-page by using the same arguments as employed by Culver--Quigley.
    \end{enumerate}
\end{remark}

The following is now immediate from \cref{e2 mass}.
\begin{proposition}[{\cite[Proposition 3.41]{CQ21}}]
    The $\textup{E}_2$-page of the $\textup{\textbf{mASS}}^\mathbb{C}(\textup{kq} \otimes \textup{kq})$ is given, modulo $v_1$-torsion, by
    \[\textup{E}_2 =\bigoplus_{k \geq 0}\textup{Ext}^{s,f,w}_{\euscr{A}(1)_\mathbb{C}^\vee}(\Sigma^{4k, 2k}B_0^\mathbb{C}(k)) \cong\bigoplus_{k \geq 0} \left(\Sigma^{4k, 2i}Z^\mathbb{C}_{\alpha(k)} \oplus \bigoplus_{j=0}^{4k-8}\Sigma^{4j, 2j}\mathbb{M}_2^\mathbb{C}[h_0]\right).\]
\end{proposition}

Recall from \cite[Theorem 2.9]{Mah81} that the \textbf{ASS}$(\text{bo} \otimes \text{bo})$ collapses on $\textup{E}_2$.
\begin{thm}[{\cite[Cor 3.43]{CQ21}}]
\label{C-coop mass collapse}
    The \textup{\textbf{mASS}}$^\mathbb{C}(\textup{kq} \otimes \textup{kq})$ collapses on the $\textup{E}_2$-page.
\end{thm}

\begin{proof}
    Betti realization sends the \textbf{mASS}($\textup{kq} \otimes \textup{kq})$ to the \textbf{ASS}($\text{bo} \otimes \text{bo}$). At the level of $\textup{E}_2$-pages, this is a map 
    \[\text{Ext}^{s,f,w}_{\euscr{A}(1)^\vee_\mathbb{C}}(\textup{H}_{**}(\textup{kq} \otimes \textup{kq})) \to \text{Ext}_{\euscr{A}(1)_{\text{cl}}^\vee}^{s,f}(\text{H}_*(\text{bo} \otimes \text{bo})),\]
    obtained by inverting $\tau$ on the $\textup{E}_2$-page of the \textbf{mASS}($\textup{kq} \otimes \textup{kq}$) and setting $\tau=1$. Since the $\textbf{ASS}(\text{bo} \otimes \text{bo})$ collapses on $\textup{E}_2$, we must have that there are no motivic differentials where both source and target are $\tau$-torsion free. However, our description of the $\textup{E}_2$-page of the $\textbf{mASS}(\textup{kq} \otimes \textup{kq})$ shows that any $\tau$-torsion class must be $h_1$-torsion free, hence is some $h_1$-tower. For degree reasons, there can be no such differentials between these towers, concluding the proof.
\end{proof}

\newpage
\part{Computations in $\mathrm{SH}(\mathbb{R})$}
\label{part 2}
\section{Algebraic preliminaries and comparisons}
\label{section 4}
In this section, we recall the $\mathbb{R}$-motivic dual Steenrod algebra $\euscr{A}^\vee_\mathbb{R}$ and introduce notation for modules to be used throughout. Then, we give a brief overview of the relationship between $\mathbb{C}$-motivic and $\mathbb{R}$-motivic homotopy and compare the respective kq-resolutions.

\subsection{Background}
\label{R-background and notation}
Let $\mathbb{M}_2^\mathbb{R} = \mathbb{F}_2[\rho, \tau]$ denote the mod-2 motivic homology of a point \cite[Corollary 6.10]{Voemotiviccohomology}, where $|\tau|=(0,-1)$ and $|\rho| = (-1, -1)$. Recall from \cref{dual motivic steenrod} that the dual motivic Steenrod algebra $\euscr{A}^\vee_\mathbb{R}$ takes the form
\[\euscr{A}^\vee_\mathbb{R}=\mathbb{M}_2^\mathbb{R}[\overline{\xi}_1, \overline{\xi}_2, \hdots, \overline{\tau}_0, \overline{\tau}_1, \hdots]/(\overline{\tau}_i^2 = \rho\overline{\tau}_{i+1} + \rho\overline{\tau}_0\overline{\xi}_{i+1} + \tau \overline{\xi}_{i+1}).\]
\cref{e2 mass} asserts that the $\textbf{mASS}^\mathbb{R}(\textup{kq} \otimes \textup{kq})$ takes the form
\[\textup{E}_2^{s,f,w} = \bigoplus_{k \geq 0} \text{Ext}_{\euscr{A}(1)^\vee_\mathbb{R}}^{s,f,w}(\Sigma^{4k, 2k}B_0^\mathbb{R}(k)) \implies \pi_{s,w}^\mathbb{R}(\text{kq} \otimes \text{kq}),\]
where $B_0^\mathbb{R}(k)$ are the $\mathbb{R}$-motivic integral motivic Brown--Gitler comodules.
Therefore, to compute this $\textup{E}_2$-page we must compute the trigraded groups $\text{Ext}^{***}_{\euscr{A}(1)^\vee_\mathbb{R}}(B_0^\mathbb{R}(k))$ for $k \geq 0.$ Since $B_0^\mathbb{R}(0) = \mathbb{M}_2^\mathbb{R}$, the first summand is isomorphic to the $\textup{E}_2$-page of the $\textbf{mASS}^{\mathbb{R}}(\textup{kq})$. This was computed by Hill in \cite[Theorem 5.6]{Hil11}. We use a presentation adopted by Guillou--Hill--Isaksen--Ravenel \cite[Theorem 6.2]{GHIRkoc2}.

\begin{thm}
    [{\cite{Hil11, GHIRkoc2}}]
    \label{thm:Ext_A1_R}
    The \textup{\textbf{mASS}$^\mathbb{R}(\textup{kq})$} takes the form
    \[ \textup{E}_2^{s,f,w} = \textup{Ext}^{***}_{\euscr{A}(1)_\mathbb{R}^\vee}(\mathbb{M}_2^\mathbb{R}, \mathbb{M}_2^\mathbb{R})= \mathbb{F}_2[\tau^4, \rho,h_0, h_1, a, b,  \tau^2 h_0, \tau h_1, b, \tau^2 a]/I \implies \pi_{s,w}^\mathbb{R}(\textup{kq}).\]
    The spectral sequence collapses at $\textup{E}_2$. The degrees of the generators are listed in \Cref{GHIRgenerators} and the ideal of relations $I$ is described in \Cref{GHIRrelations}.
\end{thm}
We organize \Cref{GHIRgenerators} and \Cref{GHIRrelations} by coweight modulo 4.

\begin{table}[H]
    \centering
    \setlength{\tabcolsep}{0.5em} 
    {\renewcommand{\arraystretch}{1.2}
    \begin{tabular}{|l|l|l|}
        \hline
        Generator & $(s,f,w)$ & Coweight \\
        \hline 
        \hline
        $\rho$ & $(-1, 0, -1)$ &0\\
        $h_0$ &$(0,1,0)$ &0\\
        $h_1$ & $(1,1,1)$&0\\
        $\tau^2 a$ & $(4,3,0)$& 4\\
        $\tau^4$ &$(0,0,-4)$ & 4\\
        $b$ & $(8,4,4)$&4\\
        \hline
        \hline
        $\tau h_1$ &$(1,1,0)$ &1\\
        \hline
        \hline
        $ \tau^2 h_0$ & $(0,1,-2)$&2\\
        $a$ & $(4,3,2)$& 2 \\
        \hline
    \end{tabular}}
    \caption{Generators for $\text{Ext}^{***}_{\euscr{A}(1)_\mathbb{R}^\vee}(\mathbb{M}_2^\mathbb{R})$.}      
    \label{GHIRgenerators}
\end{table}

\begin{table}[H]
    \centering
    \setlength{\tabcolsep}{0.5em} 
    {\renewcommand{\arraystretch}{1.2}
    \begin{tabular}{|l|l|l|}
        \hline
        Relation & $(s,f,w)$ &Coweight \\
        \hline
        \hline
        $\rho h_0$ & $(-1,1,-1)$ & 0\\
        $h_0h_1$ & $(1,2,1)$ & 0\\
        $(\tau^2h_0)^2+\tau^4 h_0^2$ & $(0,2,-4)$ & 4\\
        $\tau^4  h_1^3 + \rho \cdot \tau^2 a$ & $(3,3,-1)$ & 4\\
        $\tau^2h_0\cdot a + h_0 \cdot \tau^2 a$ & $(4,4,0)$ & 4\\
        $h_1 \cdot \tau^2 a + \rho^3 b$ & $(5,4,1)$ &4\\
        $a^2 + h_0^2b$ & $(8,6,4)$ & 4\\
        $(\tau^2 a)^2 + \tau^4 h_0^2 b + \rho^2 \tau^4 h_1^2 b$ & $(8,6,0)$ & 8\\
        \hline
        \hline
        $\rho^2 \cdot \tau h_1$ & $(-1,1,-2)$ & 1\\
        $h_0 \cdot \tau h_1 + \rho h_1 \cdot \tau h_1$ & $(1,2,0)$ & 1\\
        $h_1^2 \cdot \tau h_1$ & $(3,3,2)$ & 1\\
        $\tau h_1 \cdot \tau^2 a$ & $(5,4,0)$ & 5\\
        \hline
        \hline
        $\rho \cdot \tau^2 h_0$ & $(-1,1,-3)$ & 2\\
        $ \rho^3 \cdot a$ & $(1,3,-1)$ & 2\\
        $\tau^2h_0 \cdot h_1 + \rho(\tau h_1)^2$ & $(1,2,-1)$ & 2\\
        $h_1(\tau h_1)^2 + \rho a$ & $(3,3,1)$ & 2\\
        $h_1 a$ & $(5,4,3)$ & 2\\
        $\tau^2 h_0 \cdot \tau^2a + \tau^4 h_0a$ & $(4,4,-2)$ & 6\\
        $a \cdot \tau^2 a + \tau^2 h_0 \cdot h_0 b$ & $(8,6,2)$ & 6\\
        \hline
        \hline
        $\tau^2 h_0 \cdot \tau h_1$ & $(1,2,-2)$ & 3\\
        $(\tau h_1)^3$ & $(3,3,0)$ & 3\\
        $\tau h_1 \cdot a$ & $(5,4,2)$ & 3\\
        \hline
    \end{tabular}}
    \caption{Relations for $\text{Ext}^{***}_{\euscr{A}(1)_\mathbb{R}^\vee}(\mathbb{M}_2^\mathbb{R})$.}        
    \label{GHIRrelations}
\end{table}

The data comprising $\text{Ext}^{***}_{\euscr{A}(1)^\vee_\mathbb{R}}(\mathbb{M}_2^\mathbb{R})$ is too complicated to display in one chart. To present this material more clearly, we give individual charts for each coweight modulo 4, where this 4-periodicity is induced by the generator $\tau^4$. Notice that there are no generators in coweight 3 modulo 4, and the relations ensure that there are never any classes in this coweight. Charts are written in $(s,f)$-grading, with motivic weight suppressed. We also suppress infinite $\rho$-towers, which would extend to the left in the charts, for readability. \Cref{R-Ext(M_2) cw=0} depicts the coweight 0 piece, \Cref{R-Ext(M_2) cw=1} depicts the coweight 1 piece, and \Cref{R-Ext(M_2) cw=2} depicts the coweight 2 piece.

A black $\blacksquare$ represents $\mathbb{F}_2[\rho,\tau^4]$. A black $\bullet$ represents $\mathbb{F}_2[\tau^4]$. A vertical black line represents multiplication by $h_0$. A horizontal black line represents multiplication by $\rho$. A diagonal black line represents multiplication by $h_1$. A dashed horizontal line indicates that $\rho$-multiplication hits $\tau^4$ times a generator. For example, in \Cref{R-Ext(M_2) cw=0} there is a dashed line indicating $\rho \cdot \tau^2a = \tau^4h_1^3$.
\begin{figure}[h]
    \centering
    \includegraphics[scale=0.75]{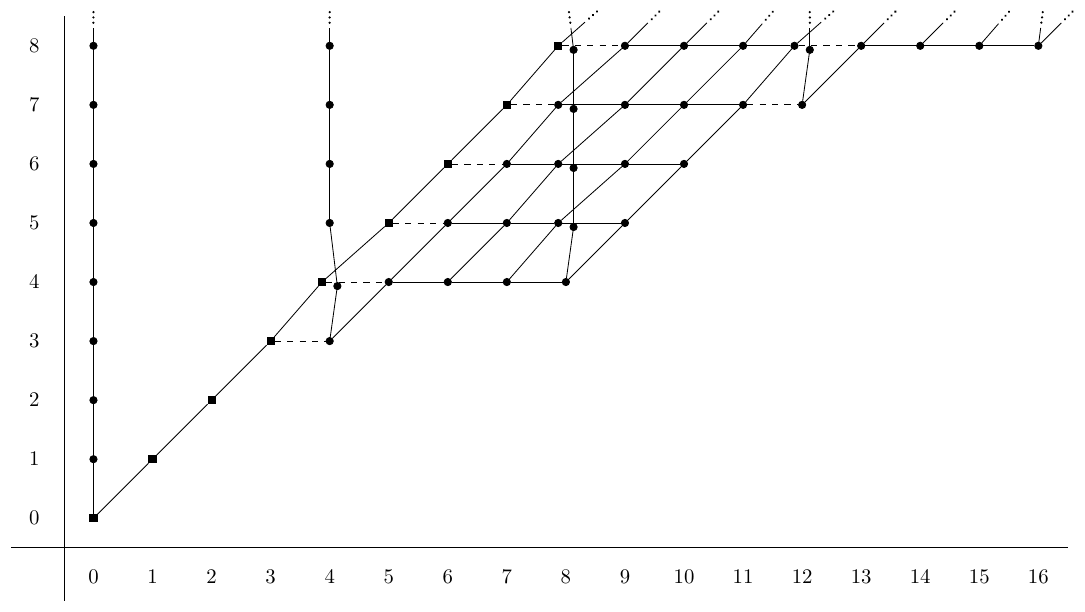}
    \caption{$\text{Ext}^{***}_{\euscr{A}(1)^\vee_\mathbb{R}.}(\mathbb{M}_2^\mathbb{R})$ in $cw \equiv 0 \, (4)$.}
    \label{R-Ext(M_2) cw=0}
\end{figure}

\begin{figure}[h]
    \centering
    \includegraphics[scale=.75]{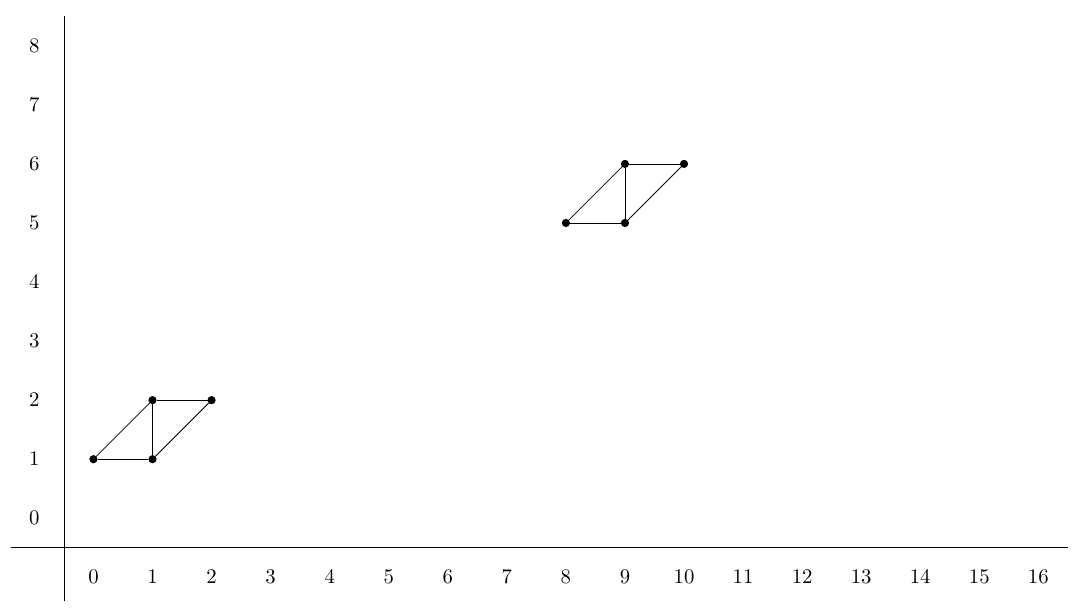}
    \caption{$\text{Ext}^{***}_{\euscr{A}(1)^\vee_\mathbb{R}.}(\mathbb{M}_2^\mathbb{R})$ in $cw \equiv 1 \, (4)$.}
    \label{R-Ext(M_2) cw=1}
\end{figure}

\begin{figure}[h]
    \centering
    \includegraphics[scale=0.75]{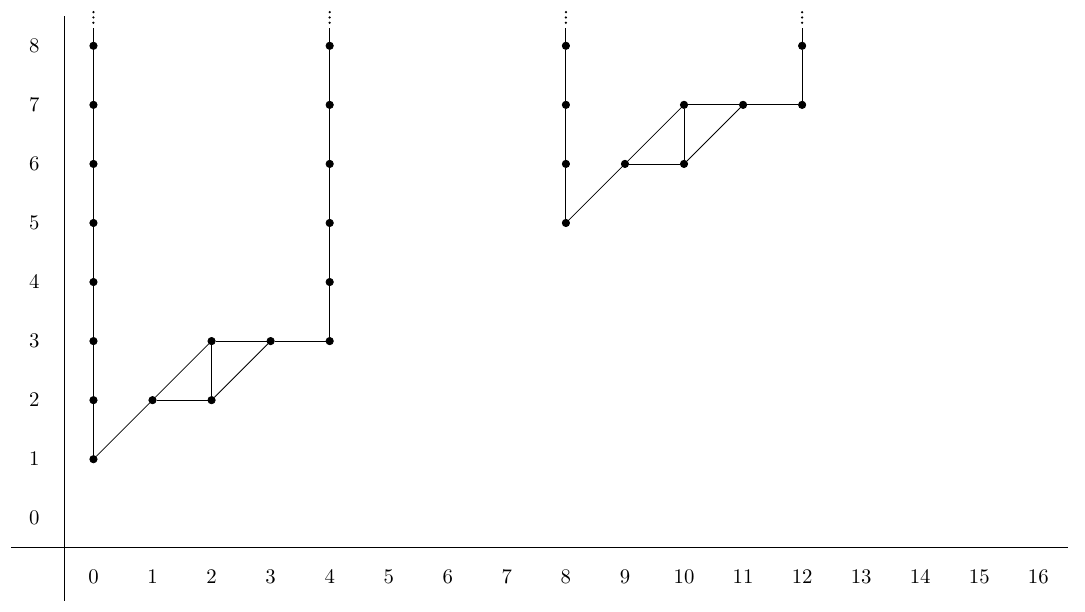}
    \caption{$\text{Ext}^{***}_{\euscr{A}(1)^\vee_\mathbb{R}.}(\mathbb{M}_2^\mathbb{R})$ in $ cw \equiv 2 \, (4)$.}
    \label{R-Ext(M_2) cw=2}
\end{figure}
\subsection{Some modules over $\text{Ext}^{***}_{\euscr{A}(1)^\vee_\mathbb{R}}(\mathbb{M}_2^\mathbb{R})$}
We now introduce notation for some modules over $\text{Ext}^{***}_{\euscr{A}(1)^\vee_\mathbb{R}}(\mathbb{M}_2^\mathbb{R})$ which will appear throughout. The reader is encouraged to consult \cref{charts} to follow along with these definitions.
\begin{definition}
\label{module defs}
    Define the following $\text{Ext}^{***}_{\euscr{A}(1)^\vee_\mathbb{R}}(\mathbb{M}_2^\mathbb{R})$-modules:
    \begin{itemize}
        \item Let $\euscr{D}$ denote the entire coweight $cw \equiv 1 \, (4)$ piece of $\text{Ext}^{***}_{\euscr{A}(1)_\mathbb{R}^\vee}
        (\mathbb{M}_2^\mathbb{R})$, called a \textit{big diamond} (see \Cref{R-Ext(M_2) cw=1});
        
        \item Let $\euscr{S}$ denote the entire coweight $cw \equiv 2 \, (4)$ piece of $\text{Ext}^{***}_{\euscr{A}(1)^\vee_\mathbb{R}}(\mathbb{M}_2^\mathbb{R})$, called a \textit{big staircase} (see \Cref{R-Ext(M_2) cw=2});
        
        \item Let $P$ be the module
        \[P = \mathbb{F}_2[\rho,\tau^4],\] 
        called a \textit{$\rho$-tower}, with the obvious $\text{Ext}^{***}_{\euscr{A}(1)^\vee_\mathbb{R}}(\mathbb{M}_2^\mathbb{R})$-module structure (see \Cref{rho tower picture}); 

        \item Let $H$ be the module
        \[H = \mathbb{F}_2[h_0,\tau^4],\]
        called an \textit{$h_0$-tower}, with the obvious $\text{Ext}^{***}_{\euscr{A}(1)^\vee_\mathbb{R}}(\mathbb{M}_2^\mathbb{R})$-module structure (see \Cref{h0 tower picture});
        
        \item Let $PH$ be the module
        \[PH = \frac{\mathbb{F}_2[h_0, \rho,\tau^4]}{(\rho h_0)},\]
        called a \textit{$(\rho, h_0)$-tower}, with the obvious $\text{Ext}^{***}_{\euscr{A}(1)^\vee_\mathbb{R}}(\mathbb{M}_2^\mathbb{R})$-module structure (see \Cref{rho h0 tower picture});

        \item Let $D$ be the module
        \[D = \frac{\mathbb{F}_2[h_0, h_1\rho,\tau^4]}{(\rho^2, h_0^2, h_1^2, \rho h_0, h_0 h_1,\rho h_1 = h_0 )},\]
        called a \textit{diamond}, with the obvious $\text{Ext}^{***}_{\euscr{A}(1)^\vee_\mathbb{R}}(\mathbb{M}_2^\mathbb{R})$-module structure (see \Cref{diamond picture});

        \item Let $S$ be the module
        \[S=\frac{\mathbb{F}_2[h_0, h_1, t, a,\tau^4]}{(h_1^3, t^2, a^2, h_0h_1, h_1a, ta, h_0^2t, h_1^2t)},\] called a \textit{staircase}, with 
        $\text{Ext}^{***}_{\euscr{A}(1)^\vee_\mathbb{R}}(\mathbb{M}_2^\mathbb{R})$-module structure given by the relations (see \Cref{staircase picture})
        \[\rho \cdot t = h_1, \quad h_1 \cdot t = \rho \cdot a,\quad h_0 \cdot t = h_1^2 = \rho^2 \cdot a, \quad\rho^2 \cdot t = 0, \quad \rho^3  \cdot a = 0;\]
        
        \item Let $T$ be the module
        \[T = \frac{\mathbb{F}_2[\rho,\tau^4]}{(\rho^2)},\]
        called a \textit{segment}, with the obvious $\text{Ext}^{***}_{\euscr{A}(1)^\vee_\mathbb{R}}(\mathbb{M}_2^\mathbb{R})$-module structure (see \Cref{segment picture});
        
        \item Let $J$ be the module
        \[J = \frac{\mathbb{F}_2[h_0, \rho,\tau^4]}{(\rho^3, \rho h_0)},\]
        called a \textit{J-tower}, with the obvious $\text{Ext}^{***}_{\euscr{A}(1)^\vee_\mathbb{R}}(\mathbb{M}_2^\mathbb{R})$-module structure (see \Cref{J-tower picture});
        
        \item Let 
        \[JD = \frac{\mathbb{F}_2[h_0, t, \rho,\tau^4]}{(\rho^3, t^2, \rho^2 t, \rho h_0,  h_0^2 t)},\]
        called a \textit{$JD$-tower}, with $\text{Ext}^{***}_{\euscr{A}(1)^\vee_\mathbb{R}}(\mathbb{M}_2^\mathbb{R})$-module structure given by the relations (see \Cref{JD-tower picture})
        \[\rho = h_1 t, \quad h_1 \cdot \rho t = h_0 t = \rho^2.\]
        \end{itemize}
\end{definition}

\begin{remark}
    As should be indicated by the figures in \cref{charts}, the naming of these modules is derived from their general shapes when depicted in charts.
\end{remark}

We construct now a more intricate family of modules. This module will track the behavior of the coweight $cw \equiv 0\, (4)$ piece of $\text{Ext}^{***}_{\euscr{A}(1)^\vee_\mathbb{R}}(\mathbb{M}_2^\mathbb{R})$ throughout the coming spectral sequences.

\begin{definition}
\label{big flag def}
   In this construction, we will attach generators to extend the coweight $cw \equiv 0 \, (4)$ piece of $\text{Ext}^{***}_{\euscr{A}(1)^\vee_\mathbb{R}}(\mathbb{M}_2^\mathbb{R})$. We continue until we have attached a generator in Adams filtration 0, at which point we stop. For $n \geq 0$, let $\euscr{F}_{s,f,w}$ be the  $\text{Ext}^{***}_{\euscr{A}(1)^\vee_\mathbb{R}}(\mathbb{M}_2^\mathbb{R})$-module concentrated in coweight $cw \equiv s-w \, (4)$, called a \textit{big flag}, described by the following process:
      \begin{itemize}
        \item[(0)] There is a copy of the coweight $cw \equiv 0 \, (4)$ piece of $\text{Ext}^{***}_{\euscr{A}(1)^\vee_\mathbb{R}}(\mathbb{M}_2^\mathbb{R})$ (see \Cref{R-Ext(M_2) cw=0}) with generator $x_0$ in degree $(s,f,w)$ called a \textit{flag}.
        \item[$(4n+1)$] There is a $PH$ generated by $x_{4n+1}$ in degree $(s-4-8n, f-1-4n, w-4-8n)$ subject to the relations 
        \[h_1x_{4n+1}=\rho^3x_{4n}\text{ and }bx_{4n+1}=\tau^2ax_{4n}.\]
        \item[$(4n+2)$] There is a $P$ generated by $x_{4n+2}$ in degree $(s-6-8n, f-2-4n, w-6-8n)$ subject to to the relations 
        \[h_1x_{4n+2}=\rho x_{4n+1} \text{ and } bx_{4n+2}=h_1^2x_{4n}.\]
        \item[$(4n+3)$]There is a $P$ generated by $x_{4n+3}$ in degree $(s-7-8n, f-3-4n, w-7-8n)$ subject to the relations 
        \[h_1x_{4n+3}=x_{4n+2}\text{ and }bx_{4n+3} = h_1x_{4n}.\] 
        \item[$(4n+4)$] There is a $PH$ generated by $x_{4n+4}$ in degree $(s-8(n+1), f-4(n+1), w-8(n+1))$ subject to the relations 
        \[h_1x_{4n+4}=x_{4n+3}\text{ and }bx_{4n+4}=x_{4n}.\] 
    \end{itemize}
    In short, the submodule $\euscr{F}_{s,f,w}$ looks like the coweight $cw \equiv 0 \, (4)$ portion of $\text{Ext}^{***}_{\euscr{A}(1)^\vee_\mathbb{R}}(\mathbb{M}_2^\mathbb{R})$ with generator shifted from $(0,0,0)$ to $(s,f,w)$, and then a complicated ``tail" pattern coming from the generator of the flag which decreases in all degrees until Adams filtration 0 is reached. Importantly, at each step we are attaching a new generator by $h_1$, and since $h_1$ has coweight 0 as an element of $\text{Ext}^{***}_{\euscr{A}(1)^\vee_\mathbb{R}}(\mathbb{M}_2^\mathbb{R})$, this ensures that each stage of this process is concentrated in the same coweight modulo 4. As an example, we depict the big flag $\euscr{F}_{12,3,6}$ in \Cref{newbigflag}.
\end{definition}

\begin{figure}[h]
\includegraphics[scale=.75]{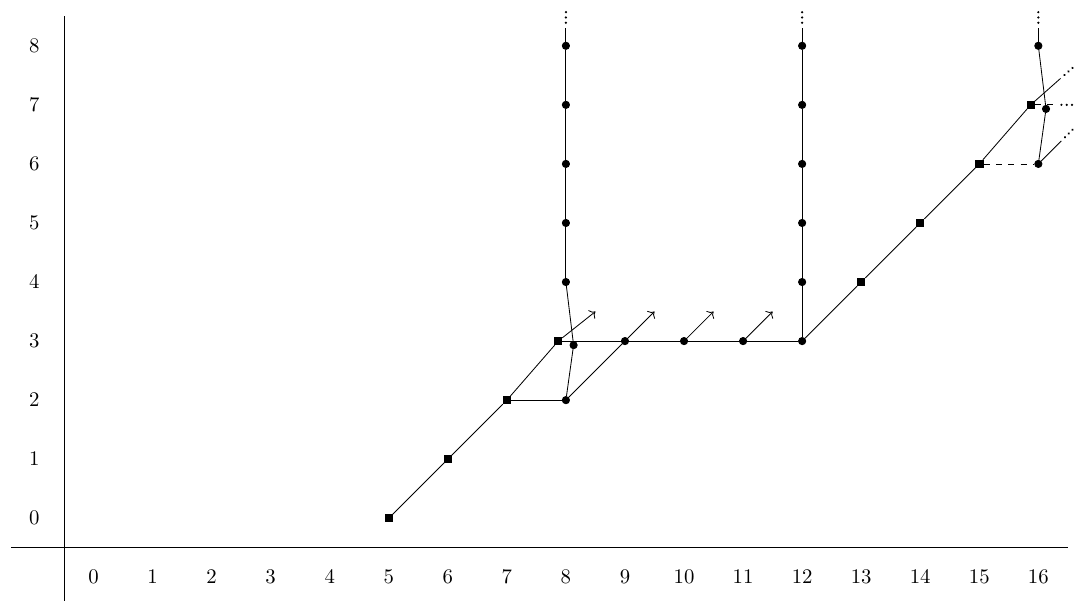}
\caption{The $\text{Ext}^{***}_{\euscr{A}(1)^\vee_\mathbb{R}}(\mathbb{M}_2^\mathbb{R})$-module $\euscr{F}_{12,3,6}$.}
\label{newbigflag}
\end{figure}

\begin{remark}
    The big flag $\euscr{F}_{0,0,0}$ is isomorphic to the coweight $cw \equiv 0 \, (4)$ piece of $\text{Ext}^{***}_{\euscr{A}(1)^\vee_\mathbb{R}}(\mathbb{M}_2^\mathbb{R})$ (see \Cref{R-Ext(M_2) cw=0}).
\end{remark}

\begin{remark}
    The big flag $\euscr{F}_{4n, n, 2n}$ is a submodule of the coweight $cw \equiv 0 \, (4)$ portion of the $\text{E}_2$-page of the $\textbf{mASS}^\mathbb{R}(\text{kq}^{\langle n \rangle})$, where $\text{kq}^{\langle n \rangle}$ denotes the $n^{th}$ Adams cover.
\end{remark}

\subsection{Comparison with $\mathbb{C}$}
\label{comparison with C}
There is a well-studied relationship between $\text{SH}(\mathbb{C})$ and $\text{SH}(\mathbb{R})$ \cite{BelIsa-stem}. In \cite{Hil11}, Hill developed the $\rho$-Bockstein spectral sequence, an algebraic spectral sequence which allows for the computation of $\mathbb{R}$-motivic information from $\mathbb{C}$-motivic information. The observation that $\mathbb{M}_2^\mathbb{C} = \mathbb{M}_2^\mathbb{R}/\rho$ lifts to the dual Steenrod algebra, giving $\euscr{A}^\vee_\mathbb{C} = \euscr{A}^\vee_\mathbb{R}/\rho$. Filtering the cobar complex which computes $\text{Ext}^{***}_{\euscr{A}^\vee_\mathbb{R}}(\mathbb{M}_2^\mathbb{R})$ by powers of $\rho$, we obtain the $\rho$-Bockstein spectral sequence
\[\textup{E}_1 = \text{Ext}^{***}_{\euscr{A}^\vee_\mathbb{C}}(\mathbb{M}_2^\mathbb{C})[\rho] \implies \text{Ext}^{***}_{\euscr{A}^\vee_\mathbb{R}}(\mathbb{M}_2^\mathbb{R}).\]
This spectral sequence converges to the $\textup{E}_2$-page of the $\textbf{mASS}^\mathbb{R}(\mathbb{S})$. One can also filter the cobar complex which computes $\text{Ext}^{***}_{\euscr{A}(1)^\vee_\mathbb{R}}(\mathbb{M}_2^\mathbb{R})$ by powers of $\rho$, giving a $\rho$-Bockstein spectral sequence
\[\textup{E}_1  = \text{Ext}^{***}_{\euscr{A}(1)^\vee_\mathbb{C}}(\mathbb{M}_2^\mathbb{C})[\rho] \implies \text{Ext}^{***}_{\euscr{A}(1)^\vee_\mathbb{R}}(\mathbb{M}_2^\mathbb{R}).\]
This spectral sequence converges to the $\textup{E}_2$-page of the $\textbf{mASS}^\mathbb{R}(\textup{kq})$.

There is also an extension of scalars functor $-\otimes \mathbb{C}:\text{SH}(\mathbb{R}) \to \text{SH}(\mathbb{C})$. For $X \in \text{SH}(\mathbb{R})$, we denote its image under extensions of scalars by $X^{\mathbb{C}}$. This functor has the property $\rho \otimes \mathbb{C} = 0$ (modulo 2) and Eilenberg-Mac Lane spectra to are sent to Eilenberg-Mac Lane spectra. In particular, this gives maps
\[\mathbb{M}_2^\mathbb{R} \to \mathbb{M}_2^\mathbb{C}\cong \mathbb{M}_2^\mathbb{R}/\rho, \quad \euscr{A}^\vee_\mathbb{R} \to \euscr{A}^\vee_\mathbb{C}\cong \euscr{A}^\vee_\mathbb{R} / \rho.\] 
Extension of scalars induces highly structured maps of Adams spectral sequences, which will allow us to lift $\mathbb{C}$-motivic extensions to $\text{SH}(\mathbb{R})$.

In the 2-complete setting, one may lift $\rho$ to the sphere, giving an element $\rho \in \pi_{-1, -1}^\mathbb{R}(\mathbb{S})$. This gives a cofiber sequence
\[\Sigma^{-1, -1}\mathbb{S} \xrightarrow{\rho} \mathbb{S} \to \mathbb{S}/\rho.\]
This cofiber sequence induces a long exact sequence in $\text{Ext}_{\euscr{A}^\vee_\mathbb{R}}$. Since $\rho \otimes \mathbb{C} = 0$, extension of scalars on this long exact sequence gives a split long exact sequence in $\text{Ext}_{\euscr{A}^\vee_\mathbb{C}}$. Moreover, we can identify $\textup{H}_{**}(\mathbb{S}/\rho) \cong \mathbb{M}_2^\mathbb{R}/\rho$ since $\mathbb{M}_2^\mathbb{R}$ is $\rho$-torsion free. This gives us the following diagram.
\[\begin{tikzcd}
	\cdots & {\text{Ext}^{***}_{\euscr{A}^\vee_{\mathbb{R}}}(\mathbb{M}_2^\mathbb{R})} & {\text{Ext}^{***}_{\euscr{A}^\vee_\mathbb{R}}(\mathbb{M}_2^\mathbb{R})} & {\text{Ext}^{***}_{\euscr{A}^\vee_\mathbb{R}}(\mathbb{M}_2^\mathbb{R}/\rho)} & \cdots \\
	\cdots & {\text{Ext}^{***}_{\euscr{A}^\vee_\mathbb{C}}(\mathbb{M}_2^\mathbb{C})} & {\text{Ext}^{***}_{\euscr{A}^\vee_\mathbb{C}}(\mathbb{M}_2^\mathbb{C})} & {\text{Ext}^{***}_{\euscr{A}^\vee_\mathbb{C}}(\mathbb{M}_2^\mathbb{C}) \oplus \text{Ext}^{***}_{\euscr{A}^\vee_\mathbb{C}}(\mathbb{M}_2^\mathbb{C})} & \cdots
	\arrow[from=1-1, to=1-2]
	\arrow["\rho", from=1-2, to=1-3]
	\arrow[from=1-2, to=2-2]
	\arrow[from=1-3, to=1-4]
	\arrow[from=1-3, to=2-3]
	\arrow[from=1-4, to=1-5]
	\arrow[from=1-4, to=2-4]
	\arrow[from=2-1, to=2-2]
	\arrow["0", from=2-2, to=2-3]
	\arrow[from=2-3, to=2-4]
	\arrow[from=2-4, to=2-5]
\end{tikzcd}\]

The splitting on the bottom row gives a lift $\text{Ext}_{\euscr{A}^\vee_\mathbb{R}}(\mathbb{M}_2^\mathbb{R}/\rho) \to\text{Ext}_{\euscr{A}^\vee_\mathbb{C}}(\mathbb{M}_2^\mathbb{C})$. In fact, this map is an isomorphism due to the following.

\begin{proposition}
\label{rho-torsion free base change prop homology}
    Let $X \in \textup{SH}(\mathbb{R})$ be any spectrum such that $\textup{H}_{**}(X)$ is $\rho$-torsion free. Then there is an isomorphism
    \[\textup{Ext}^{***}_{\euscr{A}^\vee_\mathbb{R}}(\textup{H}_{**}(X/\rho)) \xrightarrow{\cong} \textup{Ext}^{***}_{\euscr{A}^\vee_\mathbb{C}}(\textup{H}_{**}(X^ \mathbb{C})).\]
\end{proposition}

\begin{proof}
    First, since $\textup{H}_{**}(X)$ is $\rho$-torsion free, we know that $\textup{H}_{**}(X/\rho) \cong \textup{H}_{**}(X)/\rho$, where $X/\rho \simeq X \otimes \mathbb{S}/\rho$. This follows from the long exact sequence in homology associated to the cofiber sequence
    \[X \xrightarrow{\rho} X \to X/\rho.\]
    Note that in this case, we have an isomorphism $\textup{H}_{**}(X)/\rho \cong \textup{H}_{**}(X^\mathbb{C})$. The result now follows from an isomorphism of cobar complexes 
    \[C_{\euscr{A}^\vee_\mathbb{R}}(\textup{H}_{**}(X/\rho))\cong C_{\euscr{A}_\mathbb{R}^\vee}(\textup{H}_{**}(X))/\rho \cong C_{\euscr{A}_\mathbb{C}^\vee}(\textup{H}_{**}(X^\mathbb{C})).\]
    The left isomorphism holds because the cobar complex $C_{\euscr{A}_\mathbb{R}^\vee}(\textup{H}_{**}(X/\rho))$ is entirely $\rho$-torsion free. The right isomorphism holds since $\textup{H}_{**}(X)/\rho \cong \textup{H}_{**}(X^\mathbb{C})$ and $\mathbb{M}_2^\mathbb{R}/\rho \cong \mathbb{M}_2^\mathbb{C}$.
\end{proof}
One immediate consequence of \cref{rho-torsion free base change prop homology} is that $\text{Ext}^{***}_{\euscr{A}^\vee_\mathbb{C}}(\mathbb{M}_2^\mathbb{C})$ is a module over $\text{Ext}^{***}_{\euscr{A}^\vee_\mathbb{R}}(\mathbb{M}_2)$. This module structure is simple to describe. By the $\rho$-Bockstein, we can represent classes in $\text{Ext}^{***}_{\euscr{A}^\vee_\mathbb{R}}(\mathbb{M}_2^\mathbb{R})$ in the form $\rho^kx$  for some $x \in \text{Ext}^{***}_{\euscr{A}^\vee_\mathbb{C}}(\mathbb{M}_2^\mathbb{C})$. If $y \in \text{Ext}^{***}_{\euscr{A}^\vee_\mathbb{C}}(\mathbb{M}_2^\mathbb{C})$, then $\rho^kx \cdot y$ is only nonzero when $k=0$, and in this case we have $x\cdot y=xy$.
\begin{corollary}
\label{rho torison free base change homotopy groups corollary}
    Let $X \in \textup{SH}(\mathbb{R})$ be any spectrum such $\textup{H}_{**}(X)$ is $\rho$-torsion free. Then there is an isomorphism
    \[\pi_{**}^\mathbb{R}(X/\rho) \cong \pi_{**}^\mathbb{C}(X^\mathbb{C}).\]
\end{corollary}

\begin{proof}
    The map from \cref{rho-torsion free base change prop homology} gives an isomorphism of $\textup{E}_2$-pages of Adams spectral sequences.
\end{proof}

Note that $\textup{kq}^\mathbb{C}$ is the spectrum in $\text{SH}(\mathbb{C})$ representing very effective hermitian K-theory over $\mathbb{C}$, which we will abusively also denote by kq. Since $\textup{H}_{**}(\textup{kq}) \cong (\euscr{A}\modmod\euscr{A}(1))_\mathbb{R}^\vee$ is $\rho$-torsion free, it follows from \cref{rho torison free base change homotopy groups corollary} that $\pi_{**}^\mathbb{R}(\textup{kq}/\rho) \cong \pi_{**}^\mathbb{C}(\textup{kq})$. The K\"unneth spectral sequence for $\textup{H}_{**}(\textup{kq} \otimes \textup{kq})$ collapses by \cref{Kunneth}, so we also have 
\[\pi_{**}^\R((\textup{kq} \otimes \textup{kq})/\rho) \cong \pi_{**}^\mathbb{C}(\textup{kq} \otimes \textup{kq}).\] 
Combining \cref{rho torison free base change homotopy groups corollary} and \cref{kq bar homology}, we can generalize this idea.

\begin{proposition}
    The $\rho$-periodic $\textup{E}_1$-page of the real $\textup{kq}$-resolution is isomorphic to the $\textup{E}_1$-page of the complex $\textup{kq}$-resolution.
\end{proposition}

We can also prove analogous statements using the Brown--Gitler comodules. 
\begin{proposition}
\label{bg-comod R/rho = C}
    There are isomorphisms for all $k \geq 0$:
    \[\textup{Ext}^{***}_{\euscr{A}(1)^\vee_{\mathbb{R}}}(B_0^{\mathbb{R}}(k)^{\otimes i}/\rho) \cong \textup{Ext}^{***}_{\euscr{A}(1)^\vee_{\mathbb{C}}}(B_0^{\mathbb{C}}(k)^{\otimes i}).\]
\end{proposition}

\begin{proof}
    First, note that we have a diagram in Ext coming from the long exact sequences induced by the cofiber sequence
    \[\textup{kq} \xrightarrow{\rho} \textup{kq} \to \textup{kq}/\rho\]
    and extension of scalars, combined with the usual change of rings isomorphism:
    \[\begin{tikzcd}
	\cdots & {\text{Ext}^{***}_{\euscr{A}(1)^\vee_{\mathbb{R}}}(\mathbb{M}_2^\mathbb{R})} & {\text{Ext}^{***}_{\euscr{A}(1)^\vee_\mathbb{R}}(\mathbb{M}_2^\mathbb{R})} & {\text{Ext}^{***}_{\euscr{A}(1)^\vee_\mathbb{R}}(\mathbb{M}_2^\mathbb{R}/\rho)} & \cdots \\
	\cdots & {\text{Ext}^{***}_{\euscr{A}(1)^\vee_\mathbb{C}}(\mathbb{M}_2^\mathbb{C})} & {\text{Ext}^{***}_{\euscr{A}(1)^\vee_\mathbb{C}}(\mathbb{M}_2^\mathbb{C})} & {\text{Ext}^{***}_{\euscr{A}(1)^\vee_\mathbb{C}}(\mathbb{M}_2^\mathbb{C}) \oplus \text{Ext}^{***}_{\euscr{A}(1)^\vee_\mathbb{C}}(\mathbb{M}_2^\mathbb{C})} & \cdots
	\arrow[from=1-1, to=1-2]
	\arrow["\rho", from=1-2, to=1-3]
	\arrow[from=1-2, to=2-2]
	\arrow[from=1-3, to=1-4]
	\arrow[from=1-3, to=2-3]
	\arrow[from=1-4, to=1-5]
	\arrow[from=1-4, to=2-4]
	\arrow[from=2-1, to=2-2]
	\arrow["0", from=2-2, to=2-3]
	\arrow[from=2-3, to=2-4]
	\arrow[from=2-4, to=2-5]
    \end{tikzcd}\]
    This gives an isomorphism in Ext by the same argument as in \Cref{rho-torsion free base change prop homology}:
    \[\text{Ext}^{***}_{\euscr{A}(1)^\vee_\mathbb{R}}(\mathbb{M}_2^\mathbb{R}/\rho) \cong \text{Ext}^{***}_{\euscr{A}(1)^\vee_{\mathbb{C}}}(\mathbb{M}_2^\mathbb{C}).\]
    By construction, the Brown--Gitler comodules $B_0^\mathbb{R}(k)$ are $\rho$-torsion free since they are subcomodules of the dual Steenrod algebra. This extends naturally to tensor factors. The analogous diagram of long exact sequences in Ext gives the desired isomorphism.
\end{proof}

These isomorphisms will be helpful for determining hidden extensions in the $\textbf{aAHSS}(B_0^\mathbb{R}(k))$ as they allow us to compare with the computations of \cite{CQ21} recalled in \cref{section 3}.

\begin{remark}
    If $F$ is a field with a real embedding, then there is a $C_2$-equivariant Betti realization functor
    \[\text{Be}^{C_2}:\text{SH}(F) \to \text{Sp}^{C_2},\]
    obtained from the Betti realization functor of \cref{betti remark r vs c} by remembering the $C_2$-action given by complex conjugation. This functor has been well studied, see for example \cite{HelOrm-Galois}, and exhibits a close relationship between $\mathbb{R}$-motivic and $C_2$-equivariant homotopy theory. It was shown by \cite{Kon23} that
    \[\text{Be}^{C_2}(\textup{kq})_{2,\eta}^\wedge \simeq (\text{ko}_{C_2})_{2,\eta}^\wedge.\]
    Moreover, there is a decomposition of the homology of a point as $\mathbb{M}_2^{C_2} \cong \mathbb{M}_2^\mathbb{R} \oplus NC$, where $NC$ is the so-called ``negative cone". This leads to a splitting of the $C_2$-equivariant dual Steenrod algebra as $\euscr{A}^\vee_{C_2} \cong \euscr{A}_\mathbb{R}^\vee \otimes _{\mathbb{M}_2^\mathbb{R}} \mathbb{M}_2^{C_2}$, and hence an isomorphism of Ext groups:
    \[\text{Ext}^{***}_{\euscr{A}^\vee_{C_2}} \cong \text{Ext}^{***}_{\euscr{A}_\mathbb{R}^\vee} \oplus \text{Ext}^{***}_{NC}.\]
    This splitting implies that the $\textup{E}_2$-page of the $C_2$-equivariant Adams spectral sequence computing the ring of cooperations $\pi_\star^{C_2}(\text{ko}_{C_2} \otimes \text{ko}_{C_2})$ contains the $\textup{E}_2$-page of the $\textbf{mASS}(\textup{kq} \otimes \textup{kq})$ as a summand. This will be investigated further in future work with Petersen and Tatum.
\end{remark}

\section{The algebraic Atiyah--Hirzebruch spectral sequence}
\label{R-aAHSS section}
In this section, we  compute the groups $\text{Ext}^{***}_{\euscr{A}(1)^\vee_\mathbb{R}}(B_0^\mathbb{R}(1)^{\otimes i})$ for all $i \geq 0$ by a series of algebraic Atiyah--Hirzebruch spectral sequences. Motivated by the presentation of $\text{Ext}^{***}_{\euscr{A}(1)^\vee_\mathbb{R}}(\mathbb{M}_2^\mathbb{R})$ given in \cite{GHIRkoc2}, we organize these computations by coweight modulo 4. Consequently, our results are best displayed in multiple charts which should be considered simultaneously. Throughout, we use the notation of \cref{section 4}.
\label{section R-aAHSS}
\subsection{$\text{Ext}^{***}_{\euscr{A}(1)^\vee_\mathbb{R}}(B_0^\mathbb{R}(1))$}
\label{subsection:Ext_B0R1}
We first compute $\text{Ext}^{***}_{\euscr{A}(1)^\vee_\mathbb{R}}(B_0^\mathbb{R}(1))$ by the $\textbf{aAHSS}(B_0^\mathbb{R}(1))$. This spectral sequence has signature
\[\textup{E}_1^{s,f,w,a}=\text{Ext}^{s,f,w}_{\euscr{A}(1)^\vee_\mathbb{R}}(\mathbb{M}_2^\mathbb{R}) \otimes_{\mathbb{M}_2^\mathbb{R}} \mathbb{M}_2^\mathbb{R}\{[1], [\overline{\xi}_1], [\overline{\tau}_1]\} \implies \text{Ext}_{\euscr{A}(1)^\vee_\mathbb{R}}^{s,f,w}(B_0^\mathbb{R}(1)),\] 
and the differentials take the form
\[d_1:\text{E}_r^{s,f,w,a} \to \text{E}_r^{s-1, f+1, w, a-r}.\]
The $d_1$ and $d_2$ differentials in this spectral sequence are determined by our discussion in \cref{aAHSS generic}. In particular the $d_1$-differential is only nonzero on Atiyah--Hirzebruch filtration 3, where
\[d_1(\alpha[3]) = h_0\alpha[2],\]
and the $d_2$-differential is only nonzero on Atiyah--Hirzebruch filtration 2, where
\[d_2(\alpha[2]) = h_1\alpha[0].\]
Note that in terms of coweight, both of these differentials are determined by multiplication by an element in $\text{Ext}^{***}_{\euscr{A}(1)^\vee_\mathbb{R}}(\mathbb{M}_2^\mathbb{R})$ of coweight 0. This implies that we can represent the $d_1$ and $d_2$ differentials as maps from one Atiyah--Hirzebruch filtration piece to another in the exact same coweight. In particular, as the coweight $cw \equiv 3 \, (4)$ portion of $\text{Ext}^{***}_{\euscr{A}(1)^\vee_\mathbb{R}}(\mathbb{M}_2^\mathbb{R})$ was shown to be 0 in \Cref{thm:Ext_A1_R}, we can ensure that no differentials land in this vanishing region.

Given the complicated nature of this spectral sequence, we only present charts where the differential is nonzero and only include the relevant Atiyah--Hirzebruch filtrations. For example, since the $d_1$-differential is trivial on Atiyah--Hirzebruch filtration 0, we omit this filtration from the charts for the $\textup{E}_1$-page. 

The $\textup{E}_1$-page is of the $\textbf{aAHSS}(B_0^{\mathbb{R}}(1))$ is depicted in \Cref{R-aAHSS(b_0(1)) E1 coweight 0}, \Cref{R-aAHSS(b_0(1)) E1  coweight 1}, and \Cref{R-aAHSS(b_0(1)) E1 coweight 2}. The $\textup{E}_2$-page of the $\textbf{aAHSS}(B_0^\mathbb{R}(1))$ is depicted in \Cref{R-aAHSS(b_0(1)) E2 coweight 0}, \Cref{R-aahss(b_0(1)) E2 coweight 1}, and  \Cref{R-aAHSS(b_0(1)) E2 coweight 2}.
Differentials are blue, linear with respect to the underlying $\text{Ext}_{\euscr{A}(1)^\vee_\mathbb{R}}(\mathbb{M}_2^\mathbb{R})$-module structure, and preserve motivic weight.

\begin{figure}[h]
    \centering
    \includegraphics[scale=.75]{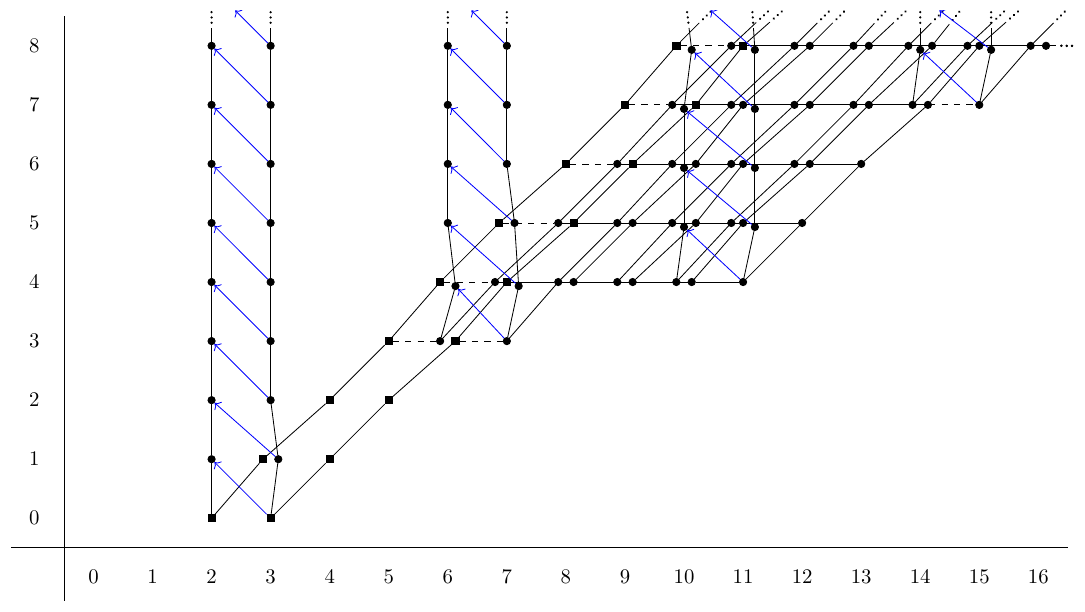}
    \caption{The $\textup{E}_1$-page of the \textbf{aAHSS}($B_0^\mathbb{R}(1)$) with $cw \equiv 0 \, (4)$ in Atiyah--Hirzebruch filtration 2 and 3.}
    \label{R-aAHSS(b_0(1)) E1 coweight 0}
\end{figure}
\begin{figure}[h]
    \centering
    \includegraphics[scale=0.75]{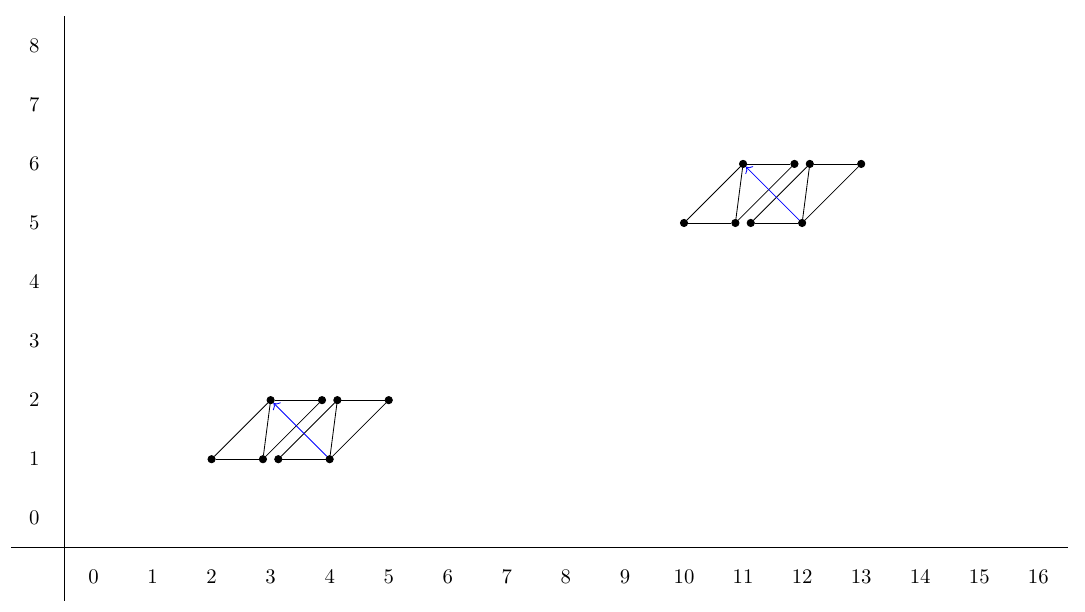}
    \caption{The $\textup{E}_1$-page of the \textbf{aAHSS}($B_0^\mathbb{R}(1)$) with $cw \equiv 1 \, (4)$ in Atiyah--Hirzebruch filtration 2 and 3.}
    \label{R-aAHSS(b_0(1)) E1  coweight 1}
\end{figure}
\begin{figure}[h]
    \centering
    \includegraphics[scale=0.75]{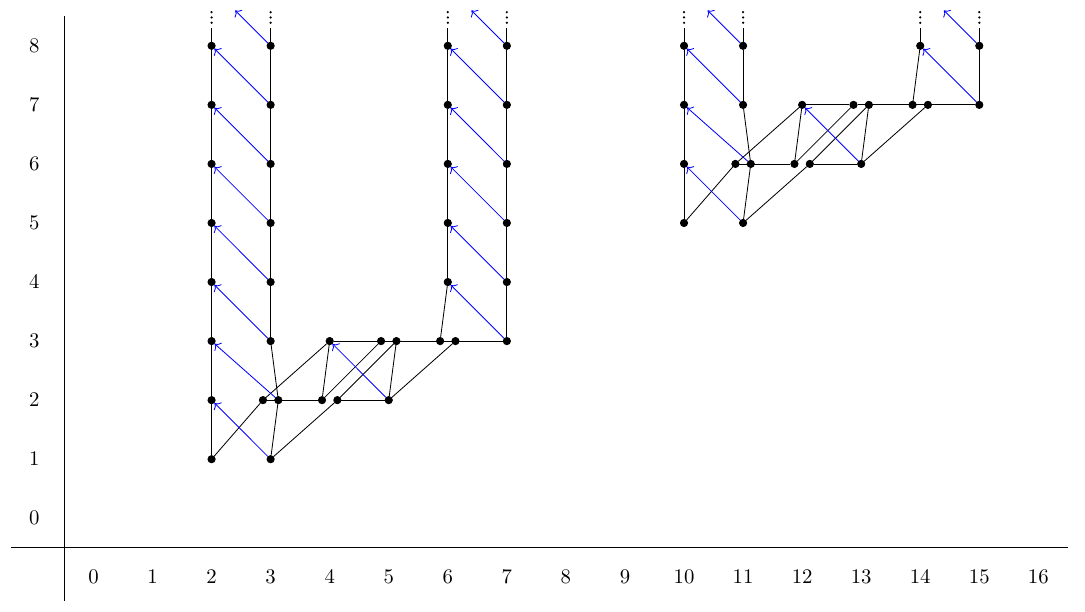}
    \caption{The $\textup{E}_1$-page of the \textbf{aAHSS}($B_0^\mathbb{R}(1)$) with $cw \equiv 2 \, (4)$ in Atiyah--Hirzebruch filtration 2 and 3.}
    \label{R-aAHSS(b_0(1)) E1 coweight 2}
\end{figure}
\begin{figure}[h]
    \centering
    \includegraphics[scale=.75]{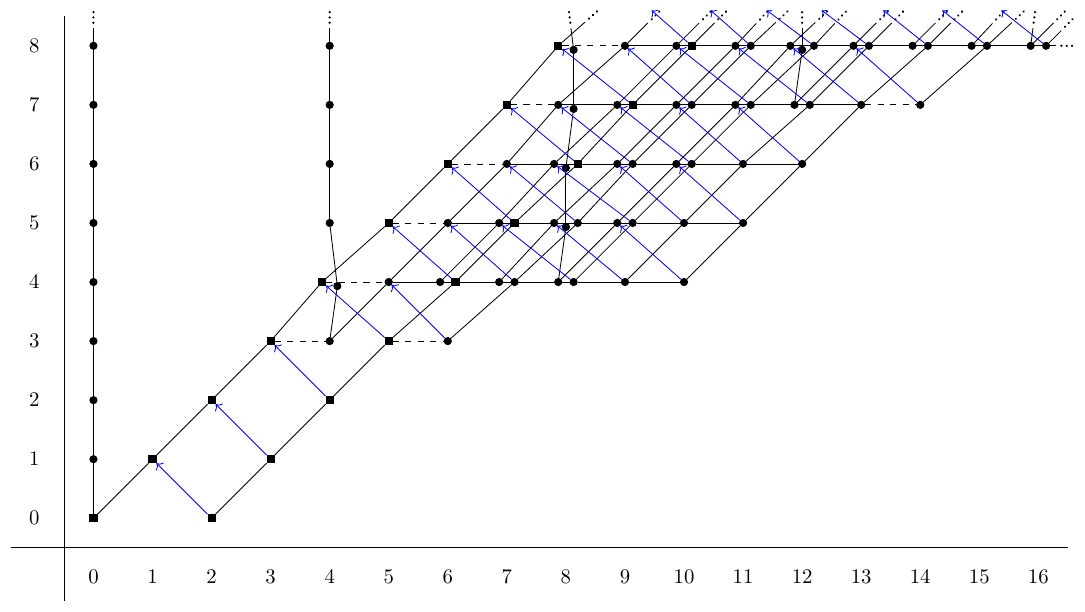}
    \caption{The $\textup{E}_2$-page of the \textbf{aAHSS}($B_0^\mathbb{R}(1)$) with $cw \equiv 0 \, (4)$ in Atiyah--Hirzebruch filtration 0 and 2.}
    \label{R-aAHSS(b_0(1)) E2 coweight 0}
\end{figure}

\begin{figure}[h]
    \centering
    \includegraphics[scale=.75]{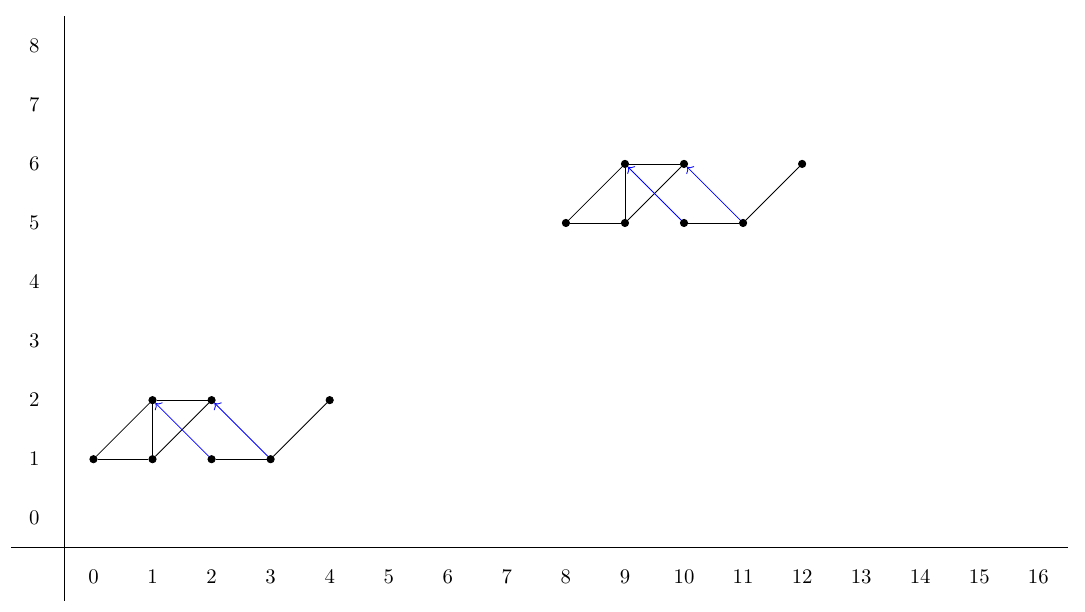}
    \caption{The $\textup{E}_2$-page of the \textbf{aAHSS}($B_0^\mathbb{R}(1)$) with $cw \equiv 1 \, (4)$ in Atiyah--Hirzebruch filtration 0 and 2.}
    \label{R-aahss(b_0(1)) E2 coweight 1}
\end{figure}

\begin{figure}[h]
    \centering
    \includegraphics[scale=.75]{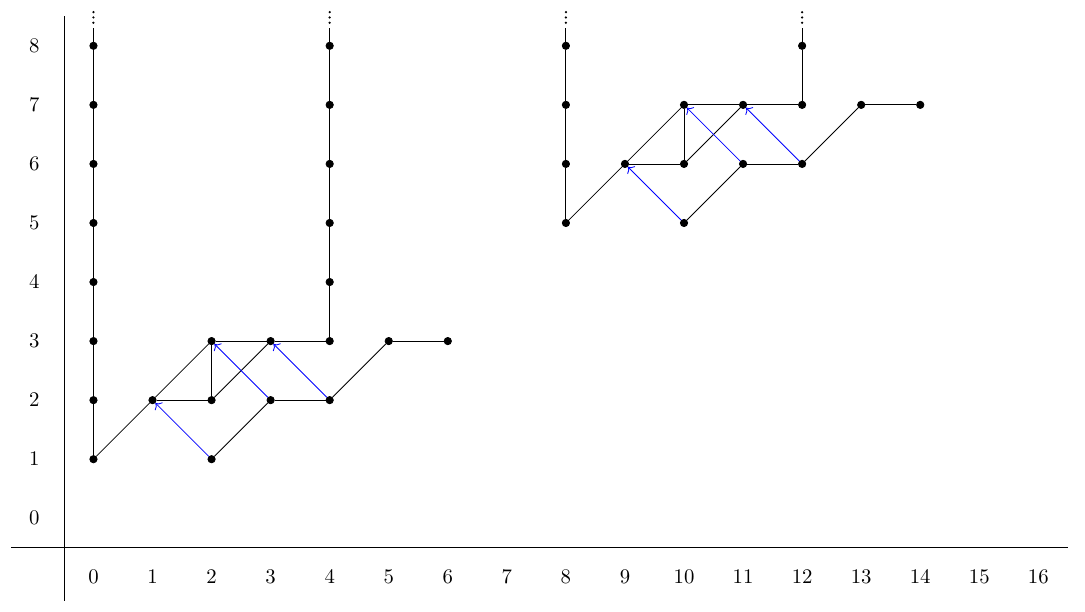}
    \caption{The $\textup{E}_2$-page of the \textbf{aAHSS}($B_0^\mathbb{R}(1)$) with $cw \equiv 2 \, (4)$ in Atiyah--Hirzebruch filtration 0 and 2.}
    \label{R-aAHSS(b_0(1)) E2 coweight 2}
\end{figure}

As we remarked in \cref{2.5}, the $d_3$-differential in this spectral sequence is dependent on the field we are working over. While both classically and $\mathbb{C}$-motivically we saw that the $d_3$-differential, which is only potentially nonzero on Atiyah--Hirzebruch filtration 3, was actually trivial for degree reasons, we see that this is not the case here. By inspection, the only possible nonzero $d_3$ differential is between Atiyah--Hirzebruch filtrations 3 and 0, and will be given by the formula
\[d_3(\alpha[3]) = \langle\alpha, h_0, h_1 \rangle[0].\]

\begin{lemma}
\label{d3 lemma}
    There is a $d_3$-differential in the \textup{\textbf{aAHSS}}\textup{(}$B_0^\mathbb{R}(1)$\textup{)}:
    \[d_3(\rho[3]) = (\tau h_1)[0].\]
\end{lemma}

\begin{proof}
    This differential is witnessed by the Massey product $\langle \rho, h_0, h_1 \rangle$. This Massey product is shown to be equal to $\tau h_1$ in \cite[Theorem 8.1]{GHIRkoc2} with zero indeterminacy.
\end{proof}
In fact, since the differentials are linear over $\text{Ext}^{***}_{\euscr{A}(1)^\vee_\mathbb{R}}(\mathbb{M}_2^\mathbb{R})$ this determines all $d_3$-differentials. For example, we have 
\[d_3(b \rho[3]) = b \cdot \tau h_1[0],\]
and we have
\[d_3(\rho \cdot \tau h_1[3]) = (\tau h_1)^2[0].\]
It is important to note that since $\tau h_1$ has coweight 1 in $\text{Ext}^{***}_{\euscr{A}(1)^\vee_\mathbb{R}}(\mathbb{M}_2^\mathbb{R})$, and since each Atiyah--Hirzebruch filtration is zero in coweight 3, the $d_3$-differential is trivial on the coweight $cw \equiv 2 \, (4)$ portion of Atiyah--Hirzebruch filtration 3. This coweight jump requires us to pay close attention to the $d_3$-differential. For example, in the coweight $cw \equiv 0 \, (4)$ piece of Atiyah--Hirzebruch filtration 3, there is an infinite $\rho$-tower stemming from $\rho[3]$. However, while 
\[d_3(\rho[3])=\tau h_1[0]\]
and 
\[d_3(\rho^2[3]) = \rho \cdot \tau h_1[0],\]
there are no other classes in the coweight $cw \equiv 1 \,(4)$ piece of Atiyah--Hirzebruch filtration 0 for the rest of this $\rho$-tower to hit. In other words, for we have
\[d_3(\rho^n[3])=0[0], \, n \geq 2.\]

\Cref{R-aAHSS(b0(1)) AHF=3 CW=0} depicts the $\textup{E}_3$-page where the coweight of Atiyah--Hirzebruch filtration 3 is $cw \equiv 0 \, (4)$ and the coweight of Atiyah--Hirzebruch filtration 0 is $cw \equiv 1 \, (4)$. \Cref{R-aAHSS(b0(1)) AHF=3 CW=1} depicts the $\textup{E}_3$-page where the coweight of Atiyah--Hirzebruch filtration 3 is $cw \equiv 1 \, (4)$ and the coweight of Atiyah--Hirzebruch filtration 0 is $cw \equiv 2 \, (4)$.

\begin{figure}[h]
    \centering
    \includegraphics[scale=.75]{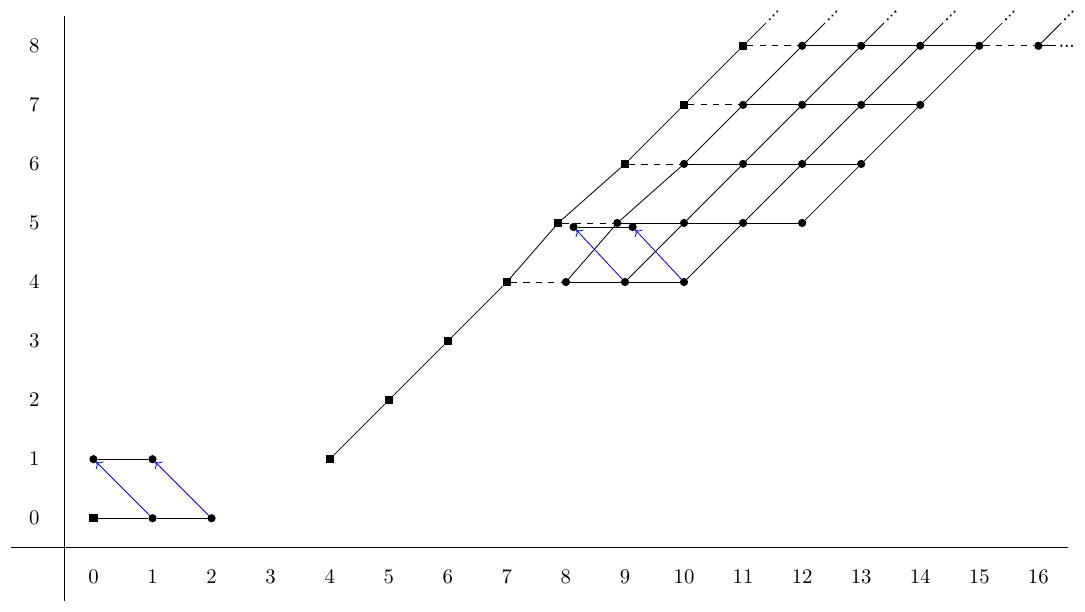}
    \caption{The $\textup{E}_3$-page of the \textbf{aAHSS}($B_0^\mathbb{R}(1)$) with $cw \equiv 0 \, (4)$ in Atiyah--Hirzebruch filtration 3 and $cw \equiv 1 \, (4)$ in Atiyah--Hirzebruch filtration 0.}
    \label{R-aAHSS(b0(1)) AHF=3 CW=0}
\end{figure}

\begin{figure}[h]
    \centering
    \includegraphics[scale=.75]{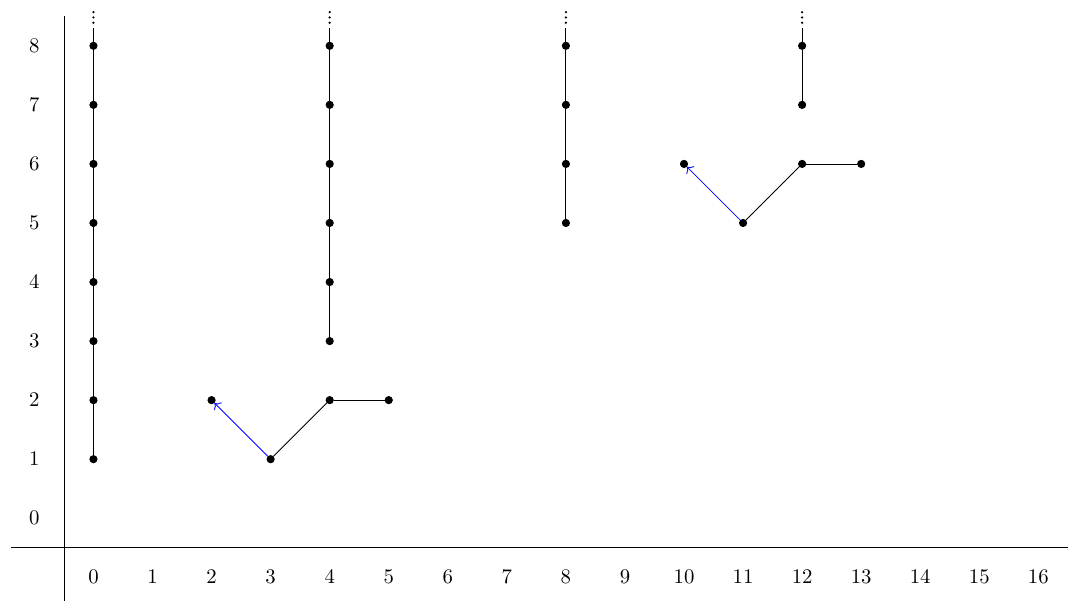}
    \caption{The $\textup{E}_3$-page of the \textbf{aAHSS}($B_0^\mathbb{R}(1)$) with $cw \equiv 1 \, (4)$ in Atiyah--Hirzebruch filtration 3 and $cw \equiv 2 \, (4)$ in Atiyah--Hirzebruch filtration 0.}
    \label{R-aAHSS(b0(1)) AHF=3 CW=1}
\end{figure}

Recall from \cref{aAHSS convergenece general} that $\textup{E}_4=\textup{E}_\infty$ for Atiyah--Hirzebruch filtration reasons.
Hidden extensions can be realized using complex Betti realization, linearity over $\text{Ext}^{***}_{\euscr{A}(1)^\vee_\mathbb{R}}(\mathbb{M}_2^\mathbb{R})$, base change $\otimes\mathbb{C}: \text{SH}(\mathbb{C}) \to \text{SH}(\mathbb{R})$, and the isomorphism from \cref{bg-comod R/rho = C}:
\[\text{Ext}^{***}_{\euscr{A}(1)^\vee_\mathbb{R}}(B_0^\mathbb{R}(1)/\rho) \cong \text{Ext}^{***}_{\euscr{A}(1)_\mathbb{C}}(B_0^\mathbb{C}(1)).\]

The following result will be useful when depicting the $\textup{E}_\infty$-page in charts.
\begin{lemma}
\label{trivial cw1 ext b01}
    The group $\textup{Ext}^{***}_{\euscr{A}(1)^\vee_\mathbb{R}}(B_0^\mathbb{R}(1))$ is zero in coweight $cw \equiv 1\, (4)$.
\end{lemma}
\begin{proof}
    In coweight $cw \equiv 1 \, (4)$, the $\textup{E}_\infty$-page of the \textbf{aAHSS}$(B_0^\mathbb{R}(1))$ is isomorphic to the a direct sum of coweight $cw \equiv 1 \, (4)$ in Atiyah--Hirzebruch filtration 0, coweight $cw \equiv 0 \, (4)$ in Atiyah--Hirzebruch filtration 2, and coweight $cw \equiv 3 \, (4)$ in Atiyah--Hirzebruch filtration 3 pieces. The differentials:
    \[d_2(b^n \cdot\tau h_1[2]) = b^nh_1 \cdot \tau h_1[0], \quad d_2( b^n\rho \cdot \tau h_1[2]) = b^n \rho h_1 \cdot \tau h_1[0], \quad n \geq 0\]
    and
    \[d_3(b^n\rho[3]) = b^n \cdot \tau h_1[0], \quad d_3(b^n\rho^2[3]) = b^n\rho \cdot \tau h_1[0], \quad n \geq 0\]
    ensure that Atiyah--Hirzebruch filtration 0 is trivial in coweight $cw \equiv 1 \, (4)$. The differentials
    \[d_1((\tau^2a)^n b^m[3]) = (\tau^2a)^nb^mh_0[2], \quad d_2((\tau^2a)^nb^m[2]) = (\tau^2a)^nb^mh_1[0], \quad n,m \geq 0\]
    ensure that Atiyah--Hirzebruch filtration 2 is trivial in coweight $cw \equiv 0 \, (4)$. Atiyah--Hirzebruch filtration 3 is trivial in coweight $cw \equiv 3 \, (4)$ since $\text{Ext}^{***}_{\euscr{A}(1)^\vee_\mathbb{R}}(\mathbb{M}_2^\mathbb{R})$ is trivial in this coweight.
\end{proof}

We organize the $\textup{E}_\infty$-page by coweight modulo 4 in the same way that we depicted $\text{Ext}^{***}_{\euscr{A}(1)^\vee_\mathbb{R}}(\mathbb{M}_2^\mathbb{R})$. As indicated by \cref{trivial cw1 ext b01}, the different Atiyah--Hirzebruch filtration pieces contribute different pieces to each coweight of the $\textup{E}_\infty$-page, suitably shifted by the degree of the cell in the filtration it is attached to.
Since $\text{Ext}^{***}_{\euscr{A}(1)^\vee_\mathbb{R}}(\mathbb{M}_2^\mathbb{R})$ is concentrated in coweights $cw \equiv 0, 1, \text{ and } 2 \, (4)$, for each Atiyah--Hirzebruch filtration piece there is a coweight on the $\textup{E}_\infty$-page which it does not contribute to.
To make this clear, we describe each coweight piece of the $\textup{E}_\infty$-page in terms of contributions from Atiyah--Hirzebruch filtration in \cref{coweight organization}.

\begin{table}[H]
\label{table-coweight contribution}
    \centering
    \setlength{\tabcolsep}{0.5em} 
    {\renewcommand{\arraystretch}{1.2}
    \begin{tabular}{|l||l|l|l|}
        \hline
       $\textup{E}_\infty$ coweight & Filtration 0 & Filtration 2 & Filtration 3 \\
       \hline
       \hline
       0 & 0 &3 & 2\\
       1 & 1 & 0& 3\\
       2 & 2 & 1& 0\\
       3 & 3 & 2& 1\\
       \hline
    \end{tabular}}
    \caption{Atiyah--Hirzebruch filtration coweight contributions in the \textbf{aAHSS}($B_0^\mathbb{R}(1))$.}    
    \label{coweight organization}
\end{table}

With this in mind, we are prepared to give charts for the $\textup{E}_\infty$-page. 
\Cref{R-b0(1) cw =0} depicts the $\textup{E}_\infty$-page in coweight $cw \equiv 0 \, (4)$. \Cref{R-b0(1) cw=2} depicts the $\textup{E}_\infty$-page in coweight $cw \equiv 2 \, (4)$. \Cref{R-b0(1) cw=3} depicts the $\textup{E}_\infty$-page in coweight $cw \equiv 3 \, (4)$. We use the same notation as previously used in this section, with some additions. A $\square$ represents $\mathbb{F}_2[\tau^4][\rho]$, where in contrast to a black $\blacksquare$, the $\rho$-tower does not support $h_1$-multiplication. For example, in coweight $cw \equiv 0 \, (2)$, there is a class $z$ in degree $(-1, 0)$ which is a $\rho$-multiple of the class in degree $(0,0)$. As the symbol at $(0,0)$ is a $\square$, the class $z$ is $h_1$-torsion.
Red lines indicate hidden extensions. A diagonal black $\nearrow$ indicates infinite $h_1$-multiplication on a $\rho$-tower which is compatible with an infinite $\rho$-tower off of a class which is $h_1$-divisible. For example, in coweight $cw \equiv 2 \, (4)$, is a relation 
\[h_1x = \rho y,\]
where $x$ is the class in degree $(3,1)$ which supports an infinite $h_1$-tower, and $y$ is the class in degree $(5,2)$ which supports an infinite $\rho$-tower.
This relation occurs in some $\mathbb{F}_2$ which is \emph{not} present in the current chart: the current $\mathbb{F}_2$ in degree $(4,2)$ is a part of an $h_0$-tower, and thus is not $h_1$-divisible. We suppress notation and indicate this relation by $\nearrow$ rather than drawing an extra $\bullet$.

\begin{figure}[h]
    \centering
    \includegraphics[scale=.75]{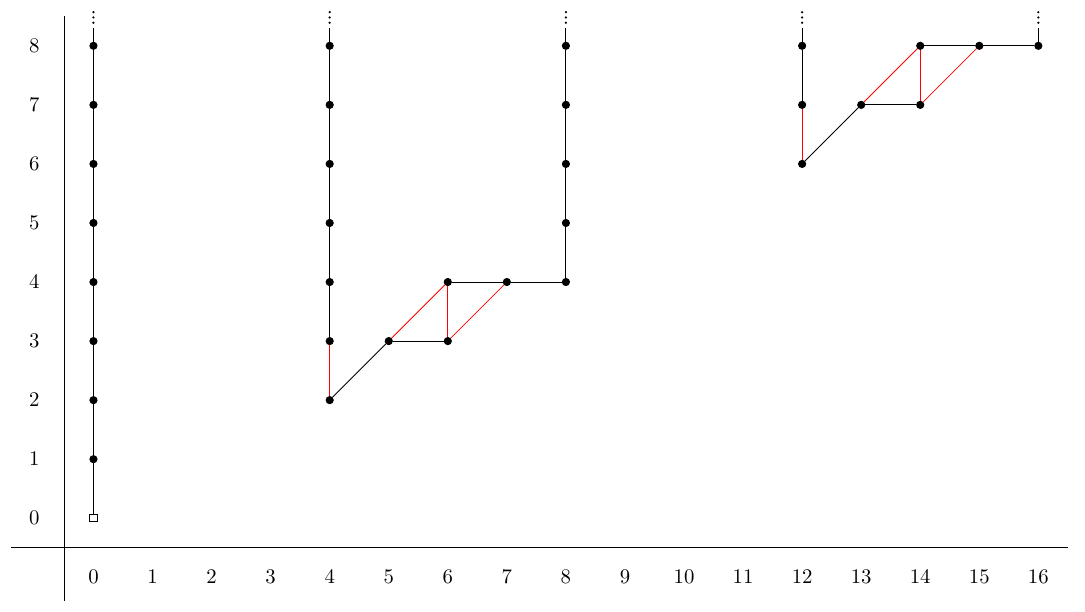}
    \caption{The $\textup{E}_\infty$-page of the \textbf{aAHSS}($B_0^\mathbb{R}(1))$ in $cw \equiv 0 \, (4)$.}
    \label{R-b0(1) cw =0}
\end{figure}

\begin{figure}[h]
    \centering
    \includegraphics[scale=.75]{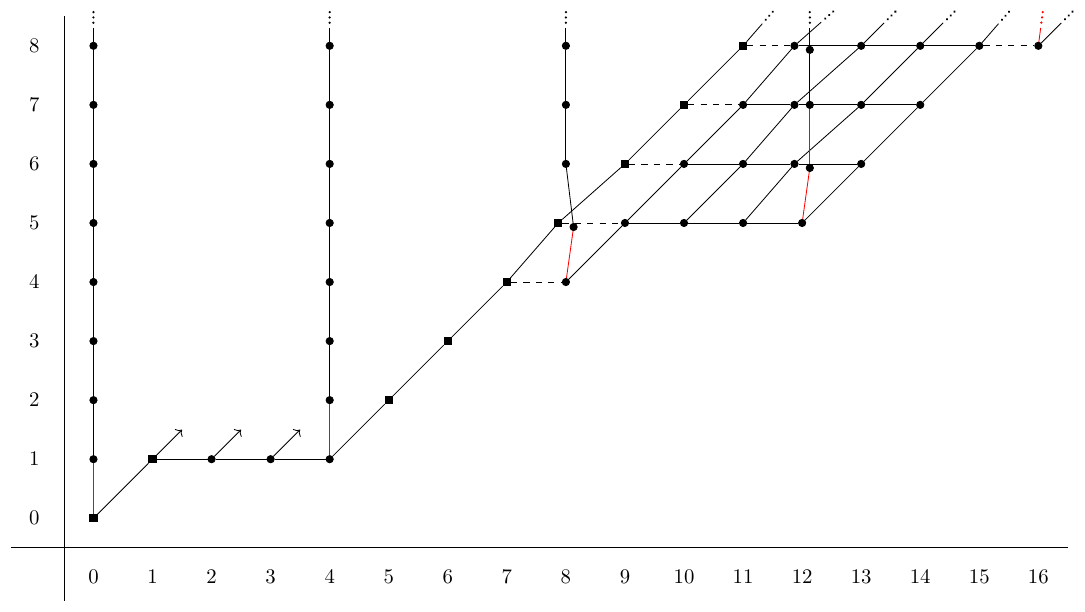}
    \caption{The $\textup{E}_\infty$-page of the \textbf{aAHSS}($B_0^\mathbb{R}(1))$ in $cw \equiv 2 \,(4)$.}
    \label{R-b0(1) cw=2}
\end{figure}

\begin{figure}[h]
    \centering
    \includegraphics[scale=.75]{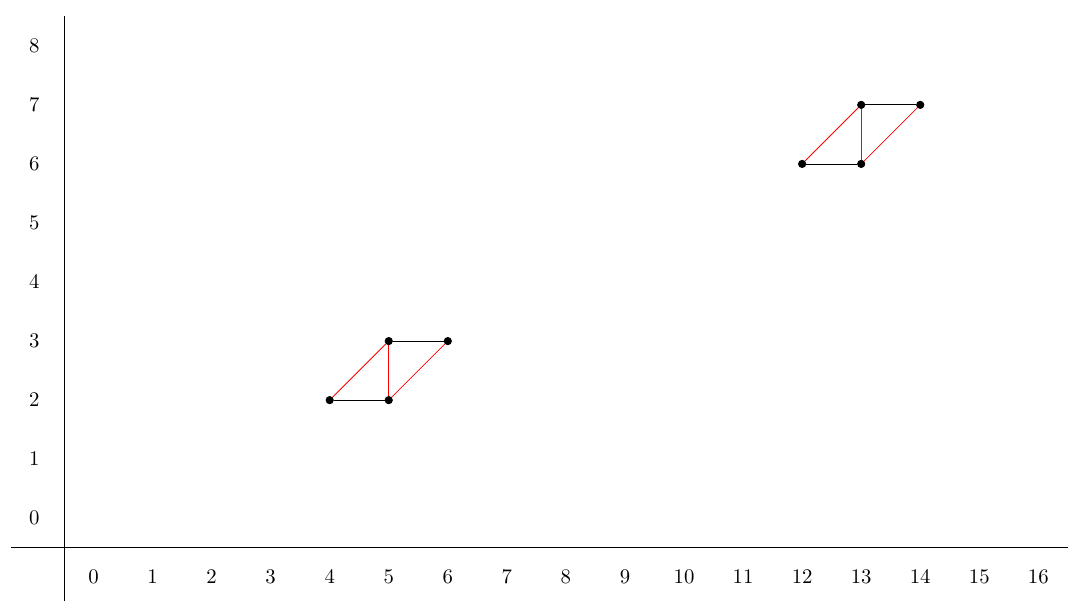}
    \caption{The $\textup{E}_\infty$-page of the \textbf{aAHSS}($B_0^\mathbb{R}(1))$ in $cw \equiv 3 \, (4)$.}
    \label{R-b0(1) cw=3}
\end{figure}

\begin{remark}
    The big flag $\euscr{F}_{4,1,2}$ is isomorphic to the coweight $cw \equiv 2 \, (4)$ piece of $\text{Ext}^{***}_{\euscr{A}(1)^\vee_\mathbb{R}}(B_0^\mathbb{R}(1))$ (see \cref{R-b0(1) cw=2}).
\end{remark}

We note that most classes on the $\textup{E}_\infty$-page are $v_1$-periodic. In fact, the only classes which are $v_1$-torsion are those in the $\rho$-tower in Adams filtration 0 of the coweight $cw \equiv 0 \, (4)$ piece which are divisible by $\rho^3$.

We close this section with the following result.
\begin{corollary}
\label{R-ksp}
    The $\textup{\textbf{mASS}}^{\mathbb{R}}(\textup{ksp})$ has signature
    \[\textup{E}_2^{s,f,w} = \textup{Ext}^{s,f,w}_{\euscr{A}(1)_\mathbb{R}^\vee}( B_0^\mathbb{R}(1)) \implies \pi^\mathbb{R}_{s,w}(\textup{ksp})\]
    and collapses on the $\textup{E}_2$-page.
\end{corollary}
\begin{proof}
    The proof is identical to the one given in \cref{C-ksp}.
\end{proof}

\subsection{$\text{Ext}^{***}_{\euscr{A}(1)^\vee_\mathbb{R}}(B_0^\mathbb{R}(1)^{\otimes 2})$ and recursive submodules}

To illustrate the numerous $v_1$-torsion classes which arise in the process of computing $\text{Ext}^{***}_{\euscr{A}(1)^\vee_\mathbb{R}}\left(B_0^\mathbb{R}(1)^{\otimes i}\right)$, which are ultimately to be ignored (see \Cref{v1 torsion}), we explicitly compute $\text{Ext}^{***}_{\euscr{A}(1)_\R^\vee}\left(B_0^\R(1)^{\otimes 2}\right)$. We form an algebraic Atiyah--Hirzebruch spectral sequence by applying the functor $\text{Ext}^{***}_{\euscr{A}(1)^\vee_\mathbb{R}}\left(B_0^\R(1) \otimes -\right)$ to \eqref{filtration:aAHSS}. This has signature
\[\textup{E}_1 = \text{Ext}^{***}_{\euscr{A}(1)^\vee_\R}(B_0^\R(1)) \otimes \mathbb{M}_2^\R\{[1], [\overline{\xi}_1], [\overline{\tau}_1]\} \implies \text{Ext}^{***}_{\euscr{A}(1)^\vee_\R}(B_0^\R(1)^{\otimes 2}).\]
The differentials are again induced by the algebraic cell structure of $B_0^\mathbb{R}(1)$. Thus, the $d_1$-differential is only nonzero on Atiyah--Hirzebruch filtration 3, where
\[d_1(\alpha[3]) = h_0\alpha[2],\]
the $d_2$-differential is only nonzero on Atiyah--Hirzebruch filtration 2, where
\[d_2(\alpha[2]) = h_1\alpha[0],\]
and the $d_3$-differential is only nonzero on Atiyah--Hirzebruch filtration 3, where
\[d_3(\alpha[3]) = \langle\alpha, h_0, h_1 \rangle[0].\]
Since the $d_1$- and $d_2$-differentials are given by a product with an element of $\text{Ext}_{\euscr{A}(1)^\vee_\mathbb{R}}^{***}(\mathbb{M}_2^\mathbb{R})$ of coweight 0, they preserve coweight between Atiyah-Hirzebruch filtrations. Similar to the previous case, the $d_3$-differential increases coweight by 1.

In fact, most of the work in running the $\textbf{aAHSS}(B_0^\mathbb{R}(1)^{\otimes 2})$ follows immediately from our previous computation of the $\textbf{aAHSS}(B_0^\mathbb{R}(1))$  in \Cref{subsection:Ext_B0R1}. By observation, we see that there is a submodule 
\[M_1 :=\Sigma^{4,2}\text{Ext}^{***}_{\euscr{A}(1)^\vee_\mathbb{R}}(\mathbb{M}_2^\mathbb{R})\langle 1\rangle  \subseteq \text{Ext}^{***}_{\euscr{A}(1)^\vee_\mathbb{R}}(B_0^\mathbb{R}(1)),\] 
using the notation from \cref{section 3}. This submodule is very large in $\text{Ext}_{\euscr{A}(1)^\vee_\mathbb{R}}^{***}(B_0^\mathbb{R}(1))$: in coweight $cw \equiv 0 \, (4)$, $M_1$ consists of everything except for the $h_1$-torsion $(\rho, h_0)$-tower in stem 0; in coweight $cw \equiv 2 \, (4)$, $M_1$ consists of everything except for the $(\rho, h_0)$-tower in stem 0; and in coweight $cw \equiv 3 \, (4)$, it consists of every nonzero element. 

As a result, there is a submodule of the $\textup{E}_1$-page of the $\textbf{aAHSS}(B_0^\mathbb{R}(1)^{\otimes 2})$ given by
\begin{equation}
\label{eq:summand for aahss b0r1_2}
M_1 \otimes \mathbb{M}_2^\mathbb{R}\{[1], [\overline{\xi}_1], [\overline{\tau}_1]\}.
\end{equation}
The differentials on this submodule are identical to the differentials in the $\textbf{aAHSS}(B_0^\mathbb{R}(1))$. Thus, one can depict the differentials in the $\textbf{aAHSS}(B_0^\mathbb{R}(1))$ on the summand \eqref{eq:summand for aahss b0r1_2} by shifting the charts given in \Cref{R-aAHSS(b_0(1)) E1 coweight 0} through \Cref{R-aAHSS(b_0(1)) E2 coweight 2} to the appropriate tridegree and coweight. The hidden extensions on this summand are the exact same as those in the $\textbf{aAHSS}(B_0^\mathbb{R}(1))$.

The differentials on the classes not in the $M_1 \otimes \mathbb{M}_2^\mathbb{R}\{[1], [\overline{\xi}_1], [\overline{\tau}_1]\}$ are straightforward to compute by the formulas given, as are hidden extensions. Similar to the previous case, the behavior of the differentials ensures that the $\textup{E}_\infty$-page is trivial in coweight $cw \equiv 1 \, (4)$. We present the $\textup{E}_\infty$ page, modulo $v_1$-torsion, below. \Cref{b0(1)^2 coweight 0} depicts the $\textup{E}_\infty$-page in coweight $cw \equiv 0 \, (4)$. \Cref{b0(1)^2 coweight 1} depicts the $\textup{E}_\infty$-page in coweight $cw \equiv 1 \, (4)$. \Cref{b0(1)^2 coweight 2} depicts the $\textup{E}_\infty$-page in coweight $cw \equiv 2 \, (4)$. Note that we do not present hidden extensions in red in these charts.

\begin{figure}[h]
    \centering
    \includegraphics[scale=.75]{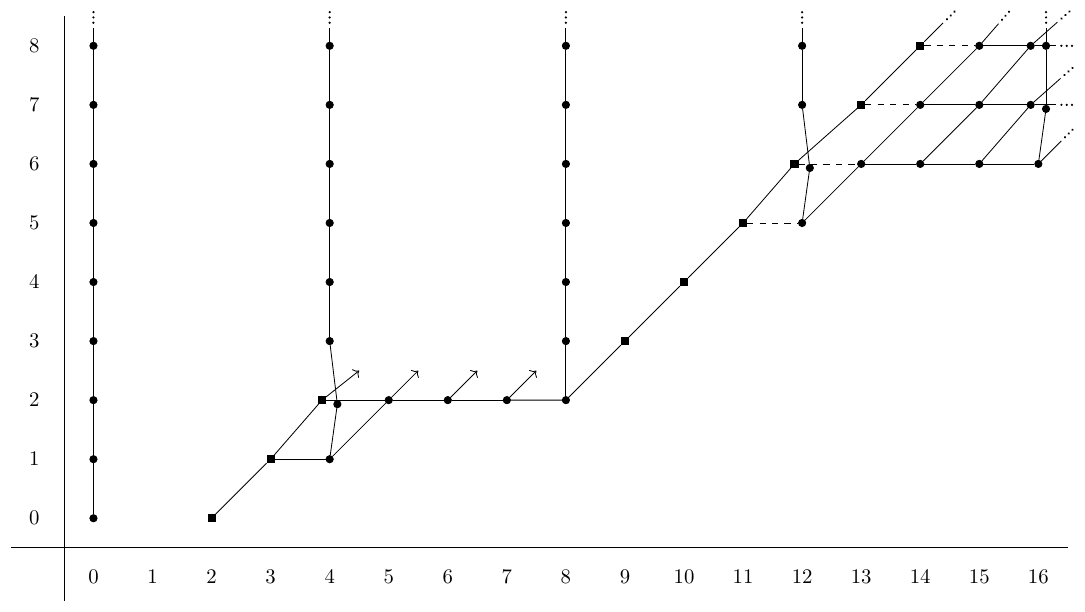}
    \caption{The $\textup{E}_\infty$-page of the \textbf{aAHSS}$\left(B_0^\mathbb{R}(1)^{\otimes 2}\right)$ in $cw \equiv 0 \, (4)$, modulo $v_1$-torsion.}
    \label{b0(1)^2 coweight 0}
\end{figure}
\begin{figure}[h]
    \centering
    \includegraphics[scale=.75]{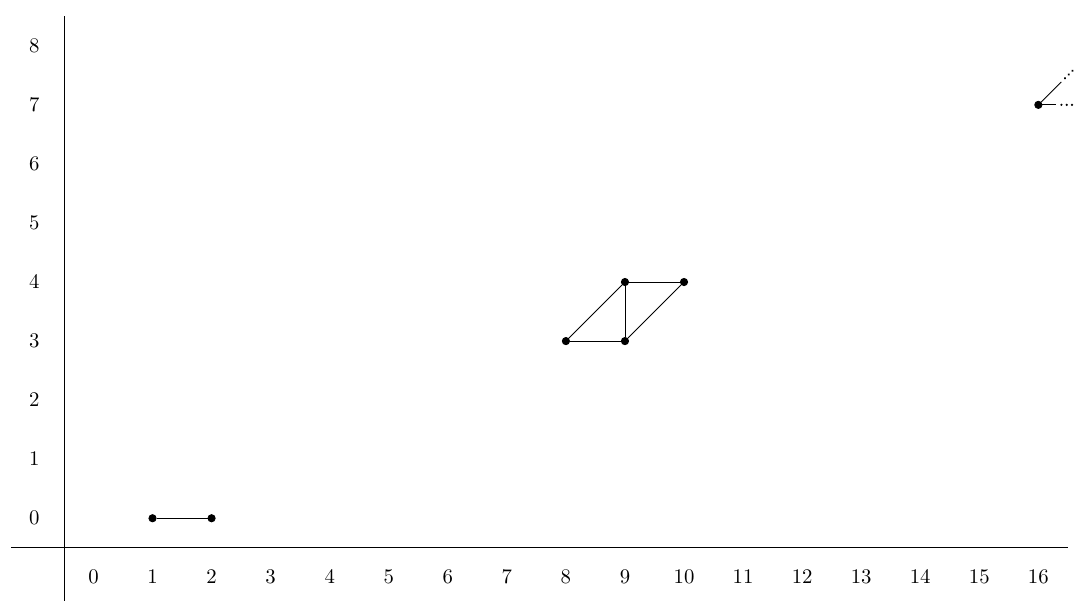}
    \caption{The $\textup{E}_\infty$-page of the \textbf{aAHSS}$\left(B_0^\mathbb{R}(1)^{\otimes 2}\right)$ in $cw \equiv 1 \, (4)$, modulo $v_1$-torsion.}
    \label{b0(1)^2 coweight 1}
\end{figure}
\begin{figure}[h]
    \centering
    \includegraphics[scale=.75]{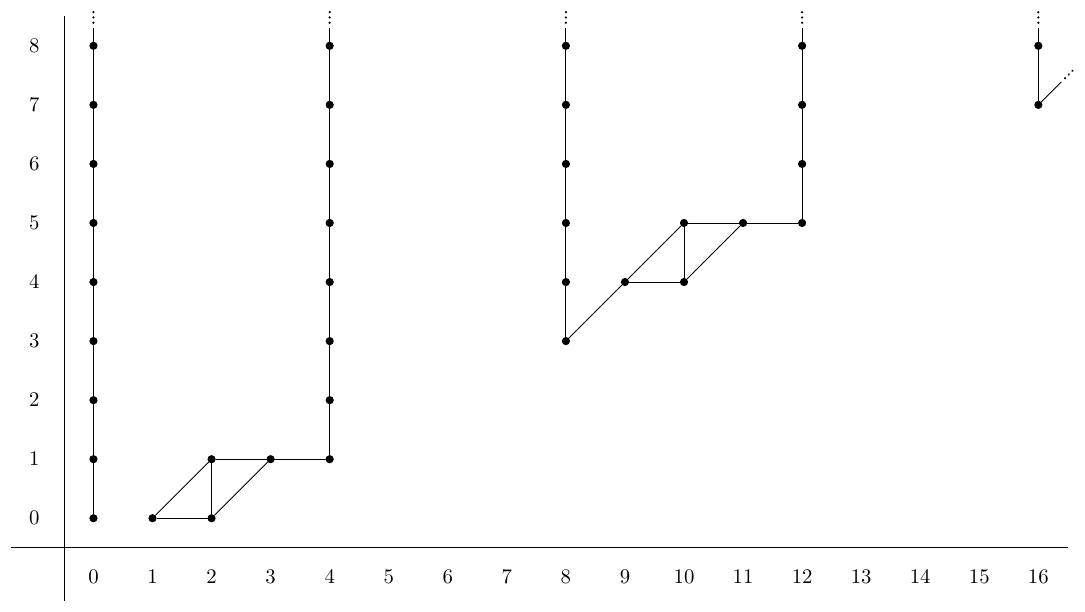}
    \caption{The $\textup{E}_\infty$-page of the \textbf{aAHSS}$\left(B_0^\mathbb{R}(1)^{\otimes 2}\right)$ in $cw \equiv 2 \, (4)$, modulo $v_1$-torsion.}
    \label{b0(1)^2 coweight 2}
\end{figure}

\begin{remark}
    The big flag $\euscr{F}_{8,2,4}$ is isomorphic to the summand of the coweight $cw \equiv 0 \, (4)$ piece of $\text{Ext}^{***}_{\euscr{A}(1)^\vee_\mathbb{R}}(B_0^\mathbb{R}(1)^{\otimes 2})$ concentrated in stems $s \geq 2$ (see \Cref{b0(1)^2 coweight 0}).
\end{remark}

Before giving a general formula for $\text{Ext}^{***}_{\euscr{A}(1)_\mathbb{R}^\vee}\left(B_0^\mathbb{R}(1)^{\otimes i}\right)$, we make a few simple observations. The group $\text{Ext}^{***}_{\euscr{A}(1)^\vee_\mathbb{R}}(\mathbb{M}_2^\mathbb{R})$, is only nonzero in coweights $cw \equiv 0, 1, 2\, (4)$. After running the $\textbf{aAHSS}(B_0^\mathbb{R}(1)$), we saw that there was a submodule 
\[M_1 :=\Sigma^{4,2}\text{Ext}^{***}_{\euscr{A}(1)^\vee_\mathbb{R}}(\mathbb{M}_2^\mathbb{R})\langle 1\rangle  \subseteq \text{Ext}^{***}_{\euscr{A}(1)^\vee_\mathbb{R}}(B_0^\mathbb{R}(1)),\] 
Note that if $x$ has coweight $cw \equiv t \, (4),$ then $\Sigma^{4,2}x\langle 1 \rangle$ has coweight $cw \equiv t +2 \, (4)$. In particular, $M_1$ is concentrated in coweights $cw \equiv 0, 2, 3 \, (4)$. By observation, we see that in fact all classes in $\text{Ext}^{***}_{\euscr{A}(1)^\vee_\mathbb{R}}(B_0^\mathbb{R}(1))$ are concentrated in these coweights. Similarly, after running the $\textbf{aAHSS}\left(B_0^\mathbb{R}(1)^{\otimes 2}\right)$ we see that there is a submodule
\[M_2 := \Sigma^{8,4}\text{Ext}^{***}_{\euscr{A}(1)^\vee_\mathbb{R}}(\mathbb{M}_2^\mathbb{R}) \langle 2 \rangle \subseteq \text{Ext}^{***}_{\euscr{A}(1)^\vee_\mathbb{R}}\left(B_0^\mathbb{R}(1)^{\otimes 2}\right).\]
It follows that $M_2$ is concentrated in coweights $cw \equiv 0,1,2 \, (4)$, and by observation we see that in fact all classes in $\text{Ext}^{***}_{\euscr{A}(1)^\vee_\mathbb{R}}\left(B_0^\mathbb{R}(1)^{\otimes 2}\right)$ are concentrated in these coweights.
More generally, from the $\textbf{aAHSS}\left(B_0^\mathbb{R}(1)^{\otimes i}\right)$ we see that the group $\text{Ext}^{***}_{\euscr{A}_\mathbb{R}^\vee}(B_0^\mathbb{R}(1)^{\otimes i})$ contains a submodule
\[M_i=\Sigma^{4i, 2i}\text{Ext}^{***}_{\euscr{A}(1)^\vee_\mathbb{R}}(\mathbb{M}_2^\mathbb{R})\langle i\rangle \subseteq \text{Ext}^{***}_{\euscr{A}(1)^\vee_\mathbb{R}}\left(B_0^\mathbb{R}(1)^{\otimes i}\right),\]
and that, modulo $v_1$-torsion, $\text{Ext}^{***}_{\euscr{A}(1)_\mathbb{R}^\vee}\left(B_0^\mathbb{R}(1)^{\otimes i}\right)$ is concentrated in the same coweights as $M_i$ is. More precisely, we have the following result.
\begin{proposition}
\label{coweight vanish}
    The group $\textup{Ext}^{***}_{\euscr{A}(1)^\vee_\mathbb{R}}\left(B_0^\mathbb{R}(1)^{\otimes i}\right)/v_1\textup{-torsion}$ is trivial in coweight $cw \equiv 1 \, (4)$ for $i \equiv 1 \, (2)$ and in coweight $cw \equiv 3 \, (4)$ for $i \equiv 0 \, (2)$.
\end{proposition}

We also have the following submodule of $\text{Ext}^{***}_{\euscr{A}(1)^\vee_\mathbb{R}}(B_0^\mathbb{R}(1)^{\otimes i}).$

\begin{proposition}
\label{nested}
    Let $i \geq 1$. For all $1 \leq j \leq i$, the group $\textup{Ext}^{***}_{\euscr{A}(1)^\vee_\mathbb{R}}\left(B_0^\mathbb{R}(1)^{\otimes i}\right)/v_1\textup{-torsion}$ contains as a submodule:
    \[\Sigma^{4i, 2i}\textup{Ext}^{***}_{\euscr{A}(1)^\vee_\mathbb{R}}\left(B_0^\mathbb{R}(1)^{\otimes j}\right)\langle i \rangle \subseteq \textup{Ext}^{***}_{\euscr{A}(1)^\vee_\mathbb{R}}\left(B_0^\mathbb{R}(1)^{\otimes i}\right).\]
\end{proposition}
\begin{proof}
    We have shown this to be true in the case of $i=1$ in \cref{R-ksp}. The differentials in the $\textbf{aAHSS}\left(B_0^\mathbb{R}(1)^{\otimes 2}\right)$ on the submodule $M_1$ are isomorphic to the differentials in the $\textbf{aAHSS}(B_0^\mathbb{R}(1))$, hence the submodule $M_2$ of $\text{Ext}^{***}_{\euscr{A}(1)^\vee_\mathbb{R}}\left(B_0^\mathbb{R}(1)^{\otimes 2}\right)$ extends to give a submodule
    \[\Sigma^{8,4}\text{Ext}^{***}_{\euscr{A}(1)^\vee_\mathbb{R}}\left(B_0^\mathbb{R}(1)\right)\langle 2 \rangle \subseteq \text{Ext}^{***}_{\euscr{A}(1)^\vee_\mathbb{R}}\left(B_0^\mathbb{R}(1)^{\otimes 2}\right).\]
    For arbitrary $i$, the submodule $M_{i-1} \subseteq \text{Ext}^{***}_{\euscr{A}(1)^\vee_\mathbb{R}}\left(B_0^\mathbb{R}(1)^{\otimes i-1}\right)$ extends to the submodule
    \[\Sigma^{4(i-1), 2(i-1)}\text{Ext}^{***}_{\euscr{A}(1)^\vee_\mathbb{R}}\left(B_0^\mathbb{R}(1)^{\otimes j}\right) \langle i-1 \rangle\]
    for each $1 \leq j \leq i-1$. The differentials in the $\textbf{aAHSS}\left(B_0^\mathbb{R}(1)^{\otimes i}\right)$ are isomorphic to the differentials in the $\textbf{aAHSS}\left(B_0^\mathbb{R}(1)^{\otimes j}\right)$ when restricted to this submodule, hence the result of the spectral sequence gives the desired submodules in the target.
\end{proof}
\begin{remark}
    It is important that we work modulo $v_1$-torsion to make the above claim. For example, by computing the $\textbf{aAHSS}\left(B_0^\mathbb{R}(1)^{\otimes 3}\right)$ one can observe that the group $\text{Ext}^{***}_{\euscr{A}(1)^\vee_\mathbb{R}}\left(B_0^\mathbb{R}(1)^{\otimes 3}\right)$ is nonzero in every coweight modulo 4. However, the only classes in coweight $cw \equiv 1 \, (4)$ are $v_1$-torsion for degree reasons. Moreover, these $v_1$-torsion classes are permament cycles of the $\textbf{aAHSS}\left(B_0^\mathbb{R}(1)^{\otimes 4}\right)$, so their impact on this inductive process is negligible.
\end{remark}

\subsection{$\text{Ext}^{***}_{\euscr{A}(1)^\vee_\mathbb{R}}\left(B_0^\mathbb{R}(1)^{\otimes i}\right)$}
To aid in our presentation of $\text{Ext}^{***}_{\euscr{A}(1)^\vee_\mathbb{R}} 
\left(B_0^\mathbb{R}(1)^{\otimes i}\right)$,
we use the notation from \cref{R-background and notation} to construct a family of $\text{Ext}^{***}_{\euscr{A}(1)_\mathbb{R}^\vee}(\mathbb{M}_2^\mathbb{R})$-modules which we organize by coweight modulo 4. As we will see, each module is concentrated in precisely 3 of these coweights, reflecting our observations from \cref{coweight vanish}.

\begin{definition}
    For each $i \geq 0$, let $Z_i^\mathbb{R}$ be the  $\text{Ext}^{***}_{\euscr{A}(1)^\vee_\mathbb{R}}(\mathbb{M}_2^\mathbb{R})$-module given by the direct sum of the columns of the corresponding table.
    
    For $i =4k$, let $Z_i^\mathbb{R}$ be the direct sum of the columns of \cref{Zi i=4k}:
    \begin{table}[H]
    \centering
    \setlength{\tabcolsep}{0.5em} 
    {\renewcommand{\arraystretch}{1.2}
    \begin{tabular}{|l|l|l|}
    \hline
    $cw \equiv 0 \, (4)$ & $cw \equiv 1 \, (4)$ & $cw \equiv 2 \, (4)$  \\
    \hline
    \hline
    $\euscr{F}_{4i, i,2i}$ &$\Sigma^{2i, i}\euscr{D}$ & $\Sigma^{2i, i}\euscr{S}$ \\
        
    $\oplus$ &  & $\oplus$ \\
        
    $\bigoplus_{j=1}^{2k}\Sigma^{4j-4, 4\lfloor(j+1)/2 \rfloor-4}H$   & & $\bigoplus_{j=1}^{2k}\Sigma^{4j-4, 4\lfloor j/2 \rfloor -2}H$   
     \\
    \hline
    \end{tabular}}
    \caption{$Z_i^\mathbb{R}$ for $i = 4k$}
    \label{Zi i=4k}
    \end{table}
    For $i =4k+1$, let $Z_i^\mathbb{R}$ be the direct sum of the columns of \cref{Zi i=4k+1}:   
    \begin{table}[H]
    \setlength{\tabcolsep}{0.5em} 
    {\renewcommand{\arraystretch}{1.2}
    \begin{tabular}{|l|l|l|}
    \hline
    $cw \equiv 0 \, (4)$ & $cw \equiv 2 \, (4)$ & $cw \equiv 3 \, (4)$  \\
    \hline
    \hline
    $\Sigma^{2i+2, i+1}\euscr{S}\langle 1\rangle$&    $\euscr{F}_{4i, i,2i}$ &$\Sigma^{2i+2, i+1}\euscr{D}\langle 1\rangle$  \\
        
    $\oplus$ & $\oplus$ &  \\
        
    $\bigoplus_{j=1}^{2k}\Sigma^{4j-4, 4\lfloor (j+1)/2 \rfloor-4}H$   &    $\bigoplus_{j=1}^{2k}\Sigma^{4j-4, 4\lfloor j/2 \rfloor-2}H$   & \\ 
    $\oplus $ & &\\
    $\Sigma^{2i+2, i-1} J$ & & \\
    \hline
    \end{tabular}}
    \caption{$Z_i^\mathbb{R}$ for $i = 4k+1$}
    \label{Zi i=4k+1}
    \end{table}
    For $i=4k+2$, let $Z_i^\mathbb{R}$ be the direct sum of the columns of \cref{Zi i=4k+2}: 
    \begin{table}[H]
    \centering
    \setlength{\tabcolsep}{0.5em} 
    {\renewcommand{\arraystretch}{1.2}
    \begin{tabular}{|l|l|l|}
    \hline
    $cw \equiv 0 \, (4)$ & $cw \equiv 1 \, (4)$ & $cw \equiv 2 \, (4)$  \\
    \hline
    \hline
    $\euscr{F}_{4i, i,2i}$ &$\Sigma^{2i+4, i+2}\euscr{D}\langle 2\rangle$ & $\Sigma^{2i+4, i+2}\euscr{S}\langle 2\rangle$ \\
        
    $\oplus$ & $\oplus$ & $\oplus$ \\
        
    $\bigoplus_{j=1}^{2k+1}\Sigma^{4j-4, 4\lfloor(j+1)/2 \rfloor-4}H$   &$\Sigma^{2i-2, i-1}T$&$\bigoplus_{j=1}^{2k+1}\Sigma^{4j-4, 4\lfloor j/2 \rfloor-2}H$\\  

    &  & $\oplus$ \\
    &     &  $\Sigma^{2i, i}JD \langle 1 \rangle$\\  
    \hline
    \end{tabular}}
    \caption{$Z_i^\mathbb{R}$ for $i = 4k+2$}
    \label{Zi i=4k+2}
    \end{table}    
    For $i=4k+3$, let $Z_i^\mathbb{R}$ be the direct sum of the columns of \cref{Zi i=4k+3}:
    \begin{table}[H]
    \setlength{\tabcolsep}{0.5em} 
    {\renewcommand{\arraystretch}{1.2}
    \begin{tabular}{|l|l|l|}
    \hline
    $cw \equiv 0 \, (4)$ & $cw \equiv 2 \, (4)$ & $cw \equiv 3 \, (4)$  \\
    \hline
    \hline
    $\Sigma^{2i-2, i-1}\euscr{S}\langle -1\rangle$&    $\euscr{F}_{4i, i,2i}$ &$\Sigma^{2i-2, i-1}\euscr{D}\langle -1\rangle$  \\
        
    $\oplus$ & $\oplus$ &  \\
        
    $\bigoplus_{j=1}^{2k+1}\Sigma^{4j-4, 4\lfloor j/2 \rfloor}H$ &    $\bigoplus_{j=1}^{2k+2}\Sigma^{4j-4, 4\lfloor(j+1)/2 \rfloor-2}H$   &  \\  
    \hline
    \end{tabular}}
    \caption{$Z_i^\mathbb{R}$ for $i = 4k+3$}
    \label{Zi i=4k+3}
    \end{table}    
\end{definition}
We are now prepared to state the $\mathbb{R}$-motivic analogue of \cref{C-b0(1)^i}.
\begin{thm}
\label{b_0(1) powers proof}
    There is an isomorphism of $\euscr{A}(1)^\vee_\mathbb{R}$-comodules and $\textup{Ext}^{***}_{\euscr{A}(1)^\vee_\mathbb{R}}(\mathbb{M}_2^\mathbb{R})$-modules:
    \[\frac{\textup{Ext}^{***}_{\euscr{A}(1)^\vee_\mathbb{R}}\left(B_0^\mathbb{R}(1)^{\otimes i}\right)}{v_1\textup{-torsion}} \cong Z^{\mathbb{R}}_i\]
\end{thm}

\begin{proof}
    The case of $i=1$ was computed in \cref{R-ksp}, and the case of $i=2$ is presented in \Cref{b0(1)^2 coweight 0}, \Cref{b0(1)^2 coweight 1}, and \Cref{b0(1)^2 coweight 2}. The general case is obtained by applying the functor $\text{Ext}^{***}_{\euscr{A}(1)^\vee_\mathbb{R}}(B_0^\mathbb{R}(1)^{\otimes i-1}\otimes -)$ to \eqref{filtration:aAHSS} and using the resulting $\textbf{aAHSS}(B_0^\mathbb{R}(1)^{\otimes i})$:
    \[\text{Ext}^{***}_{\euscr{A}(1)^\vee_\mathbb{R}}(B_0^\mathbb{R}(1)^{\otimes i-1}) \otimes \mathbb{M}_2^\mathbb{R}\{[1], [\bar{\xi}_1], [\bar{\tau}_1]\} \implies \text{Ext}^{***}_{\euscr{A}(1)^\vee_\mathbb{R}}(B_0^\mathbb{R}(1)^{\otimes i}).\]
    The differentials are determined by the attaching maps in the algebraic cell complex $B_0^\mathbb{R}(1)$ as before, hence \cref{aAHSS convergenece general} implies that $\textup{E}_4 = \textup{E}_\infty$. Hidden extensions may be obtained using complex Betti realization to the classical case, linearity over $\text{Ext}^{***}_{\euscr{A}(1)^\vee_\mathbb{R}}(\mathbb{M}_2^\mathbb{R})$, and the base change functor $-\otimes \mathbb{C}: \text{SH}(\mathbb{R}) \to \text{SH}(\mathbb{C})$ combined with the isomorphism from \cref{bg-comod R/rho = C}:
    \[\textup{Ext}^{***}_{\euscr{A}(1)^\vee_{\mathbb{R}}}(B_0^{\mathbb{R}}(k)^{\otimes i}/\rho) \cong \textup{Ext}^{***}_{\euscr{A}(1)^\vee_{\mathbb{C}}}(B_0^{\mathbb{C}}(k)^{\otimes i}).\]
    The calculations are very similar for all $i$. We give the full argument in the case of $4k+1$ and leave the rest to the reader. As an example, we depict $Z_5^\mathbb{R}$ in \Cref{Z_5 cw=0}, \Cref{Z5 cw=2}, and \Cref{Z5 cw=3}.
    
    Let $i = 4k$, and consider the $\textbf{aAHSS}(B_0^\mathbb{R}(1)^{\otimes 4k+1})$. By induction, the $\textup{E}_1$-page is isomorphic to 
    \[Z_{4k}^\mathbb{R} \otimes \mathbb{M}_2^\mathbb{R} \{[1], [\bar{\xi}_1], [\bar{\tau}_1]\}.\]
    Recall from \cref{table-coweight contribution} the contributions from each Atiyah--Hirzebruch filtration to each coweight of the $\textup{E}_\infty$-page of the spectral sequence. In particular, since $Z_{4k}^\mathbb{R}$ is trivial in coweight $cw \equiv 3 \, (4)$, the $\textup{E}_\infty$-page consist of summands from only two Atiyah--Hirzebruch filtration pieces in    
    coweights $cw \equiv 0, 1, 3 \,(4)$. We proceed by analyzing the $\textup{E}_\infty$-page of the $\textbf{aAHSS}\left(B_0^\mathbb{R}(1)^{\otimes 4k+1}\right)$ one coweight at a time.
    \newline
    
    \fbox{$\textup{E}_\infty$-\textbf{coweight $cw \equiv 0 \, (4)$}}\newline
    
    We begin by analyzing the coweight $cw \equiv 0 \, (4)$ portion of the spectral sequence. This  consists of classes coming from Atiyah--Hirzebruch filtration 0 in coweight $cw \equiv 0 \, (4)$ and from Atiyah--Hirzebruch filtration 3 in coweight  $cw \equiv 2 \, (4)$. There is a $d_1$-differential between the $h_0$-towers of Atiyah--Hirzebruch filtrations 2 and 3 in coweight $cw \equiv 2 \, (4)$. Letting $x_j$ be a generator for $\Sigma^{4j-4, 4\lfloor j/2 \rfloor -2}H$, this differential is determined by
    \[d_1(x_j [3]) = h_0 \cdot x_j[2].\]
    We depict this differential in \Cref{fig:mainproof:d1_Htowers_cw0}.
    \begin{figure}[H]
        \centering
        \includegraphics[scale=.75]{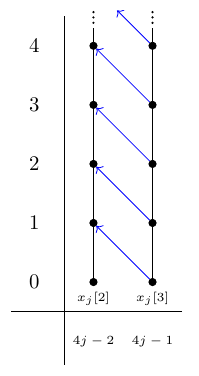}
        \caption{The $d_1$-differential between the $H$-towers from Atiyah--Hirzebruch filtrations 3 to 2 in coweight $cw \equiv 2 \, (4)$.}
        \label{fig:mainproof:d1_Htowers_cw0}
    \end{figure}    
    There is also a $d_1$-differential between big staircases $\Sigma^{2i, i}\euscr{S}$ in Atiyah--Hirzebruch filtrations 2 and 3 in coweight $cw \equiv 2 \, (4)$. The classes in each staircase in Atiyah--Hirzebruch filtration 3 which do not support a $d_1$-differential are those which do not support an $h_0$-tower. These classes cannot support any further differentials for degree reasons, and so they survive to the $\textup{E}_\infty$-page (see \Cref{R-aAHSS(b_0(1)) E1 coweight 2}).

    The $d_2$-differential between Atiyah--Hirzebruch filtrations 0 and 2 is given by the $h_1$-attaching map in $B_0^\mathbb{R}(1)$. Since the $h_0$-towers from Atiyah--Hirzebruch filtration 0 in coweight $cw\equiv 0 \, (4)$ are $h_1$-torsion, each tower survives to $\textup{E}_3$. None of these $h_0$-towers can be the target of a $d_3$-differential for degree reasons, so they also survive to the $\textup{E}_\infty$-page. 
    
    The $d_2$-differential between the big flags $\euscr{F}_{4i, i, 2i}$ in Atiyah--Hirzebruch filtrations 0 and 2 in coweight $cw \equiv 0 \, (4)$ is determined by the structure given in \cref{big flag def}. The classes in the big flag in Atiyah--Hirzebruch filtration 0 which survive to $\textup{E}_3$ are those classes which are not $h_1$-divisible. For $i=4k$, these are precisely the $h_0$-towers in stems $s \equiv i+4 \, (8)$, the $J$-towers in stems $s \equiv i \, (8)$ with Adams filtration $f \geq 1$, and the $(\rho,h_0)$-tower in Adams filtration $f=0$. For degree reasons, the $d_3$-differential is trivial on Atiyah--Hirzebruch filtration 0 in coweight $cw \equiv 0 \, (4)$ (see \Cref{R-aAHSS(b_0(1)) E2 coweight 0}).

    Thus, in coweight $cw \equiv 0 \, (4)$, the $\textup{E}_\infty$-page is given by the $h_0$-towers from Atiyah--Hirzebruch filtration 0 in coweight $cw \equiv 0\, (4)$ which do not come from the big flag: 
    \[\bigoplus_{j=1}^{k+1}\Sigma^{4j-4, 4 \lfloor (j+1)/2 \rfloor -4}H;\]
    the $h_0$-towers, $J$-towers and $(\rho, h_0)$-tower remaining from the big flag $\euscr{F}_{4i, i, 2i}$ in Atiyah--Hirzebruch filtration 0 in coweight $cw \equiv 0 \, (4)$; and the remaining classes from the big staircase $\Sigma^{2i, i}\euscr{S}$ in Atiyah--Hirzebruch filtration 3 in coweight $cw \equiv 2 \, (4)$. 
    
    Base change to $\mathbb{C}$ and linearity give hidden extensions between the $h_0$-towers and $J$-towers from Atiyah--Hirzebruch filtration 0 and the remaining classes from the big staircase in Atiyah--Hirzebruch filtration 3. To be precise, the classes in $\Sigma^{2i, i}\euscr{S}[3]$ in smallest stem and Adams filtration which survive to $\textup{E}_\infty$ are represented by 
    \[\Sigma^{2i, i}(h_1\cdot (\tau^2 h_0)) [3],\quad \Sigma^{2i, i}(h_1^2 \cdot (\tau^2 h_0)) [3],\quad \Sigma^{2i, i}(\rho a) [3],\] and their $\tau^4$-multiples. Let $x_{i-4}[0]$ denote the generator of $H \subseteq \euscr{F}_{4i, i, 2i}$ from Atiyah--Hirzebruch filtration 0 in lowest stem and Adams-filtration, and let $x_{i-5}[0]$ denote the generator of $J \subseteq \euscr{F}_{4i, i, 2i}$ from the same Atiyah--Hirzebruch filtration. Base change to $\mathbb{C}$ gives hidden extensions
    \[h_0 \cdot \Sigma^{2i, i}(h_1 \cdot (\tau^2 h_0))[3] = x_{i-4}[0], \quad h_1 \cdot \Sigma^{2i, i}(\rho a)[3] = \rho x_{i-5}[0].\]
    Then, linearity over $\text{Ext}^{***}_{\euscr{A}(1)^\vee_\mathbb{R}}(\mathbb{M}_2^\mathbb{R})$ gives us that
    \[h_1 \cdot \Sigma^{2i, i}(h_1^2 \cdot \tau^2 h_0)[3] =  h_1 \cdot \Sigma^{2i, i}(\rho^2 a)[3] = \rho \cdot h_1 \cdot \Sigma^{2i, i}(\rho a)[3] = \rho^2x_{i-5}[0]\]
    and
    \[h_0 \cdot \Sigma^{2i, i}(\rho a)[3] = h_0 \cdot \Sigma^{2i, i}(h_1(\tau h_1)^2)[3] = h_1 \cdot \Sigma^{2i, i}(h_0(\tau h_1)^2)[3] = h_1 \cdot \Sigma^{2i, i}(\rho^2 a)[3] = \rho^2 x_{i-5}[0].\]
    Multiplication by $b$ gives the same extensions throughout the $\textup{E}_\infty$-page. Altogether, this gives the summand
    \[\Sigma^{2i+4, i+2}\euscr{S} \langle 1 \rangle.\]
    For degree reasons, the elements in the $(\rho, h_0)$-tower from Atiyah--Hirzebruch filtration 0 which are divisible by $\rho^3$ must be $v_1$-torsion by inspection. Since our result is modulo this torsion, this gives the last summand
    \[\Sigma^{2i-2, i-1}J.\]
    The reader may benefit from consulting the chart in \Cref{R-b0(1) cw =0}.
    
    As an example, \Cref{Z_5 cw=0} depicts the coweight $cw \equiv 0 \, (2)$ portion of $Z_5^\mathbb{R}$. In this case, we have:
    \begin{itemize}
        \item the $h_0$-towers in stems 0 and 4 come from Atiyah--Hirzebruch filtration 0 in coweight $cw \equiv 0 \, (4)$;
        \item the $J$-tower in stem 8 is the $v_1$-torsion free component of the $(\rho, h_0)$-tower in Atiyah--Hirzebruch filtration 0;
        \item the staircase $\euscr{S}\langle 1\rangle$ in stem 12 consists of the $h_0$-towers and $J$-towers remaining from the big flag in Atiyah--Hirzebruch filtration 0 and the remainder of the big staircase from Atiyah--Hirzebruch filtration 3, together with hidden extensions.
    \end{itemize}
    \begin{figure}[H]
        \centering
        \includegraphics[scale=.5]{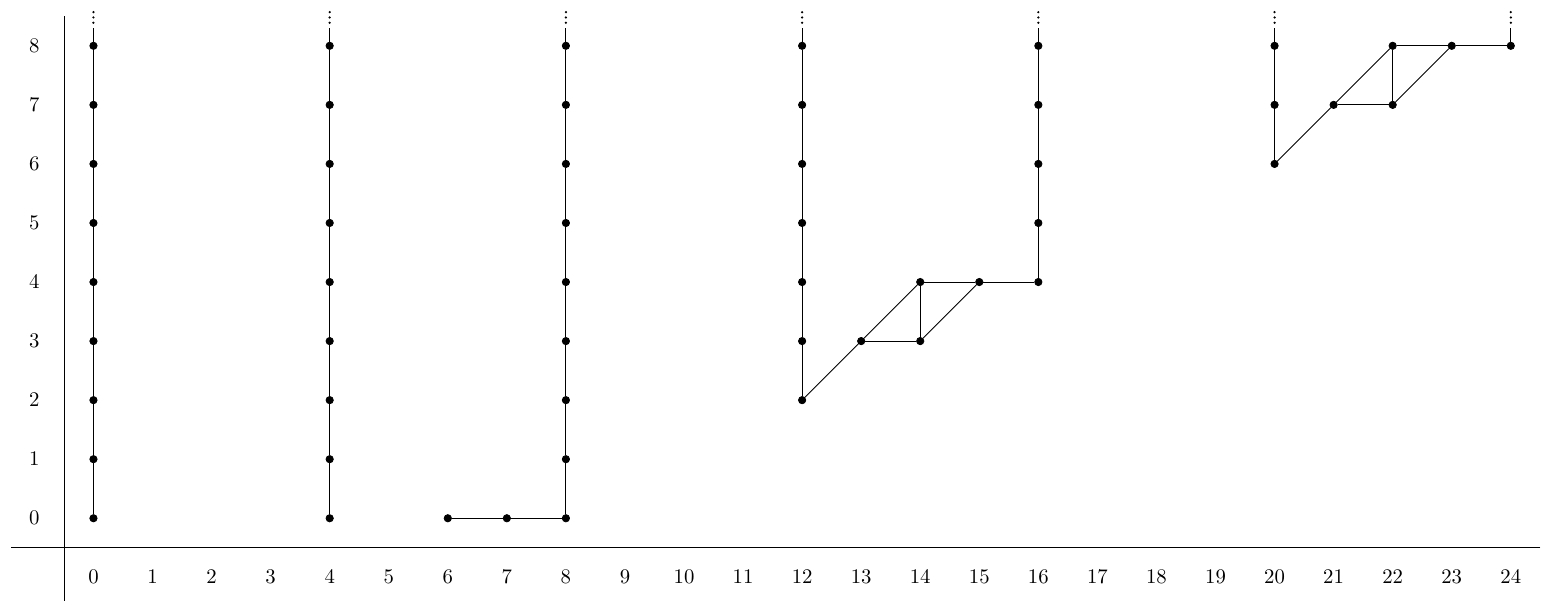}
        \caption{$Z_{5}^\mathbb{R}$ in $cw \equiv 0 \, (4)$.}
        \label{Z_5 cw=0}
    \end{figure}
    
    \fbox{$\textup{E}_\infty$-\textbf{coweight $cw \equiv 1 \, (4)$}}\newline

    The coweight $cw \equiv 1 \, (4)$ portion of the spectral sequence consists of classes coming from Atiyah--Hirzebruch filtration 0 in coweight $cw \equiv 1 \, (4)$ and from Atiyah--Hirzebruch filtration 2 in coweight $cw \equiv 0 \, (4)$. The $d_1$-differential between Atiyah--Hirzebruch filtrations 2 and 3 kills all of the classes from Atiyah--Hirzebruch filtration 2 in coweight $cw \equiv 0 \, (4)$ which are divisible by $h_0$ (see \Cref{R-aAHSS(b_0(1)) E1 coweight 0} for the differential on the big flag summands and \Cref{fig:mainproof:d1_Htowers_cw0} for the differentials on the $h_0$-towers). The $d_2$-differential between Atiyah--Hirzebruch filtrations 0 and 2 kills all of the classes from Atiyah--Hirzebruch filtration 2 in coweight $cw \equiv 0 \, (4)$ which support multiplication by $h_1$ (see \Cref{R-aAHSS(b_0(1)) E2 coweight 0}). By observation, we see that every class remaining from the big flag $\euscr{F}_{4i, i, 2i}$ on the $\textup{E}_2$-page supports $h_1$-multiplication. Moreover, the Atiyah--Hirzebruch on $B_0^\mathbb{R}(1)$ filtration forces $\textup{E}_3=\textup{E}_\infty$ on Atiyah--Hirzebruch filtration 2, so the only classes which survive are the bases of the $h_0$-towers 
    \begin{equation}
    \label{eq:cw1summand}
    \bigoplus_{j=1}^{2k}\Sigma^{4j-4, 4 \lfloor(j+1)/2 \rfloor -4}\mathbb{F}_2[\tau^4].
    \end{equation}
    
    The $d_2$-differential between Atiyah--Hirzebruch filtrations 0 and 2 kills all of the classes from Atiyah--Hirzebruch filtration 0 in coweight $cw \equiv 1 \, (4)$ which are divisible by $h_1$. This can be recovered by appropriately shifting the analogous $d_2$-differential in the $\textbf{aAHSS}(B_0^\mathbb{R}(1))$ (see \Cref{R-aahss(b_0(1)) E2 coweight 1}). The $d_3$-differential between Atiyah--Hirzebruch filtrations 0 and 3 is determined by the formula 
    \[d_3(\rho[3]) = (\tau h_1)[0]\]
    established in \cref{d3 lemma}. In particular, each remaining class from Atiyah--Hirzebruch filtration 0 in coweight $cw \equiv 1$ is the target of a differential. Similar to the $d_2$-differential just mentioned, this can be recovered by appropriately shifting the analogous $d_3$-differential in the $\textbf{aAHSS}(B_0^\mathbb{R}(1))$ (see \Cref{R-aAHSS(b0(1)) AHF=3 CW=0}). Thus, there are no contributions to the $\textup{E}_\infty$-page in coweight $cw \equiv 1 \, (4)$ from Atiyah--Hirzebruch filtration 0. For degree reasons, the classes in the summand \eqref{eq:cw1summand} are $v_1$-torsion, hence the $\textup{E}_\infty$-page is trivial in coweight $cw \equiv 1 \, (4)$.
    \newline
    
    \fbox{$\textup{E}_\infty$-\textbf{coweight $cw \equiv 2 \, (4)$}}\newline

    The coweight $cw \equiv 2 \, (4)$ portion of the spectral sequence consists of classes coming from Atiyah--Hirzebruch filtration 0 in coweight $cw \equiv 2 \, (4)$, from Atiyah--Hirzebruch filtration 2 in coweight $cw \equiv 1 \, (4)$, and from Atiyah--Hirzebruch filtration 3 in coweight $cw \equiv 0 \, (4)$.
    The $d_2$-differential between Atiyah--Hirzebruch filtrations 0 and 2 kills all of the classes from Atiyah--Hirzebruch filtration 0 in coweight $cw \equiv 2 \, (4)$ which are divisible by $h_1$. Notice that while the behavior on the big staircase $\Sigma^{2i, i}\euscr{S}$ can be recovered from the analogous differential in the $\textbf{aAHSS}(B_0^\mathbb{R}(1))$ (see \Cref{R-aAHSS(b_0(1)) E2 coweight 2}), the summand of $H$-towers
    \begin{equation}
    \label{eq:cw2Htowers}
    \bigoplus_{j=1}^{2k}\Sigma^{4j-4, 4 \lfloor j/2 \rfloor -2}H
    \end{equation}
    in Atiyah--Hirzebruch filtration 0 is entirely $h_1$-torsion, hence survives to $\textup{E}_3$. The $d_3$-differential between Atiyah--Hirzebruch filtration 0 in coweight $cw \equiv 2 \, (4)$ and Atiyah--Hirzebruch filtration 3 $cw \equiv 1 \, (4)$ can be recovered from the $\textbf{aAHSS}(B_0^\mathbb{R}(1))$ (see \Cref{R-aAHSS(b0(1)) AHF=3 CW=1}). Notice that there can be no $d_3$ differentials whose target are the $h_0$-towers in \eqref{eq:cw2Htowers} for degree reasons. This determines the contributions to the $\textup{E}_\infty$-page in coweight $cw \equiv 2 \, (4)$ coming from Atiyah--Hirzebruch filtration 0.

    The $d_1$-differential between Atiyah--Hirzebruch filtrations 2 and 3 in coweight $cw \equiv 1 \, (4)$ can be recovered from the $\textbf{aAHSS}(B_0^\mathbb{R}(1))$ (see \Cref{R-aAHSS(b_0(1)) E1  coweight 1}). Similarly, the $d_2$-differential between Atiyah--Hirzebruch filtrations 0 and 2 in coweight $cw \equiv 1 \, (4)$ can be recovered from the $\textbf{aAHSS}(B_0^\mathbb{R}(1))$ (see \Cref{R-aahss(b_0(1)) E2 coweight 1}). For degree reasons, we have that $\textup{E}_3=\textup{E}_\infty$ on Atiyah--Hirzebruch filtration 2, so this determines the contributions to the $\textup{E}_\infty$-page in coweight $cw \equiv 2 \, (4)$ from Atiyah--Hirzebruch filtration 2.

    The $d_1$-differential between Atiyah--Hirzebruch filtrations 2 and 3 kills all classes from Atiyah--Hirzebruch filtration 3 in coweight $cw \equiv 0 \, (4)$ which support $h_0$-multiplication. In particular, this kills the summand $\bigoplus_{j=1}^{2k}\Sigma^{4j-4, r \lfloor(j+1)/2 \rfloor -4}H$ and all $h_0$-towers in the big flag $\euscr{F}_{4i, i, 2i}$. The $d_3$-differential between Atiyah--Hirzebruch filtrations 0 and 3 is determined by the formula 
    \[d_3(\rho[3]) = (\tau h_1)[0]\]
    established in \cref{d3 lemma}. In particular, using the notation of \cref{big flag def}, we have (see \Cref{R-aAHSS(b0(1)) AHF=3 CW=0})
    \[d_3(\rho \cdot x_{4k}[3]) = (\Sigma^{2i, i}\tau h_1)[0].\]
    Linearity gives the rest of the differentials. Note that this is the differential which ensures that Atiyah--Hirzebruch filtration 0 is trivial in coweight $cw \equiv 1 \, (4)$ on the $\textup{E}_\infty$-page.

    Thus, in coweight $cw \equiv 2 \, (4)$, the $\textup{E}_\infty$-page is given by the remaining $h_0$-towers from the big staircase $\Sigma^{2i, i}\euscr{S}$ from Atiyah--Hirzebruch filtration 0 in coweight $cw \equiv 2 \, (4)$ and the $h_0$-towers 
    \[\bigoplus_{j=1}^{2k}\Sigma^{4j-4, 4\lfloor j/2 \rfloor -2}H\]
    which do not come from the big staircase from Atiyah--Hirzebruch filtration 0 in coweight $cw \equiv 2 \, (4)$; the remaining $\mathbb{F}_2[\tau^4]$'s from Atiyah--Hirzebruch filtration 2 in coweight $cw \equiv 1 \, (4)$; and the remaining submodule of the big flag $\euscr{F}_{4i, i, 2i}$ from Atiyah--Hirzebruch filtration 3 in coweight $cw \equiv 0 \, (4)$. Base change to $\mathbb{C}$ and linearity give the remaining hidden extensions, which we discuss now. 
    
    There is a hidden $h_0$-extension between each $\mathbb{F}_2[\tau^4]$ from Atiyah--Hirzebruch filtration 2 and the $h_0$-tower's concentrated in stems $s \equiv i+4 \, (8)$ from  Atiyah--Hirzebruch filtration 0. Additionally, there is a hidden $h_0$-extension between each $\mathbb{F}_2[\tau^4]$ from Atiyah--Hirzebruch filtration 2 and the remaining submodule of the big flag $\euscr{F}_{4i, i, 2i}$ from Atiyah--Hirzebruch filtration 3. For example, there is an $h_0$-extension 
    \[h_0 \cdot h_1x_{4k}[3]= \Sigma^{2i, i}(h_1 \cdot \tau h_1)[2].\]
    Finally, there is a hidden $h_0$-extension between the $h_0$-towers concentrated in stems $s \equiv i \, (8)$ from Atiyah--Hirzebruch filtration 0 and the remaining submodule of the big flag $\euscr{F}_{4i, i, 2i}$ from Atiyah--Hirzebruch filtration 3. For example, there is an $h_0$-extension 
    \[h_0 \cdot \rho^3x_{4k}[3] = \Sigma^{2i, i}(\tau^2h_0)[0].\]
    By inspection, we see that the hidden extensions assemble all of the remaining classes on the $\textup{E}_\infty$-page into the big flag
    \[\euscr{F}_{4i+4, 1+1, 2i+2}.\]
    The reader may benefit from consulting the chart in \Cref{R-b0(1) cw=2}.
    
    As an example, \Cref{Z5 cw=2} depicts the coweight $cw \equiv 2 \, (4) $ portion of $Z_5^\mathbb{R}$. In this case, we have:
    \begin{itemize}
        \item the $h_0$-towers in stems 0 and 4 come from the $h_0$-towers in Atiyah--Hirzebruch filtration 0 in coweight $cw \equiv 2 \, (4)$ which do not come from the big staircase;
        \item the big flag $\euscr{F}_{16, 4, 8}$ consists of the $h_0$-towers from the big staircase from Atiyah--Hirzebruch filtration 0 in coweight $cw \equiv 2 \, (4)$, the remaining $\mathbb{F}_2[\tau^4]$'s from Atiyah--Hirzebruch filtration 2 in coweight $cw \equiv 1 \, (4)$, and the remaining submodule of the big flag from Atiyah--Hirzebruch filtration 3 in coweight $cw \equiv 0 \, (4)$, together with hidden extensions.
    \end{itemize}
    \begin{figure}[h]
        \centering
        \includegraphics[scale=.5]{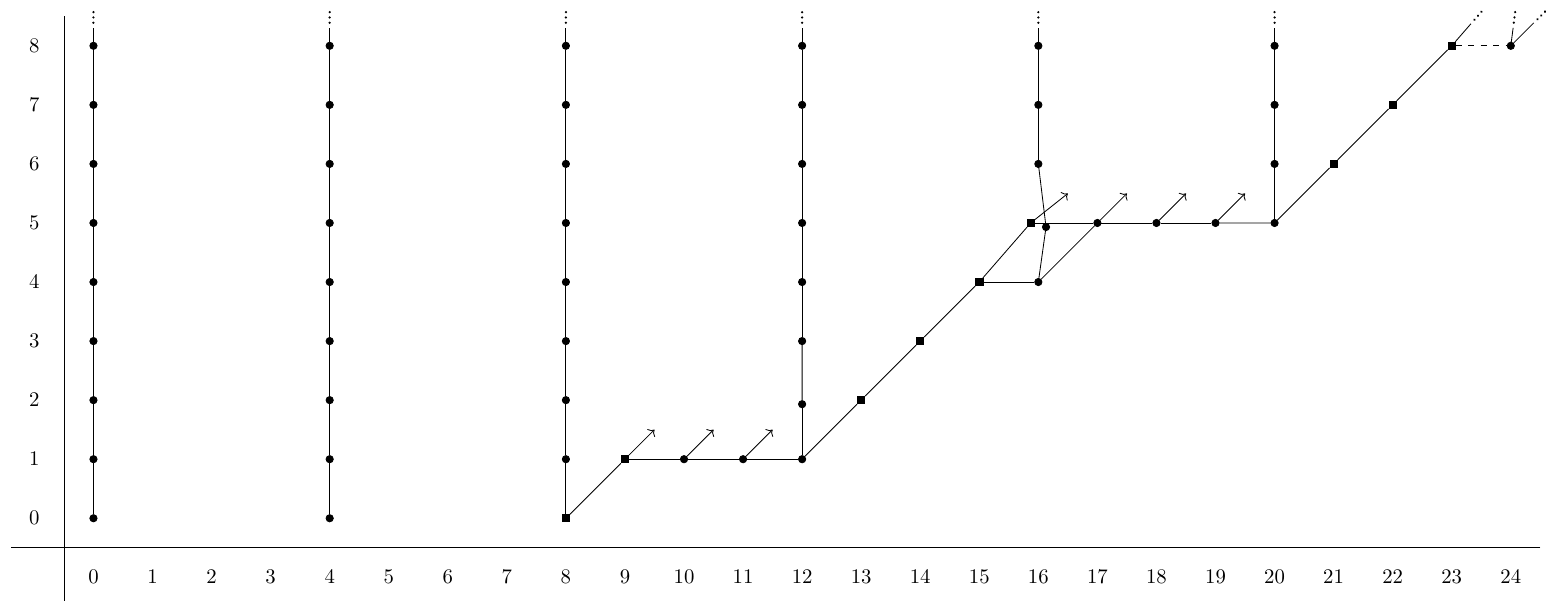}
        \caption{$Z_5^\mathbb{R}$ in $cw \equiv 2 \, (4)$.}
        \label{Z5 cw=2}
    \end{figure}

    \fbox{$\textup{E}_\infty$-\textbf{coweight} $cw \equiv 3 \, (4)$}\newline

    Finally, the coweight $cw \equiv 3 \, (4)$ portion of the spectral sequence consists of classes coming from Atiyah--Hirzebruch filtration 2 in coweight $cw \equiv 2 \, (4)$ and from Atiyah--Hirzebruch filtration 3 in coweight $cw \equiv 1 \, (4)$. The $d_1$-differential between Atiyah--Hirzebruch filtrations 2 and 3 kills all classes from Atiyah--Hirzebruch filtration 2 in coweight $cw \equiv 2 \, (4)$ which are $h_0$-divisible. In particular, only the bases of the $h_0$-towers $\bigoplus_{j=1}^{2k} \Sigma^{4j-4, 4\lfloor j/2 \rfloor -2}H$ survive, to the $\textup{E}_2$-page:
    \begin{equation}
    \label{eq:main:cw3}
    \bigoplus_{j=1}^{2k}\Sigma^{4j-4, 4 \lfloor j/2 \rfloor -2}\mathbb{F}_2[\tau^4],
    \end{equation}
    and only bases of the $h_0$-towers in the big staircase $\Sigma^{2i, i}\euscr{S}$, as well as the classes that are not $h_0$-divisible, survive (see \Cref{R-aAHSS(b_0(1)) E1 coweight 2} for the differential on the big staircase summands and \Cref{fig:mainproof:d1_Htowers_cw0} for the differential on the $h_0$-towers). The $d_2$-differential between Atiyah--Hirzebruch filtrations 0 and 2 kills all classes from Atiyah--Hirzebruch filtration 2 in coweight $cw \equiv 2 \, (4)$ which support $h_1$-multiplication. The classes in \eqref{eq:main:cw3} do not support $h_1$-multiplication, and neither do the classes $b^na$ and  $b^n\rho a$ remaining from the staircases (see \Cref{R-aAHSS(b_0(1)) E2 coweight 2}), so they survive to the $\textup{E}_3$-page. For degree reasons, we have that $\textup{E}_3=\textup{E}_\infty$ on Atiyah--Hirzebruch filtration 2, so this determines the contributions to the $\textup{E}_\infty$-page in coweight $cw \equiv 3 \, (4)$ from Atiyah--Hirzebruch filtration 2.

    The $d_1$-differential between Atiyah--Hirzebruch filtrations 2 and 3 kills all classes from Atiyah--Hirzebruch filtration 3 in coweight $cw \equiv 1 \, (4)$ which support $h_0$-multiplication. This can be recovered from the $\textbf{aAHSS}(B_0^\mathbb{R}(1))$ (see \Cref{R-aAHSS(b_0(1)) E1  coweight 1}). The $d_3$-differential between Atiyah--Hirzebruch filtrations 3 and 0 is only nonzero on the bottom left corner of the remnants of the diamond (see \Cref{R-aAHSS(b0(1)) AHF=3 CW=1}). This determines the contributions to the $\textup{E}_\infty$-page in coweight $cw \equiv 3 \, (4)$ from Atiyah--Hirzebruch filtration 3.

    Thus, in coweight $cw \equiv 1 \, (4)$, the $\textup{E}_\infty$-page is given by the segments $T$ coming from Atiyah--Hirzebruch filtrations 2 and 3 and the bases of the $h_0$-towers from Atiyah--Hirzebruch filtration 2. These segments are connected by hidden extensions in the following way. Let $\Sigma^{2i, i}(h_1(\tau h_1))[3]$ and $\Sigma^{2i, i}\rho h_1 (\tau h_1)[3]$ be the classes remaining from Atiyah--Hirzebruch filtration 3 in lowest Adams filtration, and let $\Sigma^{2i, i}a[2]$ and $\Sigma^{2i, i}\rho a [2]$ be the classes remaining from Atiyah--Hirzebruch filtration 2 in lowest Adams filtration. Base change to $\mathbb{C}$ gives a hidden extension
    \[h_1 \cdot \Sigma^{2i, i}( h_1 (\tau h_1))[3] = a[2].\]
    Linearity over $\text{Ext}^{***}_{\euscr{A}(1)^\vee_\mathbb{R}}(\mathbb{M}_2^\mathbb{R})$ gives us that
    \[h_1 \cdot \Sigma^{2i, i}(\rho h_1(\tau h_1))[3] = \rho a[2]\]
    and
    \[h_0 \cdot \Sigma^{2i, i}(h_1 (\tau h_1))[3] = h_1 \cdot \Sigma^{2i, i}(h_0 (\tau h_1))[3]=h_1 \cdot \Sigma^{2i, i}(\rho h_1 (\tau h_1))[3] = \rho a [2].\]
    Multiplication by $b$ gives the same extensions throughout the $\textup{E}_\infty$-page. Altogether, this gives the summand
    \[\Sigma^{2i+4, i+2} \euscr{D} \langle 1 \rangle.\]
    For degree reasons, the bases of the $h_0$-towers from Atiyah--Hirzebruch filtration 2 are $v_1$-torsion. The reader may benefit from consulting the chart in \Cref{R-b0(1) cw=3}.
    
    As an example, \Cref{Z5 cw=3} depicts the coweight $cw \equiv 3 \, (4)$ portion of $Z_5^\mathbb{R}$. In this case, we have only a $\euscr{D} \langle 1 \rangle$ in stem 13 which consists of the segments from Atiyah--Hirzebruch filtration 2 in coweight $cw \equiv 2 \, (4)$ and from Atiyah--Hirzebruch filtration 3 in coweight $cw \equiv 1 \, (4)$, together with hidden extensions.
    
    \begin{figure}[h]
        \centering
        \includegraphics[scale=.5]{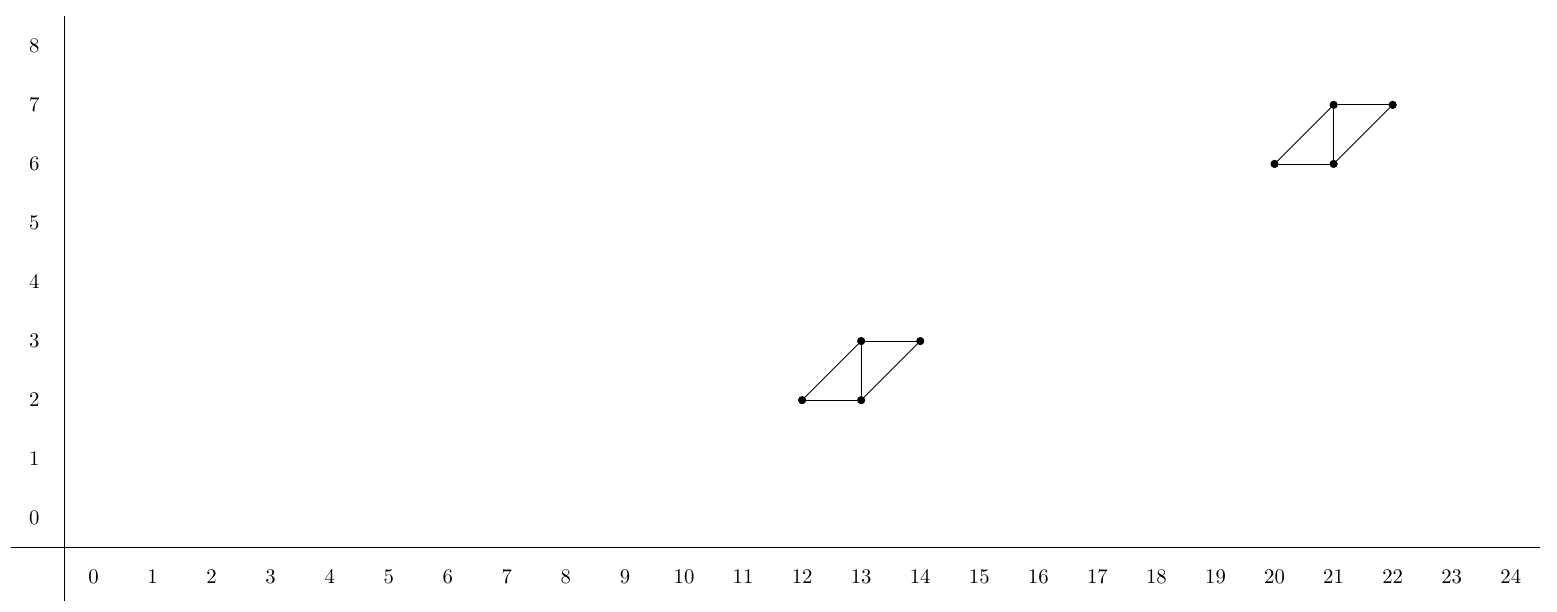}
        \caption{$Z_5^\mathbb{R}$ in $cw \equiv 3 \, (4)$.}
        \label{Z5 cw=3}
    \end{figure}

    Thus, we have shown that for $i=4k+1$, $\text{Ext}^{***}_{\euscr{A}(1)^\vee_\mathbb{R}}\left(B_0^\mathbb{R}(1)^{\otimes i}\right)/v_1\text{-torsion}$ is given by the direct sum of columns of \cref{Zi i=4k+1}. The other congruence classes are proved in a similar fashion.
\end{proof}

\begin{remark}
    It is possible that one can recover \cref{b_0(1) powers proof} by using the $\rho$-Bockstein spectral sequence and \cite[Lemma 3.36]{CQ21}. This spectral sequence has signature
    \[\textup{E}_1 = \text{Ext}^{***}_{\euscr{A}(1)^\vee_\mathbb{C}}\left(B_0^\mathbb{C}(1)^{\otimes i}\right)[\rho] \implies \text{Ext}^{***}_{\euscr{A}(1)^\vee_\mathbb{R}}\left(B_0^\mathbb{R}(1)^{\otimes i}\right).\]
    While we do not analyze this spectral sequence here, we remark that the extension problems we encounter throughout the algebraic Atiyah--Hirzebruch process are the same extension problems one encounters in the $\rho$-Bockstein spectral sequence.
\end{remark}

\section{The ring of cooperations}
\label{section-Rcoop}
In this section, we continue the inductive procedure and compute the groups $\text{Ext}^{***}_{\euscr{A}(1)^\vee_\mathbb{R}}(B_0^\mathbb{R}(k))$. Then, we assemble our results to compute the ring of cooperations $\pi_{**}^\mathbb{R}(\text{kq} \otimes \text{kq})$. As an application, we determine the structure of the $\text{E}_1$-page of the $\mathbb{R}$-motivic kq-resolution.

\subsection{$\text{Ext}^{***}_{\euscr{A}(1)^\vee_\mathbb{R}}(B_0^\mathbb{R}(k))$}
\label{section for b0R(k)}
The goal of this section is to compute $\text{Ext}^{***}_{\euscr{A}(1)^\vee_\mathbb{R}}(B_0^\mathbb{R}(k))$. We begin with the following lemma.
\begin{lemma}
\label{lemma about b1(k)}
    There is an isomorphism of $\euscr{A}(1)^\vee_\mathbb{R}$-comodules and $\textup{Ext}^{***}_{\euscr{A}(1)^\vee_\mathbb{R}}(\mathbb{M}_2^\mathbb{R})$-modules:
    \[\textup{Ext}^{***}_{\euscr{A}(0)^\vee_\mathbb{R}}(B_1^\mathbb{R}(k)) \cong \bigoplus_{j=0}^{k}\Sigma^{4j, 2j}\mathbb{F}_2[\rho, \tau^2, h_0]/(\rho h_0).\]
\end{lemma}

\begin{proof}
    Let $C_{\euscr{A}(0)_{\mathbb{R}}^\vee}(B_1^\mathbb{R}(k))$ denote the cobar complex computing $\text{Ext}^{***}_{\euscr{A}(0)^\vee_\mathbb{R}}(B_1^\mathbb{R}(k))$. This is a cosimplicial ring which takes the form
    \[\begin{tikzcd}
	{B_1^\mathbb{R}(k)} & {\euscr{A}(0)^\vee_\mathbb{R} \otimes B_1^\mathbb{R}(k)} & {\euscr{A}(0)^\vee_\mathbb{R} \otimes \euscr{A}(0)^\vee_\mathbb{R} \otimes B_1^\mathbb{R}(k)} & \cdots
	\arrow[from=1-1, to=1-2]
	\arrow[from=1-2, to=1-3]
	\arrow[from=1-3, to=1-4]
    \end{tikzcd}\]
    after totalization. Since the coaction of $\euscr{A}(0)^\vee_\mathbb{R}$ on $B_1^\mathbb{R}(k)$ fixes $\rho$, all of the maps in $C_{\euscr{A}(0)^\vee_\mathbb{R}}(B_1^\mathbb{R}(k))$ are $\rho$-linear. This implies that the filtration by powers of $\rho$
    \[C_{\euscr{A}(0)^\vee_\mathbb{R}}(B_1^\mathbb{R}(k)) \supset \rho \cdot C_{\euscr{A}(0)^\vee_\mathbb{R}}(B_1^\mathbb{R}(k)) \supset \rho^2 \cdot C_{\euscr{A}(0)^\vee_\mathbb{R}}(B_1^\mathbb{R}(k)) \supset \cdots\]
    is a filtration of chain complexes. Moreover, we have the following identifications
    \[\begin{array}{ccc}
        \mathbb{M}_2^\mathbb{R}/\rho \cong \mathbb{M}_2^\mathbb{C}, & \euscr{A}(0)^\vee_\mathbb{R}/\rho = \euscr{A}(0)^\vee_\mathbb{C}, &  B_1^\mathbb{R}(k)/\rho =  B_1^\mathbb{C}(k).
    \end{array} \]
    Thus, the filtration on $C_{\euscr{A}(0)^\vee_\mathbb{R}}(B_1^\mathbb{R}(k))$ gives a $\rho$-Bockstein spectral sequence (see \cref{comparison with C}) of the form:
    \[\textup{E}_1 = \text{Ext}^{***}_{\euscr{A}(0)^\vee_\mathbb{C}}( B_1^\mathbb{C}(k))[\rho] \implies \text{Ext}^{***}_{\euscr{A}(0)^\vee_\mathbb{R}}(B_1^\mathbb{R}(k)).\]
    The groups $\text{Ext}^{***}_{\euscr{A}(0)^\vee_\mathbb{C}}(B_1^\mathbb{C}(k))$ were calculated by Culver-Quigley \cite[Remark 3.25]{CQ21}. Modulo $v_1$-torsion, we have
    \[\text{Ext}^{***}_{\euscr{A}(0)^\vee_\mathbb{C}}(B_1^\mathbb{C}(k)) \cong \bigoplus_{j=0}^k\Sigma^{4j, 2j}\mathbb{M}^\mathbb{C}_2[h_0].\]
    In the case of $k=0$, we have that $B_1^\mathbb{C}(0) \cong \mathbb{M}_2^\mathbb{C}$, and this $\rho$-Bockstein spectral sequence was completely determined by Hill \cite[Theorem 3.1]{Hil11}. We can rewrite the $\textup{E}_1$-page as
    \[\textup{E}_1 = \text{Ext}^{***}_{\euscr{A}(0)^\vee_\mathbb{C}}(\mathbb{M}_2^\mathbb{C})[\rho] = \mathbb{M}_2^\mathbb{C}[h_0][\rho],\]
    with a differential $d_1(\tau) = \rho h_0$. The Liebniz rule immediately determines the end result:
    \[\text{Ext}^{***}_{\euscr{A}(0)^\vee_\mathbb{R}}(\mathbb{M}_2^\mathbb{R}) \cong \mathbb{F}_2[\rho, \tau^2, h_0]/(\rho h_0).\]
    More generally, the $\textup{E}_1$-page of the $\rho$-Bockstein takes the form
    \[\textup{E}_1 = \bigoplus_{j=0}^k \Sigma^{4j, 2j}\mathbb{M}_2^\mathbb{C}[h_0][\rho],\]
    and the differentials can be determined using the cobar complex. In particular, there are  differentials 
    \[d_1(\Sigma^{4j, 2j}\tau) = \Sigma^{4j, 2j}(\rho h_0)\]
    for each $0 \leq j \leq k$ on the classes $\tau$ in each $H$-tower in the decomposition of $\text{Ext}^{***}_{\euscr{A}(0)^\vee_\mathbb{C}}(B_1^\mathbb{C}(k))$. The Liebniz rule immediately determines the end result:
    \[\text{Ext}^{***}_{\euscr{A}(0)^\vee_\mathbb{R}}(B_1^\mathbb{R}(k)) \cong \bigoplus_{j=0}^k\Sigma^{4j, 2j}\mathbb{F}_2[\rho, \tau^2, h_0]/(\rho h_0),\]
    concluding the proof.
\end{proof}

\begin{figure}[h]
    \centering
    \includegraphics[scale=0.75]{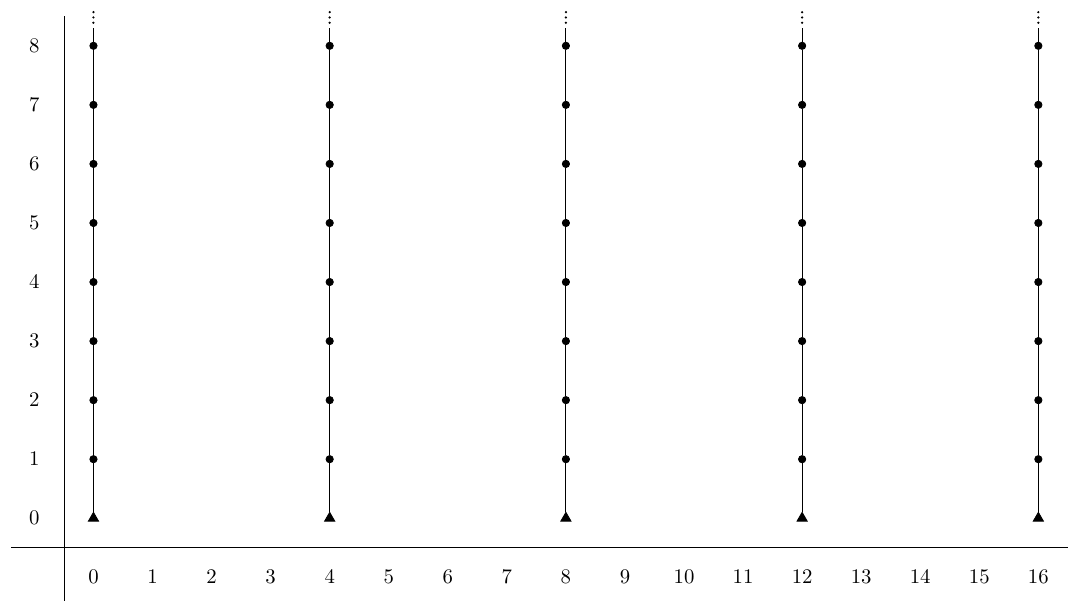}
    \caption{The group $\text{Ext}^{***}_{\euscr{A}(0)^\vee_\mathbb{R}}(B_1^\mathbb{R}(k)).$ A black $\blacktriangle$ denotes $\mathbb{F}_2[\rho, \tau^2]$, and a black $\bullet$ denotes $\mathbb{F}_2[\tau^2]$. A vertical black line represents multiplication by $h_0$.}
    \label{Ext over a(0)}
\end{figure}

We will now combine the results of \cref{b_0(1) powers proof} and the short exact sequences of motivic Brown--Gitler comodules from \Cref{ses bg}. The following result should be compared with \cite[Prop 2.6]{Mah81}, \cite[Prop 3.3]{BOSS19}, and \cref{C-b0(k) ext}.
\begin{thm}
\label{r-ext b0(k)}
    There is an isomorphism of $\euscr{A}(1)_\mathbb{R}^\vee$-comodules and $\textup{Ext}^{***}_{\euscr{A}(1)^\vee_\mathbb{R}}(\mathbb{M}_2^\mathbb{R})$-modules:
    \[\frac{\textup{Ext}^{***}_{\euscr{A}(1)^\vee_\mathbb{R}}(B_0^\mathbb{R}(k))}{v_1\textup{-torsion}} \cong \Sigma^{4k-4, 2k-2}Z_{\mathbb{\alpha}(k)}^\mathbb{R} \oplus \bigoplus_{j=0}^{4k-8}\left(\Sigma^{4j, 2j}H \oplus \Sigma^{4j, 2j-2}H\right) ,\]
    where $\alpha(k)$ is the number of 1's in the dyadic expansion of $k$. There are $\tau^2$-extensions between $\Sigma^{4j, 2j}H$ and $\Sigma^{4j, 2j-2}H.$
\end{thm}

\begin{proof}
    We will induct on $k$ and use the short exact sequences of \cref{ses bg}. For $k=1$, this was shown in \cref{R-ksp}. We assume now the theorem is true for all $i<k$ and divide the proof into two cases: when $k$ is even, or when $k$ is odd.

    Suppose that $k$ is even. \cref{ses bg} provides us with a short exact sequence of $\euscr{A}(1)^\vee_\mathbb{R}$-comodules:
    \begin{align}
    \label{ses R-proof even}
    0 \to \Sigma^{2k, k}B_0^\mathbb{R}(\tfrac{k}{2}) \to B_0^\mathbb{R}(k) \to B_1^\mathbb{R}(\tfrac{k}{2} -1) \otimes (\euscr{A}(1) \modmod \euscr{A}(0))_\mathbb{R}^\vee \to 0.
    \end{align}
    Applying $\text{Ext}^{***}_{\euscr{A}(1)^\vee_\mathbb{R}}( -)$ gives a long exact sequence of $\text{Ext}^{***}_{\euscr{A}(1)^\vee_\mathbb{R}}(\mathbb{M}_2^\mathbb{R})$-modules. Moreover, after killing $v_1$-torsion the connecting homomorphism is trivial, giving a short exact sequence whose middle term is $\text{Ext}^{***}_{\euscr{A}(1)^\vee_\mathbb{R}}(B_0^\mathbb{R}(k))$. Therefore, the Ext group in question decomposes into the Ext groups of the kernel and cokernel of the short exact sequence (\ref{ses R-proof even}).
    By the inductive hypothesis, we know that
    \[\text{Ext}^{***}_{\euscr{A}(1)^\vee_\mathbb{R}}\left(\Sigma^{2k, k}B_0^\mathbb{R}(\tfrac{k}{2})\right) \cong \Sigma^{2k, k}\left(\Sigma^{2k-4, k-2}Z_{\mathbb{\alpha}(k/2)}^\mathbb{R} \oplus \bigoplus_{j=0}^{2k-8}(\Sigma^{4j, 2j}H \oplus \Sigma^{4j, 2j-2}H)\right).\]
    A change of rings isomorphism gives an isomorphism for the right-hand side:
    \[\text{Ext}^{***}_{\euscr{A}(1)^\vee_\mathbb{R}}\left(B_1^\mathbb{R}(\tfrac{k}{2}-1) \otimes (\euscr{A}(1) \modmod \euscr{A}(0))^\vee_\mathbb{R}\right) \cong \text{Ext}^{***}_{\euscr{A}(0)^\vee_\mathbb{R}}\left(B_1^\mathbb{R}(\tfrac{k}{2}-1)\right).\]
    The group $\text{Ext}^{***}_{\euscr{A}(0)^\vee_\mathbb{R}}(B_1^\mathbb{R}(\tfrac{k}{2}-1))$ was calculated in \cref{lemma about b1(k)}. Thus the right-hand side and left-hand side assemble to give
    \begin{align*}
        \text{Ext}^{***}_{\euscr{A}(1)^\vee_\mathbb{R}}(B_0^\mathbb{R}(k)) \cong \Sigma^{4k-4, 2k,2}Z_{\alpha(k)}^\mathbb{R} &\oplus \bigoplus_{j=0}^{2k-8}\left(\Sigma^{4j+2k, 2j+2k}H \oplus \Sigma^{4j+4k, 2j-2+2k}H\right)   \\
         & \oplus \bigoplus_{j=0}^{k/2-1}\Sigma^{4j, 2j}\mathbb{F}_2[\rho, \tau^2, h_0](\rho h_0),
    \end{align*}
    using that $\alpha(\frac{k}{2}) = \alpha(k)$. Notice that for degree reasons, any $\rho$-divisible class coming from $\text{Ext}^{***}_{\euscr{A}(0)^\vee_\mathbb{R}}(B_1^\mathbb{R}(\tfrac{k}{2}-1))$ is $v_1$-torsion. Combining and reindexing the $H$-towers in the second and third summands gives the result.

    Suppose now that $k$ is odd. \cref{ses bg} provides us with a different short exact sequence of $\euscr{A}(1)^\vee_\mathbb{R}$-comodules:
    \begin{align}
    \label{ses R-proof odd}
        0 \to \Sigma^{2(k-1), k-1}B_0^\mathbb{R}(\tfrac{k-1}{2}) \otimes B_0^\mathbb{R}(1) \to B_0^\mathbb{R}(k) \to B_1^\mathbb{R}(\tfrac{k-1}{2}-1) \otimes (\euscr{A}(1) \modmod \euscr{A}(0))_\mathbb{R}^\vee \to 0.
    \end{align}
    Again, modulo $v_1$-torsion the connecting homomorphism in the long exact sequence obtained by applying $\text{Ext}^{***}_{\euscr{A}(1)^\vee_\mathbb{R}}(-)$ is trivial, giving a short exact sequence whose middle term is $\text{Ext}^{***}_{\euscr{A}(1)^\vee_\mathbb{R}}(B_0^\mathbb{R}(k)).$ We proceed by analyzing the Ext groups of the kernel and cokernel of the short exact sequence (\ref{ses R-proof odd}). A change of rings isomorphism gives the right-hand side:
    \[\text{Ext}^{***}_{\euscr{A}(1)^\vee_\mathbb{R}}\left(B_1^\mathbb{R}(\tfrac{k-1}{2}-1) \otimes (\euscr{A}(1) \modmod \euscr{A}(0))^\vee_\mathbb{R}\right) \cong \text{Ext}^{***}_{\euscr{A}(0)^\vee_\mathbb{R}}\left(B_1^\mathbb{R}(\tfrac{k-1}{2}-1)\right).\]
    The group $\text{Ext}^{***}_{\euscr{A}(0)^\vee_\mathbb{R}}\left(B_1^\mathbb{R}(\tfrac{k-1}{2}-1)\right)$ was calculated in \cref{lemma about b1(k)}. To calculate the left-hand side, we can use the $\textbf{aAHSS}\left(B_0^\mathbb{R}(\tfrac{k-1}{2}) \otimes B_0^\mathbb{R}(1)\right)$. This spectral sequence is obtained by applying the functor $\text{Ext}^{***}_{\euscr{A}(1)^\vee_\mathbb{R}}\left(B_0^\mathbb{R}(\tfrac{k-1}{2}) \otimes -\right)$ to \eqref{filtration:aAHSS} and has signature
    \[\textup{E}_1 = \text{Ext}^{***}_{\euscr{A}(1)^\vee_\mathbb{R}}\left(B_0^\mathbb{R}(\tfrac{k-1}{2})\right) \otimes \mathbb{M}_2^\mathbb{R}\{[1], [\bar{\xi}_1], [\bar{\tau}_1]\} \implies \text{Ext}^{***}_{\euscr{A}(1)^\vee_\mathbb{R}}\left(B_0^\mathbb{R}(\tfrac{k-1}{2}) \otimes B_0^\mathbb{R}(1)\right).\]
    By the inductive hypothesis, we have an isomorphism modulo $v_1$-torsion
    \begin{align}
    \label{aAHSS in R-B0k proof}
    \text{Ext}^{***}_{\euscr{A}(1)^\vee_\mathbb{R}}\left(B_0^\mathbb{R}(\tfrac{k-1}{2})\right) \cong \Sigma^{2k-6, k-3}Z_{\alpha(k)-1}^\mathbb{R} \oplus \bigoplus_{j=0}^{2k-10}\left(\Sigma^{4j, 2j}H \oplus \Sigma^{4j, 2j-2}H\right),
    \end{align}
    using that $\alpha(\tfrac{k-1}{2}) = \alpha(k)-1$. This splitting descends to a splitting of the spectral sequence, and we may analyze each summand of (\ref{aAHSS in R-B0k proof}) individually. The left-hand summand is handled by \cref{b_0(1) powers proof}, as it is isomorphic to the $\textbf{aAHSS}\left(\Sigma^{2k-6, k-3}B_0^\mathbb{R}(1)^{\otimes \alpha(k)}\right)$, so we have a summand of $\Sigma^{2k-2, k-1}Z_{\alpha(k)}^\mathbb{R}$ in $\text{Ext}^{***}_{\euscr{A}(1)^\vee_\mathbb{R}}\left(B_0^\mathbb{R}(\tfrac{k-1}{2}) \otimes B_0^\mathbb{R}(1)\right)$. The algebraic Atiyah--Hirzebruch spectral sequence on the right-hand summand collapses on $\textup{E}_2$. The remaining classes in Adams filtration 0 are $v_1$-torsion, and so the result is isomorphic to the original right-hand summand. Thus, we have an isomorphism:
    \begin{align*}
    &\Sigma^{2(k-1), k-1}\text{Ext}^{***}_{\euscr{A}(1)^\vee_\mathbb{R}}\left(B_0^\mathbb{R}(\tfrac{k-1}{2}) \otimes B_0^\mathbb{R}(1)\right)\\ 
    &\cong\Sigma^{2(k-1), k-1}\left( \Sigma^{2k-2, k-1}Z^\mathbb{R}_{\alpha(k)} \oplus \bigoplus_{j=0}^{2k-10}\left(\Sigma^{4j, 2j}H \oplus \Sigma^{4j, 2j-2}H\right)\right).
    \end{align*}
    Assembling the right-hand side and left-hand side gives
    \begin{align*}
        \text{Ext}^{***}_{\euscr{A}(1)^\vee_\mathbb{R}}(B_0^\mathbb{R}(k)) \cong \Sigma^{4k-4,2k-2} Z^\mathbb{R}_{\alpha(k)} &\oplus\bigoplus_{j=0}^{2k-10}\left(\Sigma^{4j+2(k-1), 2j+k-1}H \oplus \Sigma^{4j+2(k-1), 2j-2+k-1}H\right) \\
        &\oplus \bigoplus_{j=0}^{(k-1)/2-1}\Sigma^{4j, 2j}\mathbb{F}_2[\rho, \tau^2, h_0]/(\rho h_0).
    \end{align*}
    As before, for degree reasons, any $\rho$-divisible class coming from $\text{Ext}^{***}_{\euscr{A}(0)^\vee_\mathbb{R}}(B_1^\mathbb{R}\left(\tfrac{k-1}{2}-1)\right)$ is $v_1$-torsion. Combining and reindexing the $H$-towers in the second and third summands gives the result.
\end{proof} 

\begin{remark}
    The summand 
    \[\bigoplus_{j=0}^{4k-8}\left(\Sigma^{4j, 2j}H \oplus \Sigma^{4j, 2j-2}H\right)\]
    appearing in this decomposition is a ghost of the failure to express these Ext groups in terms of Adams covers of kq and ksp. It was shown in \cite{Hil11} that the motivic Adams spectral sequence for $\text{H}\mathbb{Z}$ collapses at the $\textup{E}_2$-page and has signature
    \[\textup{E}_2^{s,f,w} = \text{Ext}^{s,f,w}_{\euscr{A}(0)_\mathbb{R}^\vee}(\mathbb{M}_2^\mathbb{R}) \cong \mathbb{F}_2[\rho, \tau^2, h_0]/(\rho h_0) \implies \pi_{s,w}^\mathbb{R}(\text{H}\mathbb{Z}).\]
    In the Adams covers for kq and ksp, one sees summands of $\pi_{**}^\mathbb{R}(\text{H}\mathbb{Z})$. In the context of our calculations, we only obtain summands of $\pi_{**}^\mathbb{R}(\text{H}\mathbb{Z}/\rho)$.
\end{remark}

\subsection{The ring of cooperations}
In this subsection and the next, all results are implicitly computed modulo $v_1$-torsion. We assemble our results to compute the ring of cooperations $\pi_{**}^\mathbb{R}(\textup{kq} \otimes \textup{kq})$. First, our computation in \cref{r-ext b0(k)} gives the following:
\begin{corollary}
\label{describe e2}
    The $\textup{E}_2$-page of the \textup{\textbf{mASS}}$^{\mathbb{R}}(\textup{kq} \otimes \textup{kq})$ is given by
    \[\textup{E}_2^{s,f,w} = \bigoplus_{k \geq 0}\textup{Ext}^{s,f,w}_{\euscr{A}(1)_\mathbb{R}^\vee}(\Sigma^{4k, 2k}B_0^\mathbb{R}(k)) \cong \bigoplus_{k \geq 0}\left(\Sigma^{4k-4, 2k-2}Z_{\mathbb{\alpha}(k)}^\mathbb{R} \oplus \bigoplus_{j=0}^{4k-8}\left(\Sigma^{4j, 2j}H \oplus \Sigma^{4j, 2j-2}H\right)\right).\]
\end{corollary}
\begin{proof}
    The left hand equality, which is true even in the presence of $v_1$-torsion, is the content of \cref{e2 mass}. The right hand isomorphism follows from \cref{r-ext b0(k)}.
\end{proof}

We now arrive at our main result.

\begin{thm}
\label{main}
    The \textup{\textbf{mASS}}$^\mathbb{R}(\textup{kq} \otimes \textup{kq})$ collapses on the $\textup{E}_2$-page.
\end{thm}

\begin{proof}
    There is a base change functor
    \[-\otimes\mathbb{C}:\text{SH}(\mathbb{R}) \to \text{SH}(\mathbb{C})\]
    which induces a highly structured morphism from the $\textbf{mASS}^{\mathbb{R}}(\text{kq} \otimes \text{kq})$ to the $\textbf{mASS}^{\mathbb{C}}(\text{kq} \otimes \text{kq})$. The latter spectral sequence collapses on the $\text{E}_2$-page \cite[Corollary 3.43]{CQ21}. Complex Betti realization shows that there can be no differentials in the $\textbf{mASS}^{\mathbb{R}}(\textup{kq} \otimes \textup{kq})$ with both the source and target being $\tau$-free. However, by \cref{describe e2} we see that every class on the $\textup{E}_2$-page is $\tau^4$-periodic. Hence there can be no differentials and the spectral sequence collapses.
\end{proof}

\subsection{Application to the $\text{kq}$-resolution}
\label{section:kqres}
We conclude by examining the $\text{E}_1$-page of the kq-resolution. Recall that this spectral sequence has signature
\[\textup{E}_1^{s,f,w} = \pi_{s+f, w}^{\mathbb{R}}(\text{kq} \otimes \overline{\text{kq}}^{\otimes f}) \implies \pi_{s,w}^{\mathbb{R}}(\mathbb{S}).\]
We can determine each filtration line of the $\textup{E}_1$-page by an extension of the techniques used for the computation of $\pi_{**}^{\mathbb{R}}(\text{kq} \otimes \text{kq}).$ For each $n\geq 0$, the $\textbf{mASS}^{\mathbb{R}}(\text{kq} \otimes \overline{\text{kq}}^{\otimes n})$ takes the form
\[\text{E}^{s,f,w}_2 = \text{Ext}^{s,f,w}_{\euscr{A}^\vee_{\mathbb{R}}}(\text{H}_{**}(\text{kq} \otimes \overline{\text{kq}}^{\otimes n})) \implies \pi_{s,w}^{\mathbb{R}}(\text{kq} \otimes \overline{\text{kq}}^{\otimes n}).\]
We begin by decomposing this $\textup{E}_2$-page.

\begin{lemma}
\label{lemma:KunnethForkqRes}
    There is a K\"unneth isomorphism
    \[\textup{H}_{**}(\textup{kq} \otimes \overline{\textup{kq}}^{\otimes n}) \cong \textup{H}_{**}(\textup{kq}) \otimes \textup{H}_{**}(\overline{\textup{kq}})^{\otimes n}.\]
\end{lemma}

\begin{proof}
    We induct on $n$. The K\"unneth spectral sequence \cite{DI05} takes the form
    \[\text{E}_2 = \text{Tor}^{\mathbb{M}_2^{\mathbb{R}}}(\textup{H}_{**}(\textup{kq}), \textup{H}_{**}(\overline{\textup{kq}})) \implies \textup{H}_{**}(\textup{kq} \otimes \overline{\textup{kq}}). \] 
    Since $\text{H}_{**}(\text{kq}) = (\euscr{A} \modmod \euscr{A}(1))^\vee$ is the $\mathbb{M}_2^{\mathbb{R}}$-linear dual of a finitely-generated free $\mathbb{M}_2^{\mathbb{R}}$-module, it is also free. Thus all higher Tor terms vanish, implying that the spectral sequence collapses. 
    
    Suppose the result is true for all $i <n$. We again have a  K\"unneth spectral sequence which takes the form
    \[\text{E}_2 = \text{Tor}^{\mathbb{M}_2^{\mathbb{R}}}(\textup{H}_{**}(\textup{kq} \otimes \overline{\text{kq}}^{\otimes n-1}), \textup{H}_{**}(\overline{\textup{kq}})) \implies \textup{H}_{**}(\textup{kq} \otimes \overline{\textup{kq}}^{\otimes n}). \]
    Note that $\text{H}_{**}(\overline{\text{kq}})$ is also free over $\mathbb{M}_2^{\mathbb{R}}$, which one can see by the long exact sequence in homology associated to the defining cofiber sequence
    \[\mathbb{S} \to \text{kq} \to \overline{\text{kq}}.\]
    Thus the higher Tor terms vanish, implying that the spectral sequence collapses. This gives an isomorphism
    \[\text{H}_{**}(\text{kq} \otimes \overline{\text{kq}}^{\otimes {n-1}}) \otimes \text{H}_{**}(\overline{\text{kq}}) \cong \text{H}_{**}(\text{kq} \otimes \overline{\text{kq}}^{\otimes n}).\]
    By induction, we have a K\"unneth isomorphism on the left hand factor, finishing the proof.
\end{proof}

Note that there is an isomorphism of $\euscr{A}(1)^\vee$-comodules due to \Cref{kq bar homology}
\begin{equation}
\label{eq:homologykqBar}
    \text{H}_{**}(\overline{\text{kq}}) \cong \bigoplus_{k \geq 1}\Sigma^{4k, 2k}B_0^{\mathbb{R}}(k).
\end{equation}
We can now deduce the $\textup{E}_2$-page of the $\textbf{mASS}^{\mathbb{R}}(\text{kq} \otimes \overline{\text{kq}}^{\otimes n})$.

\begin{proposition}
\label{prop:n-lineE2}
    The $\textup{E}_2$-page of the $\textup{\textbf{mASS}}^{\mathbb{R}}(\textup{kq} \otimes \overline{\textup{kq}}^{\otimes n})$ takes the form
    \[\textup{E}_2^{s,f,w}\cong \bigoplus_{K \in \euscr{K}_n}\Sigma^{4|K|, 2|K|}\textup{Ext}^{s,f,w}_{\euscr{A}(1)^\vee_{\mathbb{R}}}(B_0^{\mathbb{R}}(K)),\]
    where $\euscr{K}_n = \{K = (k_1, \dots, k_n): k_j \geq 1 \textup{ for all } j\}$, $|K| = \sum_{j=1}^nk_j$, and $B_0^{\mathbb{R}}(K) = \bigotimes_{j=1}^nB_0^{\mathbb{R}}(k_j).$        
\end{proposition}

\begin{proof}
    By \Cref{lemma:KunnethForkqRes} and the change of rings isomorphism, we may rewrite the $\textup{E}_2$-page as
    \[\text{Ext}_{\euscr{A}^\vee_{\mathbb{R}}}^{***}(\text{H}_{**}(\text{kq}) \otimes \text{H}_{**}(\overline{\text{kq}})^{\otimes n} ) \cong \text{Ext}^{***}_{\euscr{A}(1)^\vee_{\mathbb{R}}}(\text{H}_{**}(\overline{\text{kq}})^{\otimes n}).\]
    The identification of \eqref{eq:homologykqBar} allows us to rewrite the first factor of $\text{H}_{**}(\overline{\text{kq}})$ in terms of Brown--Gitler comodules, leaving us with
    \[\bigoplus_{k_1 \geq 1}\Sigma^{4k_1, 2k_1}\text{Ext}_{\euscr{A}(1)^\vee_{\mathbb{R}}}^{***}(B_0^{\mathbb{R}}(k_1) \otimes \text{H}_{**}(\text{kq})^{\otimes n-1}).\]
    Rewriting the next factor of $\text{H}_{**}(\overline{\text{kq}})$ gives us
    \[\bigoplus_{k_1 \geq 1}\Sigma^{4k_1, 2k_1}\left(\bigoplus_{k_2 \geq 1}\Sigma^{4k_2, 2k_2} \text{Ext}_{\euscr{A}(1)^\vee_{\mathbb{R}}}^{***}(B_0^{\mathbb{R}}(k_1) \otimes B_0^{\mathbb{R}}(k_2) \otimes \text{H}_{**}(\overline{\text{kq}})^{\otimes n-2})\right),\]
    which we may rewrite as
    \[\bigoplus_{k_1, k_2 \geq 1}\Sigma^{4(k_1+k_2), 2(k_1+k_2)}\text{Ext}_{\euscr{A}(1)^\vee_{\mathbb{R}}}^{***}(B_0^{\mathbb{R}}(k_1) \otimes B_0^{\mathbb{R}}(k_2) \otimes \text{H}_{**}(\overline{\text{kq}})^{\otimes n-2}).\]
    The result follows by continuing in this manner and rewriting all factors of $\text{H}_{**}(\overline{\text{kq}})$ in terms of Brown--Gitler comodules.
\end{proof}

Note that we have described the $\text{E}_2$-page of the $\textbf{mASS}^{\mathbb{R}}(\text{kq} \otimes \overline{\text{kq}}^{\otimes n})$ as a module over $\text{Ext}^{***}_{\euscr{A}(1)^\vee_{\mathbb{R}}}(\mathbb{M}_2^\mathbb{R})$, hence as a module over the $\text{E}_2$-page of the $\textbf{mASS}^{\mathbb{R}}(\text{kq})$. 

\begin{thm}
\label{thm:n-lineDifs}
    The $\textup{\textbf{mASS}}^{\mathbb{R}}(\textup{kq} \otimes \overline{\textup{kq}}^{\otimes n})$ collapses on the $\textup{E}_2$-page.
\end{thm}

\begin{proof}
    Base change to $\mathbb{C}$ induces a highly structured morphism from the $\textbf{mASS}^{\mathbb{R}}(\text{kq} \otimes \overline{\text{kq}}^{\otimes n})$ to the $\textbf{mASS}^{\mathbb{C}}(\text{kq} \otimes \overline{\text{kq}}^{\otimes n})$. The latter spectral sequence collapses on the $\text{E}_2$-page \cite[Section 4.2]{CQ21}. Complex Betti realization shows that there can be no differentials in the $\textbf{mASS}^{\mathbb{R}}(\textup{kq} \otimes \textup{kq})$ with both the source and target being $\tau$-free. By the description of the $\text{E}_2$-page in \Cref{prop:n-lineE2} and the corresponding Ext groups in \Cref{r-ext b0(k)}, we see that every class on the $\textup{E}_2$-page is $\tau^4$-periodic. Hence there can be no differentials and the spectral sequence collapses.
\end{proof}

\begin{remark}
\label{n-line over R}
    Determining the differentials in the real kq-resolution seems to be quite difficult. By complex Betti realization, we can deduce much of the behavior of the complex $\textup{kq}$-resolution from the bo-resolution (see \cref{C-betti kq res}). However, there is information in the real $\textup{kq}$-resolution which is both $\rho$-torsion and $v_1$-periodic that is not detected by these methods. We plan to analyze the real $\textup{kq}$-resolution in future work using a combination of these methods and $C_2$-equivariant homotopy theory.
\end{remark}

\appendix

\section{Charts}
\label{charts}
In this appendix, we record charts depicting the $\text{Ext}^{***}_{\euscr{A}(1)^\vee_\mathbb{R}}(\mathbb{M}_2^\mathbb{R})$-modules described in \cref{section 4}. A black $\blacksquare$ represents $\mathbb{F}_2[\rho,\tau^4]$. A black $\bullet$ represents $\mathbb{F}_2[\tau^4]$. A vertical black line represents multiplication by $h_0$. A horizontal black line represents multiplication by $\rho$. A diagonal black line represents multiplication by $h_1$.

\begin{figure}[H]
    \begin{minipage}{.45\textwidth}
    \centering
    \includegraphics[scale=0.75]{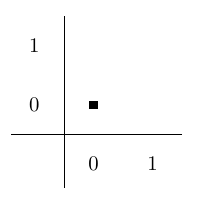}
    \caption{The module $P$.}
    \label{rho tower picture}
    \end{minipage}
    \begin{minipage}{.45\textwidth}
    \centering
        \includegraphics[scale=.75]{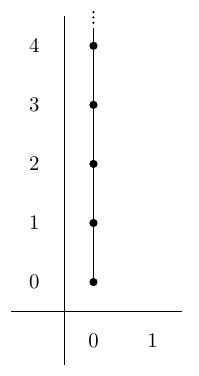}
        \caption{The module $H$.}
            \label{h0 tower picture}        
    \end{minipage}
\end{figure}

\begin{figure}[H]
    \begin{minipage}{.45\textwidth}
        \centering
        \includegraphics[scale=.75]{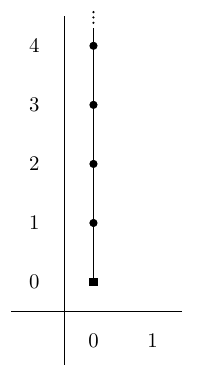}
        \caption{The module $PH$.}
            \label{rho h0 tower picture}        
    \end{minipage}
    \begin{minipage}{.45\textwidth}
            \centering
            \includegraphics[scale=.75]{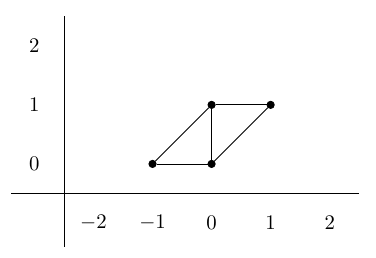}
            \caption{The module $D$.}
            \label{diamond picture}
    \end{minipage}
\end{figure}

\begin{figure}[H]
    \begin{minipage}{.45\textwidth}
            \centering
            \includegraphics[scale=.75]{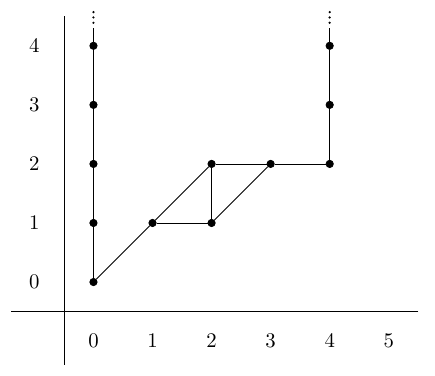}
            \caption{The module $S$.}
            \label{staircase picture}
    \end{minipage}
    \begin{minipage}{.45\textwidth}
            \centering
            \includegraphics[scale=.75]{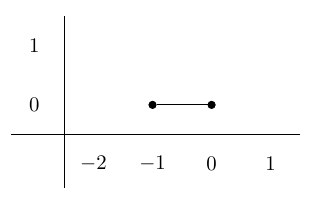}
            \caption{The module $T$.}
            \label{segment picture}
    \end{minipage}
\end{figure}
\begin{figure}[H]
    \begin{minipage}{.45\textwidth}
            \centering
            \includegraphics[scale=.75]{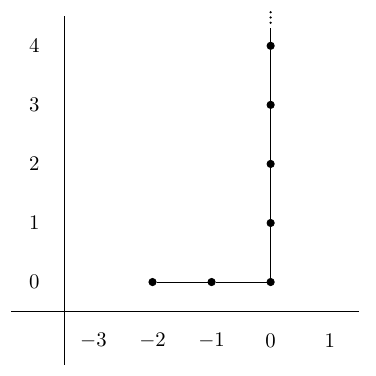}
            \caption{The module $J$.}
            \label{J-tower picture}
    \end{minipage}
    \begin{minipage}{.45\textwidth}
            \centering
            \includegraphics[scale=.75]{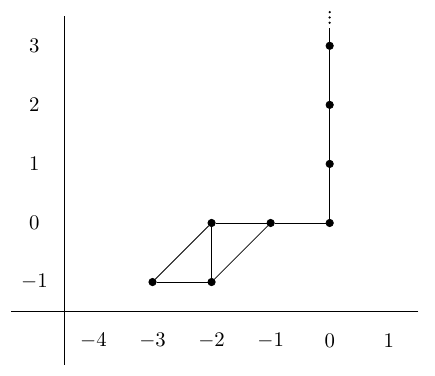}
            \caption{The module $JD$.}
            \label{JD-tower picture}
    \end{minipage}
\end{figure}

\printbibliography

\end{document}